\documentclass[12pt]{amsart}

\usepackage[headings]{fullpage}

\usepackage{amsmath,amssymb}
\usepackage[utf8,utf8x]{inputenc}
\usepackage[pagebackref=true,
pdftitle = {Supersolvable simplicial Arrangements},
pdfkeywords = {},
pdfsubject = {},
pdfauthor = {}]
{hyperref}
\usepackage{hyperref}
\usepackage{longtable}
\usepackage{graphicx}
\usepackage{pict2e}
\usepackage{tikz}
\usetikzlibrary{calc} 
\usetikzlibrary{intersections}

\usepackage{xspace}

\usepackage{subcaption}

\newtheorem{propo}{Proposition}[section]
\newtheorem{corol}[propo]{Corollary}
\newtheorem{theor}[propo]{Theorem}
\newtheorem{lemma}[propo]{Lemma}
\newtheorem{conje}[propo]{Conjecture}

\theoremstyle{definition}
\newtheorem{defin}[propo]{Definition}
\newtheorem{examp}[propo]{Example}

\theoremstyle{remark}
\newtheorem{remar}[propo]{Remark}

\numberwithin{equation}{section}

\newenvironment{psmallmatrix}
  {\left(\begin{smallmatrix}}
  {\end{smallmatrix}\right)}

\newcommand{\NN }{\mathbb{N}}
\newcommand{\CC }{\mathbb{C}}
\newcommand{\RR }{\mathbb{R}}

\newcommand{\ZZ }{\mathbb{Z}}
\newcommand{\KK }{\mathbb{K}}
\newcommand{\LL }{\mathbb{L}}

\newcommand{\id }{\mathrm{id}}

\newcommand{\Rc }{\mathcal{R}}
\newcommand{\Ac }{\mathcal{A}}
\newcommand{\Eg }{\mathcal{E}}
\newcommand{\Bc }{\mathcal{B}}
\newcommand{\Dc }{\mathcal{D}}
\newcommand{\Kc }{\mathcal{K}}
\newcommand{\Vg }{\mathcal{V}}

\newcommand{\Cc }{\mathcal{C}}
\newcommand{\Gc }{\mathcal{G}}

\DeclareMathOperator{\diag}{diag}

\DeclareMathOperator{\GL}{GL}

\DeclareMathOperator{\PGAL}{P\Gamma L}
\DeclareMathOperator{\rank}{rank}

\DeclareMathOperator{\md}{mod}

\newcommand{\sss}{super\-sol\-vable sim\-plicial\xspace}
\newcommand{\isss}{ir\-re\-du\-cible\- super\-sol\-vable sim\-plicial\xspace}

\title[Supersolvable simplicial arrangements]
{Supersolvable simplicial arrangements}

\author{M.~Cuntz}
\address{Michael Cuntz,
Institut für Algebra, Zahlentheorie und Diskrete Mathematik,
Fakultät für Mathematik und Physik,
Leibniz Universität Hannover,
Welfengarten 1,
D-30167 Hannover, Germany}
\email{cuntz@math.uni-hannover.de}

\author{P.~M\"ucksch}
\address{Paul M\"ucksch,
Fakult\"at f\"ur Mathematik, Ruhr-Universit\"at Bochum,
D-44780 Bochum, Germany}
\email{paul.muecksch@rub.de}

\begin{document}


\begin{abstract}
Simplicial arrangements are classical objects in discrete geometry. Their classification remains an open problem but there is a list conjectured to be complete at least for rank three. A further important class in the theory of hyperplane arrangements with particularly nice geometric, algebraic, topological, and combinatorial properties are the supersolvable arrangements. In this paper we give a complete classification of supersolvable simplicial arrangements (in all ranks). For each fixed rank, our classification already includes almost all known simplicial arrangements. Surprisingly, for irreducible simplicial arrangements of rank greater than three, our result shows that supersolvability imposes a strong integrality property; such an arrangement is called crystallographic. Furthermore we introduce Coxeter graphs for simplicial arrangements which serve as our main tool of investigation.
\end{abstract}

\maketitle

\section{Introduction}

A simplicial arrangement is a finite set of hyperplanes, i.e.\ codimension one subspaces, in a finite dimensional real vector space such that the ambient space is cut into simplicial cones by these hyperplanes.
They were introduced by E.~Melchior \cite{MR0004476} in 1941 by the means of triangulations of the projective plane by a finite set of projective lines.

B.\ Gr\"unbaum \cite{MR2485643} gave a list of rank $3$ simplicial arrangements, the slightly extended list \cite{Cun12} is conjectured to be complete.
But not much is known about simplicial arrangements of higher rank.
In a series of papers I.~Heckenberger and the first author investigate a class of objects called finite Weyl groupoids, a generalization of Weyl groups. Their work results in a complete classification of these objects, \cite{MR3341467}.
Since Weyl groupoids are in one to one correspondence with
crystallographic arrangements \cite{MR2820159}, and these constitute a large subclass of the known simplicial arrangements,
this explains a large subset of the arrangements in Gr\"unbaum's list.

The list given by Gr\"unbaum contains two infinite series of irreducible simplicial arrangements of rank three parametrized by positive integers.
They are denoted $\Rc(1) = \{ \Ac(2n,1) \mid n \geq 3 \}$ and $\Rc(2) = \{\Ac(4m+1,1) \mid m \geq 2 \}$.
The irreducible simplicial $3$-arrangements which do not belong to one
of these infinite classes are called \emph{sporadic}.
One observes that each of the
$94$ sporadic arrangements in \cite{Cun12}
consists of no more than $37$ hyperplanes.
So the following is conjectured:
\begin{conje}[cf.\ {\cite[Conj.\ 1.6]{MR3391182}}]\label{conj:37}
Let $\Ac$ be an irreducible simplicial arrangement of rank three.
If $\vert \Ac \vert > 37$ then $\Ac \in \Rc(1)\cup\Rc(2)$.
\end{conje}
D.\ Geis and the first author observed that simpliciality is a purely combinatorial property of the intersection lattice of an arrangement \cite{MR3391182}.
This combinatorial characterization suggests a connection of the class of simplicial arrangements with other classes of arrangements which can be defined combinatorially.

Supersolvable arrangements were first considered by R.~Stan\-ley \cite{MR0309815}. They are now a well studied class of arrangements.
Supersolvable arrangements possess particularly nice algebraic, geometric, topological, and combinatorial properties,
cf.\ \cite[Theorems 2.63, 3.81, 4.58, 5.113]{orlik1992arrangements}.
Looking at the list of all known simplicial arrangements (including the known higher rank cases)
one further observes that almost all of them  belong to the class of supersolvable arrangements.

As the list (at least for rank $3$) is conjectured to be complete and a conceptional approach towards a general classification is still missing, one might ask if there is an approach for a subclass with additional properties, e.g.\ supersolvable simplicial arrangements.
This approach is chosen in the present article resulting in our following main theorem, a complete classification (for rank $3$ up to lattice equivalence) of
(irreducible) supersolvable simplicial arrangements:
\begin{theor}
Let $\Ac$ be an \isss $\ell$-ar\-range\-ment, $(\ell \geq 3)$. Then for $\Ac$ one of the following cases holds:
\begin{enumerate}
\item $\ell=3$ and $\Ac$ is $L$-equivalent to exactly one of the arrangements in $\Rc(1)\cup\Rc(2)$, or
\item $\ell \geq 4$ and $\Ac$ is linearly isomorphic to exactly one of the reflection arrangements $\Ac(A_\ell),\Ac(C_\ell)$ or
to $\Ac_\ell^{\ell-1}  = \Ac(C_\ell) \setminus  \{ \{x_1=0\} \}$.
In particular $\Ac$ is crystallographic.
\end{enumerate}
\end{theor}

As a result of Part (1) of the above theorem we can reformulate Conjecture \ref{conj:37} in the following way:
\begin{conje}
Let $\Ac$ be an irreducible simplicial $3$-arrangement.
If $\vert \Ac \vert > 37$ then $\Ac$ is supersolvable.
\end{conje}

The article is organized as follows.
Firstly we recall the basic notions from the theory of hyperplane arrangements
and some properties of supersolvable and simplicial arrangements which we frequently need later on.
In Subsection \ref{ssec:simplArr} we further comment on the more general notion of combinatorial simpliciality
and its behavior with respect to some standard constructions for arrangements.
In Section \ref{sec:Coxeter graphs} we introduce Coxeter graphs, our main tool for a detailed investigation of simplicial arrangements.
In the last three sections we prove our main theorem giving the aforementioned classification.

\subsection*{Acknowledgment}
We are grateful to the referee for valuable suggestions and comments on an earlier version of our manuscript.

\section{Recollection and Preliminaries}\label{sec:recprl}

We review the required notions and definitions, cf.~\cite{orlik1992arrangements}.
Furthermore in Subsection \ref{ssec:simplArr} we prove some basic properties of simplicial arrangements.

\subsection{Arrangements of hyperplanes}

Let $\KK$ be a field. An $\ell$\emph{-ar\-range\-ment of hyperplanes} is a pair $(\Ac,V)$,
where $\Ac$ is a finite set of hyperplanes (codimension $1$ subspaces) in
the finite dimensional vector space $V \cong \KK^\ell$.
For $(\Ac,V)$ we simply write $\Ac$ if the vector space $V$ is unambiguous.
We denote the empty $\ell$-arrangement by $\Phi_\ell$.

If $\alpha \in V^*$ is a linear form, we write $\alpha^\perp = \ker(\alpha)$ and interpret $\alpha$ as a normal vector
for the hyperplane $H = \alpha^\perp$.
Let $\Ac = \{H_1,\ldots,H_n\}$ be an arrangement in $V=\RR^\ell$.
If we choose a basis $x_1,\ldots,x_\ell$  for $V^*$ and if $\alpha_j = \sum_{i=1}^{\ell} a_{ij} x_i \in V^*$
such that $H_j = \alpha_j^\perp$ then we say that $\Ac$ is \emph{given explicitly by the matrix} $(a_{ij})_{1\le i\le \ell, 1\le j\le n} \in \KK^{\ell \times n}$.


The \emph{intersection lattice} $L(\Ac)$ of $\Ac$ is the set of all subspaces $X$ of $V$ of the form
$X = H_1 \cap \ldots \cap H_r$ with $\{ H_1,\ldots,H_r\} \subseteq \Ac$, partially ordered by reverse inclusion:
$$
X \leq Y \iff Y \subseteq X, \quad \text{ for } X,Y \in L(\Ac).
$$
If $X \in L(\Ac)$, then the \emph{rank} $r(X)$ of $X$ is defined as $r(X) := \ell - \dim{X}$, i.e.\ the codimension of $X$ and the rank of the arrangement
$\Ac$ is defined as $r(\Ac) := r(T(\Ac))$ where $T(\Ac) := \bigcap_{H \in \Ac} H$ is the \emph{center} of $\Ac$.
An $\ell$-arrangement $\Ac$ is called \emph{essential} if $r(\Ac)=\ell$.

For $X \in L(\Ac)$, we define the \emph{localization}
$$
\Ac_X := \{ H \in \Ac \mid X \subseteq H \}
$$
of $\Ac$ at $X$, and the \emph{restriction of} $\Ac$ to $X$,
$(\Ac^X,X)$, where
$$
\Ac^X := \{ X\cap H \mid H \in \Ac \setminus \Ac_X \}.
$$
For $X,Y \in L(\Ac)$ with $X \leq Y$ we define the \emph{interval}
$$
[X,Y] = \{ Z \in L(\Ac) \mid X \leq Z \leq Y \}.
$$
Note that we have $(\Ac_Y)^X = (\Ac^X)_Y$, and $L((\Ac_Y)^X) \cong [X,Y]$, i.e.\ intervals in $L(\Ac)$ 
are again intersection lattices of restricted and localized arrangements.

For $0 \leq q \leq \ell$ we write $L_q(\Ac) := \{ X \in L(\Ac) \mid r(X) = q \}$.
If $X$ is a subspace of $V$ and $X \subseteq H$ for all $H \in \Ac$ then $H/X$ is a hyperplane
in $V/X$ for all $H \in \Ac$ and we can define the quotient
arrangement $(\Ac / X, V/X)$ by $\Ac / X := \{ H/X \mid H \in \Ac \}$.

If $(\Ac,V)$ is not essential, i.e.\ $\dim(T(\Ac)) > 0$, we sometimes identify it with the essential
$r(\Ac)$-arrangement $(\Ac/T(\Ac),V/T(\Ac))$.

Two arrangements $\Ac$ and $\Bc$ in $V$ are \emph{lattice equivalent}
or $L$-equivalent if $L(\Ac) \cong L(\Bc)$ as lattices and in this case we write $\Ac \sim_L \Bc$.
If $\Ac$ and $\Bc$ are arrangements in $V$ such that there is a $\varphi \in \GL(V)$ with $\Bc = \varphi(\Ac) = \{ \varphi(H) \mid H \in \Ac \}$
then we say that $\Ac$ is \emph{(linearly) isomorphic} to $\Bc$.

The \emph{product} $\Ac = (\Ac_1 \times \Ac_2,V_1 \oplus V_2)$ of two arrangements $(\Ac_1,V_1)$, $(\Ac_2,V_2)$
is defined by
\begin{equation*}
\Ac := \Ac_1 \times \Ac_2 = \{ H_1 \oplus V_2 \mid H_1 \in \Ac_1 \} \cup \{ V_1 \oplus H_2 \mid H_2 \in \Ac_2 \},
\end{equation*}
see \cite[Def.~2.13]{orlik1992arrangements}. In particular $\vert \Ac \vert = \vert \Ac_1 \vert + \vert \Ac_2 \vert$.
If an arrangement $\Ac$ can be written as a non-trivial product $\Ac = \Ac_1 \times \Ac_2$, i.e.\ $\Ac_i \neq \Phi_0$,
then $\Ac$ is called \emph{reducible}, otherwise \emph{irreducible}.

\begin{propo}[{\cite[Prop.~2.14]{orlik1992arrangements}}]
Let $\Ac = \Ac_1 \times \Ac_2$ be a product. Define a partial order on $L(\Ac_1) \times L(\Ac_2)$  by
$$
(X_1,X_2) \leq (Y_1,Y_2) \iff X_1 \leq Y_1 \text{ and } X_2 \leq Y_2,
$$
for $(X_1,X_2), (Y_1,Y_2) \in L(\Ac_1) \times L(\Ac_2)$.
Then there is an isomorphism of lattices
\begin{eqnarray*}
\pi: 	&L(\Ac_1) \times L(\Ac_2) &\to L(\Ac_1 \times \Ac_2) \\
	& (X_1,X_2) &\mapsto X_1 \oplus X_2.
\end{eqnarray*}
\end{propo}

\begin{corol}\label{res loc product}
Let $\Ac = \Ac_1 \times \Ac_2$ be a product and $X = X_1\oplus X_2 \in L(\Ac)$. Then we have
$$
\Ac_X = (\Ac_1)_{X_1} \times (\Ac_2)_{X_2} \text{ and }
\Ac^X = (\Ac_1)^{X_1} \times (\Ac_2)^{X_2}.
$$
\end{corol}

For an arrangement $\Ac$ the \emph{M\"obius function} $\mu: L(\Ac) \to \ZZ$ is defined by:
$$
\mu(X) = \left\{ \begin{array}{l l}
   1 & \quad \text{ if } X=V \text{,}\\
        -\sum_{V \leq Y < X} \mu(Y) & \quad \text{ if }X \neq V\text{.}
  \end{array} \right.
$$

We denote by $\chi_\Ac(t)$
the \emph{characteristic polynomial} of $\Ac$ which is defined by:
$$
\chi_\Ac(t)= \sum_{X \in L(\Ac)} \mu(X) t^{\dim(X)}.
$$

\begin{remar}\label{rem:CharPolyRk3}
If $\Ac$ is a $3$-arrangement then the characteristic polynomial is given by
$$
	\chi_\Ac(t) = t^3 + \mu_1t^2 + \mu_2t + \mu_3,
$$
with
\begin{eqnarray*}
	\mu_1 &= &-\vert \Ac \vert, \quad \mu_2 = \sum_{X \in L_2(\Ac)} (\vert \Ac_X \vert -1), \quad
	\mu_3 =
	-1-\mu_1-\mu_2.
\end{eqnarray*}
\end{remar}

\begin{lemma}[{\cite[Lem.~2.50]{orlik1992arrangements}}]\label{charpoly product}
Let $\Ac = \Ac_1 \times \Ac_2$ be a product of two arrangements. Then
$$
\chi_\Ac(t) = \chi_{\Ac_1}(t)\chi_{\Ac_2}(t).
$$
\end{lemma}

We state the following geometric theorem generalizing the well known
Sylvester-Gallai theorem in its dual version for real arrangements.
It was first proved by Motzkin \cite{MR0041447} for $\ell =4$ and later by Hansen \cite{MR0203561} for all $\ell$.

\begin{theor}[Hansen-Motzkin] \label{gen Silvester}
Let $\Ac$ be an essential $\ell$-arrangement over $\RR$, $\ell \geq 3$. Then there is an $X \in L_{\ell-1}(\Ac)$ and an $H \in \Ac$ such that
$X = H \cap Y$ for a $Y \in L_{\ell-2}(\Ac)$, and $\Ac_X = \Ac_Y \cup \{H\}$.
In particular $\Ac_X/X$ is reducible with $\Ac_X/X \cong \Ac_Y/Y \times \{ \{0\} \}$.
\end{theor}

\subsection{Simplicial arrangements}\label{ssec:simplArr}
Many of the notions in this subsection were introduced in the more general setting of  simplicial arrangements
on convex cones and Tits arrangements
in \cite{CMW2017}.

We firstly recall the definition of a simplicial arrangement.

\begin{defin}
Let $\Ac$ be an arrangement in a finite dimensional real vector space $V$.
Then $\Ac$ is called \emph{simplicial} if every connected component of $V \setminus \bigcup_{H \in \Ac} H$
is an open simplicial cone.
We denote by $\Kc(\Ac)$ the set of connected components of $V \setminus \bigcup_{H \in \Ac} H$;
a $K \in \Kc(\Ac)$ is called a \emph{chamber}.
\end{defin}

Note that the only simplicial $1$-arrangement is the arrangement $\Ac =
\{\{0\}\}$, i.e. the non empty one, and every real $2$-arrangement with
more than one hyperplane is simplicial.

There are the following classical examples of simplicial arrangements.

\begin{examp}\label{ex Coxeter group}
Let $W \leq \GL(V)$ be a finite real reflection group acting on the real vector space $V$, i.e.\ a finite Coxeter group.
Suppose that $W$ has full rank, i.e. $\rank(W) = \dim(V)$.
Then the reflection arrangement $(\Ac(W),V)$,
(also called Coxeter arrangement), consisting of all
the reflection hyperplanes of $W$ is a simplicial arrangement.
\end{examp}

\begin{examp}
For $0 \leq k \leq \ell$ let $\Ac_\ell^k$ be the $\ell$-arrangement defined as follows
\begin{align*}
\Ac_\ell^k := \quad& \{ \ker(x_i\pm x_j) \mid 1 \leq i < j \leq \ell \} \\
	& \cup \{ \ker(x_i) \mid 1 \leq i \leq k \}.
\end{align*}
The arrangements $\Ac_\ell^k$ are simplicial, cf.\ \cite{MR3341467}.
In particular $\Ac_\ell^0 = \Ac(D_\ell)$ and $\Ac_\ell^\ell=\Ac(C_\ell)$ are the reflection
arrangements of the finite reflection groups of type $D_\ell$ and $C_\ell$ respectively.
\end{examp}

\begin{defin}
Let $\Ac$ be a simplicial $\ell$-arrangement in the real vector space $V$.
For $\alpha \in V^*$ we write $\alpha^+ = \alpha^{-1}(\RR_{> 0})$ and $\alpha^- = (-\alpha)^+$
for the positive respectively negative open half-space defined by $\alpha$.

For $K \in \Kc(\Ac)$ define the \emph{walls} of $K$ as
$$
W^K := \{ H \in \Ac \mid \dim(H\cap\overline{K}) = \ell-1 \}.
$$

If $R \subseteq V^*$ is a finite set such that $\Ac = \{ \alpha^\perp \mid \alpha \in R\}$
and $\RR\alpha \cap R = \{\pm \alpha \}$ for all $\alpha \in R$
then $R$ is called a \emph{(reduced) root system} for $\Ac$.

If $B^K \subseteq V^*$ is such that $\vert B^K \vert = \vert W^K \vert$, $W^K = \{ \alpha^\perp \mid \alpha \in B^K \}$
and $K = \cap_{\alpha \in B^K} \alpha^+$ then $B^K$ is called a \emph{basis} for $K$.

If $R$ is a root system for $\Ac$ we obtain a basis for $K$ as
$$
B_R^K := \{ \alpha \in R \mid \alpha^\perp \in W^K \text{ and } \alpha^+\cap K =  K \}.
$$
Furthermore for $\gamma \in B^K$ let $K\gamma$ be the unique adjacent chamber in $\Kc(\Ac)$, such that
$\langle \overline{K} \cap \overline{K\gamma} \rangle = \gamma^\perp$ (the linear span of $\overline{K} \cap \overline{K\gamma}$).
If there is a chosen numbering of $B^K=\{\alpha_1,\ldots,\alpha_\ell\}$
then
we simply write $K_i=K\alpha_i$.
\end{defin}

\begin{remar}\label{simpl basis}
The notions $W^K$, $R$ and $B^K$ also make sense for a not necessarily simplicial real arrangement $\Ac$.
Since the normals of the facets of a cone constitute a basis if and only if the cone is simplicial,
we observe that $B^K$ is indeed a basis of $V^*$ for all $K \in \Kc(\Ac)$ if and only if $\Ac$ is simplicial.
\end{remar}

The following notion was first introduced in \cite[Def.~2.3]{MR2820159}.
\begin{defin}
Let $\Ac$ be a simplicial arrangement.
If there exists a root system $R \subseteq V^*$ for $\Ac$ such that for all $K \in \Kc(\Ac)$ we have
$$
R \subseteq \sum_{\alpha \in B_R^K} \ZZ\alpha,
$$
then $\Ac$ is called \emph{crystallographic} and in this case we call $R$ a \emph{crystallographic root system}
for $\Ac$.
\end{defin}

\begin{examp}\label{ex:Weyl}
Let $W$ be a Weyl group, i.e.\ a crystallographic finite real reflection group with
(reduced) root system $\Phi(W)$.
Then the Weyl arrangement $\Ac(W) = \{ \alpha^\perp \mid \alpha \in \Phi(W) \}$
is a crystallographic arrangement with crystallographic root system $R = \Phi(W)$.
\end{examp}

A complete classification of crystallographic arrangements by finite Weyl grou\-po\-ids
was obtained in \cite{MR3341467}, see also \cite{MR2820159}.
It is worth mentioning that the class of crystallographic arrangements is much bigger than the class
of Weyl arrangements with many more ($74$) sporadic cases.
However, it turns out
that irreducible crystallographic arrangements of rank greater or equal to 4 are all restrictions of (irreducible) Weyl arrangements
(see for example \cite[Thm.~3.7]{MR3623727}):
\begin{theor}\label{cryst arr rk geq 4}
Let $\Ac$ be an irreducible simplicial $\ell$-arrangement with $\ell\geq 4$.
Then it is crystallographic if and only if it is a restriction of some (irreducible) Weyl arrangement.
\end{theor}

We recall the following combinatorial characterization of simplicial $3$-arrangements.
\begin{lemma}\cite[Cor.~2.7]{MR3391182}\label{point charact}
Let $\Ac$ be a $3$-arrangement. Then $\Ac$ is simplicial if and only if
$$
\mu_2 := \sum_{X \in L_2(\Ac)} (\vert \Ac_X \vert-1) = 2 \vert L_2(\Ac) \vert - 3.
$$
\end{lemma}

More generally real simplicial $\ell$-arrangements are characterized by the next combinatorial property.

\begin{lemma}\cite[Cor.~2.4]{MR3391182}\label{charpoly charact}
Let $\Ac$ be an $\ell$-arrangement. Then $\Ac$ is simplicial if and only if
$$
\ell\vert\chi_\Ac(-1)\vert - 2 \sum_{H \in \Ac} \vert \chi_{\Ac^H}(-1) \vert = 0.
$$
\end{lemma}

\begin{defin}
Let $\KK$ be any field and $\Ac$ an arrangement in $V=\KK^\ell$.
Define
$$
s(\Ac) := \ell\vert\chi_\Ac(-1)\vert - 2 \sum_{H \in \Ac} \vert \chi_{\Ac^H}(-1) \vert.
$$
If $\Ac$ satisfies $s(\Ac)=0$ then $\Ac$ is called \emph{combinatorially simplicial}, see \cite{MR3391182}.
\end{defin}

Simpliciality, at least geometrically for real arrangements, is compatible with taking localizations and restrictions,
compare with the more general statements in \cite{CMW2017}.
\begin{lemma}\label{loc res simpl}
Let $\Ac$ be a simplicial arrangement over $\RR$ and $X \in L(\Ac)$. Then we have
\begin{enumerate}
\item $(\Ac_X/X,V/X)$ is simplicial,
\item $(\Ac^X,X)$ is simplicial.
\end{enumerate}
\end{lemma}
\begin{proof}
Let $H_1,\ldots,H_{r(X)}$ be the walls of a chamber $K_X$ in $\Ac_X$. They are a subset of the walls of a chamber $K \in \Kc(\Ac)$.
If $\alpha_1,\ldots,\alpha_{r(X)}$ are corresponding normals of these walls pointing to the inside of $K$ and also $K_X$
then they are linearly independent, hence $K_X/X$ is a simplicial cone by Remark \ref{simpl basis} and $\Ac_X/X$ is simplicial.

Since every face of a simplicial cone is a simplicial cone, Statement (2) follows directly.
\end{proof}

\begin{examp}
Let $\Ac = \Ac(W)$ be the Coxeter arrangement of the finite real reflection group $W$ in $V$ and let $X \in L(\Ac)$.
Then $\Ac_X/X$ is a reflection arrangement, namely the Coxeter arrangement of a parabolic subgroup of $W$.
The arrangement $\Ac_X/X$ is simplicial in accordance with Lemma \ref{loc res simpl}(1).
\end{examp}

In the next example we see that the bigger class of combinatorially simplicial arrangements defined over arbitrary fields
is neither closed under taking localizations nor closed under taking restrictions.
\begin{examp}\label{ex comb simpl loc res}
Let $V=\CC^4$, $\zeta = -\frac{1}{2}(1-\sqrt{3}i)$ be a primitive third root of unity and $(\Ac,V)$ the complex
$4$-arrangement containing $18$ hyperplanes and defined by
$$
\begin{psmallmatrix}
0 & 1 & 1 & 1 & 1 & 1 & 1 & 1 & 1 & 1 & 0 & 0 & 0 & 0 & 0 & 0 & 0 & 0  \\
0 & -\zeta & -\zeta^2 & -1 & 0 & 0 & 0 & 0 & 0 & 0 & 1 & 1 & 1 & 1 & 1 & 1 & 0 & 0  \\
0 & 0 & 0 & 0 & -\zeta & -\zeta^2 & -1 & 0 & 0 & 0 & -\zeta & -\zeta^2 & -1 & 0 & 0 & 0 & 1 & 1  \\
1 & 0 & 0 & 0 & 0 & 0 & 0 & -\zeta & -\zeta^2 & -1 & 0 & 0 & 0 & -\zeta & -\zeta^2 & -1 & -\zeta & -\zeta^2
\end{psmallmatrix}.
$$
Note that $\Ac$ is a subarrangement of the reflection arrangement of the complex reflection group $G(3,1,4)$,
see \cite[Ch.~6.4]{orlik1992arrangements} for a definition of these reflection arrangements.
This is to say if
\begin{eqnarray*}
\Bc  &:= &\Ac(G(3,1,4)) \\
	 &= &\{ \ker(x_i-\zeta^k x_j) \mid 1 \leq i < j \leq 4,\, 0 \leq k \leq 2 \} \\
	 & &\cup \{\ker(x_i) \mid 1 \leq i \leq 4\},
\end{eqnarray*}
then we obtain $\Ac$ by removing $4$ hyperplanes,
$$
\Ac = \Bc \setminus \{\ker(x_1),\ker(x_2),\ker(x_3),\ker(x_3-x_4)\}.
$$
A quick calculation shows that $\Ac$ satisfies $s(\Ac)=0$ so it is combinatorially simplicial.
While for the reflection arrangement $\Bc$ all localizations and restrictions are again combinatorially simplicial,
localizing $\Ac$ at the rank 3 intersection $X = H_1\cap H_2 \cap H_3 \in L(\Ac)$,
where the hyperplane $H_i$ corresponds to the $i$-th column of the defining matrix
above, yields the $3$-arrangement $\Cc = \Ac_X/X$. It contains $10$ hyperplanes and is given by
$$
\begin{pmatrix}
0 & 1 & 1 & 1 & 1 & 1 & 1 & 0 & 0 & 0  \\
0 & -\zeta & -\zeta^2 & -1 & 0 & 0 & 0 & 1 & 1 & 1  \\
1 & 0 & 0 & 0 & -\zeta & -\zeta^2 & -1 & -\zeta & -\zeta^2 & -1
\end{pmatrix}.
$$
For $\Cc$ we have $s(\Cc)  = 4$, so it is not combinatorially simplicial.

Now let $H= H_8 = (1,0,0,-\zeta)^\perp \in \Ac$. Then $\Dc := \Ac^H$
contains $10$ hyperplanes and may be defined by
$$
\begin{pmatrix}
1 & \zeta & 1 & 0 & 0 & 0 & 0 & -1 & \zeta & 1 \\
0 & 0 & 0 & 1 & \zeta & 1 & 0 & \zeta & -1 & -1 \\
0 & 1 & -1 & 0 & 1 & -1 & 1 & 1 & 1 & 0
\end{pmatrix}.
$$
For $\Dc$ we have $s(\Dc)  = 4$, thus it is also not combinatorially simplicial.
\end{examp}

The product construction described above is compatible with simpliciality.

\begin{propo}\label{simpl prod_comb}
Let $\Ac_1$, $\Ac_2$ be combinatorially simplicial arrangements in $\KK^{\ell_1}$ and $\KK^{\ell_2}$ respectively.
Then the product $\Ac = \Ac_1 \times \Ac_2$ is combinatorially simplicial.
\end{propo}
\begin{proof}
Let $\Ac_1$ and $\Ac_2$ be combinatorially simplicial.
Then we have
$$
s(\Ac_1) = \ell_1\vert\chi_{\Ac_1}(-1)\vert - 2\sum_{H \in \Ac_1} \vert \chi_{\Ac_1^H}(-1) \vert = 0,
$$
and
$$
s(\Ac_2) = \ell_2\vert\chi_{\Ac_2}(-1)\vert - 2\sum_{H \in \Ac_2} \vert \chi_{\Ac_2^H}(-1) \vert = 0.
$$
By Lemma \ref{charpoly product} we have $\chi_\Ac(t) = \chi_{\Ac_1}(t)\chi_{\Ac_2}(t)$. By Corollary \ref{res loc product} we get
\begin{eqnarray*}
s(\Ac)
&= & \ell\vert\chi_\Ac(-1)\vert - 2\sum_{H \in \Ac} \vert \chi_{\Ac^H}(-1) \vert \\
&= &(\ell_1+\ell_2)\vert\chi_{\Ac_1}(-1)\chi_{\Ac_2}(-1)\vert \\
& &-\,  2\sum_{H \in \Ac_1} \vert \chi_{\Ac_2}(-1)\chi_{\Ac_1^H}(-1) \vert
         - 2\sum_{H \in \Ac_2} \vert \chi_{\Ac_1}(-1)\chi_{\Ac_2^H}(-1) \vert \\
&= &\vert \chi_{\Ac_2}(-1) \vert (\ell_1\vert\chi_{\Ac_1}(-1)\vert - 2\sum_{H \in \Ac_1} \vert \chi_{\Ac_1^H}(-1) \vert) \\
& &+\,  \vert \chi_{\Ac_1}(-1) \vert (\ell_2\vert\chi_{\Ac_2}(-1)\vert - 2\sum_{H \in \Ac_2} \vert \chi_{\Ac_2^H}(-1) \vert) \\
&= &\vert \chi_{\Ac_2}(-1) \vert s(\Ac_1) + \vert \chi_{\Ac_1}(-1) \vert s(\Ac_2) = 0.
\end{eqnarray*}
Hence $\Ac$ is combinatorially simplicial.
\end{proof}

\begin{propo}\label{simpl prod}
Let $(\Ac_1,V_1)$ and $(\Ac_2,V_2)$ be two real arrangements.
Then the product $(\Ac_1\times \Ac_2,V)$ with $V=V_1\oplus V_2$
is simplicial if and only if $\Ac_1$ and $\Ac_2$ are both simplicial.
\end{propo}
\begin{proof}
If $\Ac_1$ and $\Ac_2$ are simplicial, then $\Ac = \Ac_1 \times \Ac_2$ is simplicial by Proposition \ref{simpl prod_comb}.

Conversely, let $\Ac = \Ac_1 \times \Ac_2$ be simplicial.
Then $\Ac_i$ is isomorphic to $\Ac_{X_i} / X_i$ for $i =1,2$ as $r(X_i)$-arrangements in $V/X_i$
where $X_i = {\{0\} \oplus V_{3-i}}$. But these localizations regarded as essential arrangements in quotient spaces are
simplicial by Lemma \ref{loc res simpl}.
\end{proof}

Combinatorial simpliciality of $\Ac_1 \times \Ac_2$ does not imply combinatorial simpliciality of $\Ac_1$ and $\Ac_2$ in general:

\begin{examp}\label{ex comb simpl prod}
Let $\zeta$, $\Ac$ and $\Dc$ be as in Example \ref{ex comb simpl loc res}. Let $\Ac_1 = \Dc$ and $\Ac_2 = \Ac^H$
where $H = H_5 = (1,0,-\zeta,0)^\perp$ as in Example \ref{ex comb simpl loc res}.
Define $\omega := \frac{1}{3}(1-\zeta)$. Then $\Ac_2$ is given by
$$
\begin{pmatrix}
1 & 0 & \omega & \omega & \omega & \omega & \omega & \omega & 0 & 0 \\
0 & 0 & 1 & \zeta & \zeta^2 & 0 & 0 & 0 & \zeta & 1 \\
0 & 1 & 0 & 0 & 0 & -\zeta & -\zeta^2 & -1 & 1 & 1
\end{pmatrix}.
$$
Recall that for the non combinatorially simplicial arrangement $\Ac_1$ we have $s(\Ac_1) = 4$.
Furthermore $\chi_{\Ac_1}(t) = (t-1)(t-4)(t-5) = \chi_{\Ac_2}(t)$, and $s(\Ac_2)  = -4$.
So $\Ac_2$ is also not combinatorially simplicial.
But, similar to the proof of Proposition \ref{simpl prod}, for $\Ac_1\times\Ac_2$ we have
\begin{eqnarray*}
 s(\Ac_1\times\Ac_2) &=
 	&\vert \chi_{\Ac_2}(-1) \vert s(\Ac_1) +\,  \vert \chi_{\Ac_1}(-1) \vert s(\Ac_2) \\
&= &\vert \chi_{\Ac_1}(-1) \vert s(\Ac_1) +\,  \vert \chi_{\Ac_1}(-1) \vert s(\Ac_2) \\
&= &\vert \chi_{\Ac_1}(-1) \vert (s(\Ac_1) + s(\Ac_2)) \\
&= &\vert \chi_{\Ac_1}(-1) \vert (4-4) \\
&=  &0.
\end{eqnarray*}
So the product $\Ac_1 \times \Ac_2$ is combinatorially simplicial.
\end{examp}

The following is true for all real simplicial arrangements, cf.\ \cite[Lem.~3.29]{CMW2017}.
\begin{lemma}\label{Lem_RootsAdjK}
Let $(\Ac,V)$ be a real simplicial arrangement, $K \in \Kc(\Ac)$ a chamber with basis $B^K = \{ \alpha_1,\ldots,\alpha_\ell \}$.
Then for $\beta \in V^*$ with $\beta^\perp \in W^{K_i}$ and $K_i \subseteq \beta^+$ we have
\begin{enumerate}
\item $\beta \in \RR_{<0}\alpha_i$ if $\beta^\perp = \alpha_i^\perp$, or
\item $\beta \in \sum_{k=1}^\ell \RR_{\geq 0}\alpha_k$ if $\beta^\perp \in W^{K_i} \setminus \{\alpha_i^\perp\}$.
\end{enumerate}
\end{lemma}

\begin{lemma}\label{cij unique}
Let $\Ac$ be a real simplicial $\ell$-arrangement and $K \in \Kc(\Ac)$ with basis
$B^K = \{\alpha_1,\ldots,\alpha_\ell \}$.
Then for $1\leq i, j \leq \ell$
there are $c^K_{ij} \in \RR$
such that
$$
\{\beta_j^i = \alpha_j - c^K_{ij}\alpha_i \mid j=1,\ldots,\ell\}
$$
is a basis for $K_i$.
If $i\neq j$ then $c^K_{ij}\leq 0$ and $c^K_{ij}$ is
uniquely determined by $B^K$.
If $i=j$ then $c^K_{ij} > 1$.
\end{lemma}
\begin{proof}
Let $\beta \in V^*$ such that $\beta^\perp \in W^{K_i}$ and $K_i \subseteq \beta^+$.
Suppose that $\beta^\perp \neq \alpha_i^\perp$. Then $\beta^\perp \in \Ac_{\alpha_i^\perp \cap \alpha_j^\perp}$ for some $1\leq j \leq \ell$, $j \neq i$
and by Lemma \ref{Lem_RootsAdjK} there are $a_i,a_j \in \RR_{\geq 0}$ such that $\beta = a_i\alpha_i + a_j\alpha_j$.
Since $B^K$ is a basis for $V^*$, and $\beta \notin \langle \alpha_i \rangle$ we further have $a_j >0$.
Setting $c^K_{ij} := -\frac{a_i}{a_j}$ and $\beta_j^i = \alpha_j - c^K_{ij}\alpha_i$ we have ${\beta_j^i}^\perp = \beta^\perp$
and ${\beta_j^i}^+ = \beta^+$.
Since $B^K$ is a basis for $V^*$, for $i\neq j$ the $c^K_{ij}$ are uniquely determined.
Hence again by Lemma \ref{Lem_RootsAdjK} for some $c^K_{ii} > 1$ we have that
$\{\beta_j^i = \alpha_j - c^K_{ij}\alpha_i \mid j=1,\ldots,\ell\}$ is a basis for $K_i$.
\end{proof}

\begin{defin}
Let $\Ac$ be a real simplicial $\ell$-arrangement, $K \in \Kc(\Ac)$, and
$B^K = \{\alpha_1,\ldots,\alpha_\ell \}$ a basis for $K$.
For $i\neq j$ let $c^K_{ij}$ be the uniquely determined coefficients
from Lemma \ref{cij unique}.
For $1\leq i\leq\ell$ we set $c^K_{ii} = 2$
and define the linear map
$\sigma^K_{\alpha_i} := \sigma^K_i$ by
$$
\sigma^K_i(\alpha_j) := \alpha_j - c^K_{ij}\alpha_i
$$
for $1 \leq j \leq \ell$.
With respect to the basis $B^K$ this map is represented by the matrix
$$
S_i^K :=
\begin{pmatrix}
1 & 0  &   &  & &  &  \\
0 & 1 &   & & &  &  \\
& & \ddots &  & & & \\
-c^K_{i1} & \cdots & -c^K_{i(i-1)} & -1 & -c^K_{i(i+1)} & \cdots & -c^K_{i\ell} \\
 &  &  &  & \ddots  &  &  \\
 &  &  &  &  & 1 & 0 \\
 &  &  &  &  & 0 & 1 \\
\end{pmatrix}.
$$
\end{defin}

\begin{remar}\label{Rem_simpl_refl}
We observe that $\sigma^K_i$ is a reflection at the hyperplane $\alpha_i^\perp$.
In particular $\det(S_i^K) = -1$.
Furthermore, $c^K_{ij} \neq 0$ if and only if $c^K_{ji} \neq 0$ 
since $c^K_{ij}=0$ if and only if $|\Ac_{\alpha_i^\perp \cap \alpha_j^\perp}|=2$
if and only $c^K_{ji}=0$ by Lemma \ref{cij unique} (cf.\ \cite{MR2820159}).
\end{remar}

\begin{defin}
Let $\Ac$ be a real arrangement with chambers $\Kc(\Ac)$.
A sequence $(K^0,K^1,\ldots,$ $K^{n-1},K^n)$ of distinct chambers in $\Kc(\Ac)$ is called a \emph{gallery} if
for all $1 \leq i \leq n$ we have
$\langle\overline{K^{i-1}} \cap \overline{K^i} \rangle  = H \in \Ac$,
i.e.\ if $K^i$ and $K^{i-1}$ are adjacent
with common wall $H$.
We denote by $\Gc(\Ac)$ the set of all galleries of $\Ac$.

We say that $G \in \Gc(\Ac)$ has length $n$ if it is a sequence of $n+1$ chambers.
For $G = (K^0,\ldots,K^n) \in \Gc(\Ac)$ we denote by $b(G) = K^0$ the first chamber and by $e(G) = K^n$ the last chamber in $G$.
\end{defin}

\begin{defin}
Let $\Ac$ be a real simplicial $\ell$-arrangement.
We fix a chamber $K^0 \in \Kc(\Ac)$.
Let $\Gc(K^0,\Ac) = \{ G \in \Gc(\Ac) \mid b(G) = K^0 \}$ be the set of galleries starting with $K^0$.

Let $B^{K^0} = \{ \alpha_1^0,\ldots,\alpha_\ell^0\}$ be a basis for $K^0$.
For $(K^0,\ldots,K^n) = G \in \Gc(K^0,\Ac)$ we denote by $B^{K^n}_G = B_G$ the basis for $K^n$ induced by $G$ and $B^{K^0}$, i.e.\ such that
$$
B^{K^{i+1}} = \sigma^{K^i}_{\mu_i}(B^{K^i}) = \{\alpha_j^{i+1} = \sigma^{K^i}_{\mu_i}(\alpha^i_j) = \alpha_j^i - c^{K^i}_{\mu_ij}\alpha_{\mu_i}^i \mid 1 \leq j \leq \ell\},
$$
where $K^{i+1} = K^i_{\mu_i}$, $\mu_i \in \{1,\ldots,\ell\}$, and $0 \leq i \leq n-1$.
\end{defin}

\begin{defin}
Let $\Ac$ be a real simplicial $\ell$-arrangement, $K \in \Kc(\Ac)$.
We call a basis $B^K = \{\alpha_1,\ldots,\alpha_\ell\}$ \emph{locally crystallographic}  if
the $c^K_{ij}$ are all integers.

If $B^K$ is a locally crystallographic basis
then we call the matrix $C^K = (c^K_ {ij})_{i,j=1,\ldots,\ell}$ the Cartan matrix of $B^K$.
\end{defin}

\begin{examp}\label{ex:Akl_cryst}
Let $\Ac = \Ac_\ell^k$. Then $\Ac$ is crystallographic with crystallographic root system $R$.
In particular for $K \in \Kc(\Ac)$ the basis $B^K_R$ is a locally crystallographic basis for $K$ and the
corresponding Cartan matrix is (up to simultaneous permutation of columns and rows) one of the matrices displayed in Table \ref{tab:CartanCoxeter}, see \cite[Prop.~3.8]{MR3341467}.
\begin{table}
\begin{tabular}{c|c}
$C^K$ & $\Gamma(K)$ \\
\hline
\\
$
A_\ell:
\begin{pmatrix}
2 & -1 & 0 & \cdots & 0 \\
-1 & 2 & -1 & \cdots & 0 \\
0 & -1 & 2 & \cdots & 0 \\
\vdots& \vdots & \ddots & \ddots & \vdots \\
0 & 0 & \cdots & 2 & -1 \\
0 & 0 & \cdots & -1 & 2
\end{pmatrix}
$
&
\parbox[c]{4.5cm}{
\begin{tikzpicture}
\filldraw [black]
	(0,0) circle [radius=2pt]
	(1,0) circle [radius=2pt]
	(2,0) circle [radius=2pt]
	(3,0) circle [radius=2pt]
	(4,0) circle [radius=2pt];
\draw
	(0,0) -- (1,0) -- (2,0) -- (2.2,0)
	(2.8,0) -- (3,0) -- (4,0);
\node at (2.5,0)  {$\ldots$};

\node[below] at (0,0)  {$1$};
\node[below] at (1,0)  {$2$};
\node[below] at (2,0)  {$3$};
\node[below] at (3,0)  {$\ell-1$};
\node[below] at (4,0)  {$\ell$};
\end{tikzpicture}\hspace{1cm}
}
\\
\\
$
C_\ell:
\begin{pmatrix}
2 & -1 & 0 & 0 & \cdots & 0 \\
-2 & 2 & -1 & 0 & \cdots & 0 \\
0 & -1 & 2 & -1 & \cdots & 0 \\
\vdots& \vdots & \ddots & \ddots & \ddots & \vdots \\
0 & 0 & \cdots & -1 & 2 & -1 \\
0 & 0 & \cdots & 0 & -1 & 2
\end{pmatrix}
$
&
\parbox[c]{4.5cm}{
\begin{tikzpicture}
\filldraw [black]
	(0,0) circle [radius=2pt]
	(1,0) circle [radius=2pt]
	(2,0) circle [radius=2pt]
	(3,0) circle [radius=2pt]
	(4,0) circle [radius=2pt];
\draw
	(0,0) -- (1,0) -- (2,0) -- (2.2,0)
	(2.8,0) -- (3,0) -- (4,0);
\node at (2.5,0)  {$\ldots$};
\node[above] at (0.5,0)  {$4$};

\node[below] at (0,0)  {$1$};
\node[below] at (1,0)  {$2$};
\node[below] at (2,0)  {$3$};
\node[below] at (3,0)  {$\ell-1$};
\node[below] at (4,0)  {$\ell$};
\end{tikzpicture}
}
\\
\\
$
D_\ell:
\begin{pmatrix}
2 & 0 & -1 & 0 & 0 & \cdots & 0 \\
0 & 2 & -1 & 0 & 0 &\cdots & 0 \\
-1 & -1 & 2 & -1 & 0 &\cdots & 0 \\
0 & 0 & -1 & 2 & -1 & \cdots & 0 \\
\vdots& \vdots & \vdots &\ddots & \ddots & \ddots & \vdots \\
0 & 0 & 0 & \cdots & -1 & 2 & -1 \\
0 & 0 & 0 & \cdots & 0 & -1 & 2
\end{pmatrix}
$
&
\parbox[c]{4.5cm}{
\begin{tikzpicture}
\filldraw [black]
	(0,0.5) circle [radius=2pt]
	(0,-0.5) circle [radius=2pt]
	(1,0) circle [radius=2pt]
	(2,0) circle [radius=2pt]
	(3,0) circle [radius=2pt]
	(4,0) circle [radius=2pt];
\draw
	(0,-0.5) -- (1,0) -- (0,0.5)
	(1,0) -- (2,0) -- (2.2,0)
	(2.8,0) -- (3,0) -- (4,0);
\node at (2.5,0)  {$\ldots$};

\node[below] at (0,-0.5)  {$1$};
\node[above] at (0,0.5)  {$2$};
\node[below] at (1,0)  {$3$};
\node[below] at (2,0)  {$4$};

\node[below] at (3,0)  {$\ell-1$};
\node[below] at (4,0)  {$\ell$};
\end{tikzpicture}
}
\\
\\
$
D'_\ell:
\begin{pmatrix}
2 & -1 & -1 & 0 & 0 & \cdots & 0 \\
-1 & 2 & -1 & 0 & 0 &\cdots & 0 \\
-1 & -1 & 2 & -1 & 0 &\cdots & 0 \\
0 & 0 & -1 & 2 & -1 & \cdots & 0 \\
\vdots& \vdots & \vdots &\ddots & \ddots & \ddots & \vdots \\
0 & 0 & 0 & \cdots & -1 & 2 & -1 \\
0 & 0 & 0 & \cdots & 0 & -1 & 2
\end{pmatrix}
$
&
\parbox[c]{4.5cm}{
\begin{tikzpicture}
\filldraw [black]
	(0,0.5) circle [radius=2pt]
	(0,-0.5) circle [radius=2pt]
	(1,0) circle [radius=2pt]
	(2,0) circle [radius=2pt]
	(3,0) circle [radius=2pt]
	(4,0) circle [radius=2pt];
\draw
	(0,0.5) -- (0,-0.5) -- (1,0) -- (0,0.5)
	(1,0) -- (2,0) -- (2.2,0)
	(2.8,0) -- (3,0) -- (4,0);
\node at (2.5,0)  {$\ldots$};

\node[below] at (0,-0.5)  {$1$};
\node[above] at (0,0.5)  {$2$};
\node[below] at (1,0)  {$3$};
\node[below] at (2,0)  {$4$};

\node[below] at (3,0)  {$\ell-1$};
\node[below] at (4,0)  {$\ell$};
\end{tikzpicture}
}
\\
\end{tabular}
\vspace{0.2cm}
\caption{Cartan matrices and Coxeter graphs.}
\label{tab:CartanCoxeter}
\end{table}
\end{examp}

\begin{defin}
Let $B^K$ be a locally crystallographic basis with Cartan matrix $C^K$.
If (up to simultaneous permutation of columns and rows) $C^K$ is one of the matrices shown in the left column of Table \ref{tab:CartanCoxeter}
then we say $C^K$ is of type $A,C,D$, or $D'$ respectively.
\end{defin}

If $B^K$ is a locally crystallographic basis with Cartan matrix of type $A,C,D$, or $D'$ then the corresponding Coxeter graph
$\Gamma(K)$ (see Section \ref{sec:Coxeter graphs}) is displayed in the right column of Table \ref{tab:CartanCoxeter}.

\begin{lemma}\label{cij=cij'}
Let $\Ac$ be a simplicial $\ell$-arrangement, $K \in \Kc(\Ac)$ with basis $B^K = \{\alpha_1,\ldots,\alpha_\ell\}$,
and $K_i$ an adjacent chamber.
Then for all $1 \leq j \leq \ell$ we have $c^{K_i}_{ij} = c^K_{ij}$ and in particular
$\sigma^K_i \circ \sigma^{K_i}_i = \sigma^{K_i}_i \circ \sigma^K_i = \id$.
\end{lemma}
\begin{proof}
We have ${\sigma^K_i(\alpha_j)}^\perp = {\beta^i_j}^\perp = (\alpha_j - c^K_{ij}\alpha_i)^\perp \in W^{K_i}$
but similarly $\sigma^{K_i}_i(\beta^i_j)^\perp = \alpha_j^\perp = (\beta^i_j - c^{K_i}_{ij}\beta^i_i)^\perp
= (\alpha_j - c^K_{ij}\alpha_i - c^{K_i}_{ij}(-\alpha_i))^\perp \in W^K$.
Thus $c^K_{ij} = c^{K_i}_{ij}$.
\end{proof}

Similarly to the crystallographic case we have the following.
\begin{lemma}[{cf.\ \cite[Lem.~4.5]{MR2498801}}]\label{cij=0}
Let $\Ac$ be a simplicial $\ell$-arrangement, $K$, $B^K$, and $K_i$ as before.
Let $i \neq j$ and suppose $c^K_{ij} = 0$. Then $c^{K_i}_{jk} = c^K_{jk}$ for all $k \in \{1,\ldots,\ell\}$.
\end{lemma}
\begin{proof}
The proof is the same as in \cite{MR2498801}.

If $k =i$ then by Lemma \ref{cij=cij'} $c^K_{jk} = c^K_{kj} = 0 = c^{K_i}_{kj} = c^{K_i}_{jk}$. And if $k = j$ then
all the coefficients are equal to $2$. So let $k \in \{1,\ldots,\ell\}\setminus \{i,j\}$.
Since $c^K_{ij} = 0$ we have $\vert \Ac_{\alpha_i^\perp \cap \alpha_j^\perp} \vert = 2$.
So application of $\sigma^{K_i}_j \circ \sigma^K_i$ and $\sigma^{K_j}_i \circ \sigma^K_j$ on $\alpha_k$ should yield
a normal of the same wall of the chamber $K\alpha_i\sigma^{K}_i(\alpha_j) = K\alpha_j\sigma^{K}_j(\alpha_i)$.
Now
\begin{eqnarray*}
\sigma^{K_i}_j(\sigma^K_i(\alpha_k)) &=
	&\sigma^{K}_i(\alpha_k) - c^{K_i}_{jk}\sigma^K_i(\alpha_j) \\
&= 	&\alpha_k - c^K_{ik}\alpha_i - c^{K_i}_{jk}(\alpha_j - c^K_{ij}\alpha_i) \\
&= 	&\alpha_k - c^K_{ik}\alpha_i - c^{K_i}_{jk}\alpha_j,
\end{eqnarray*}
and similarly
\begin{eqnarray*}
\sigma^{K_j}_i(\sigma^K_j(\alpha_k)) &= &\alpha_k - c^K_{jk}\alpha_j - c^{K_j}_{ik}\alpha_i.
\end{eqnarray*}
Since $i,j,k$ are pairwise different and $\{\alpha_1\ldots,\alpha_\ell\}$ are linearly independent,
comparing the coefficients of $\alpha_j$ in both terms gives $c^{K_i}_{jk} = c^K_{jk}$.
\end{proof}

\subsection{Supersolvable arrangements}

An element $X \in L(\Ac)$ is called \emph{modular} if $X + Y \in L(\Ac)$ for all $Y \in L(\Ac)$.
An arrangement $\Ac$ with $r(\Ac) = \ell$ is called \emph{supersolvable}
if the intersection lattice $L(\Ac)$ is supersolvable, i.e.\ there is
a maximal chain of modular elements
$$
V = X_0 < X_1 < \ldots <X_\ell = T(\Ac),
$$
$X_i \in L(\Ac)$ modular.
For example an essential $3$-arrangement $\Ac$ is supersolvable if there exists an $X \in L_2(\Ac)$ which is
connected to all other $Y \in L_2(\Ac)$ by a suitable hyperplane $H \in \Ac$, (i.e.\ $X + Y \in \Ac$).

\begin{lemma}[{\cite[Lem.~2.27]{orlik1992arrangements}}]\label{Lem_ModCompl}
Let $\Ac$ be an essential $\ell$-arrangement, $X \in L_{\ell-1}(\Ac)$ a modular element and $H \in \Ac \setminus \Ac_X$.
Then $|\Ac^H|=|\Ac_X|$.
\end{lemma}

Supersolvability is preserved by taking localizations and restrictions,
see \cite[Lem.~2.6]{MR3252667}, and \cite[Prop.~3.2]{MR0309815}:

\begin{lemma}\label{SS localization}
Let $\Ac$ be an arrangement, $X \in L(\Ac)$ and $Y \in L(\Ac)$ a modular element with $X \subseteq Y$.
Then $Y$ is modular in $L(\Ac_X)$. Moreover if $\Ac$ is supersolvable so is $\Ac_X$ for all $X \in L(\Ac)$.
\end{lemma}

\begin{lemma}\label{SS restriction}
Let $\Ac$ be an arrangement, $X \in L(\Ac)$ and $Y \in L(\Ac)$ a modular element.
Then $Y\cap X$ is modular in $L(\Ac^X)$.
In particular if $\Ac$ is supersolvable so is $\Ac^X$ for all $X \in L(\Ac)$.
\end{lemma}

Combining the previous two lemmas with Lemma \ref{loc res simpl} we obtain the following.

\begin{lemma}\label{SSS loc res}
Let $\Ac$ be a real \sss arrangement and $X \in L(\Ac)$. Then we have
\begin{enumerate}
\item $(\Ac_X/X,V/X)$ is supersolvable and simplicial,
\item $(\Ac^X,X)$ is supersolvable and simplicial.
\end{enumerate}
\end{lemma}

Furthermore, as (geometric) simpliciality, supersolvability is compatible with products.

\begin{lemma}[{\cite[Prop.~2.5]{MR3251720}}] \label{ss prod}
Let $\Ac = \Ac_1 \times \Ac_2$ be a product. Then $\Ac$ is supersolvable if and only if
$\Ac_1$ and $\Ac_2$ are both supersolvable.
\end{lemma}

So together with Proposition \ref{simpl prod} we get the following.
\begin{propo}\label{sss prod}
Let $\Ac_1$ and $\Ac_2$ be real arrangements and $\Ac = \Ac_1 \times \Ac_2$ their product.
Then $\Ac$ is supersolvable and simplicial
if and only if $\Ac_1$ and $\Ac_2$ are both supersolvable and simplicial.
\end{propo}

Because of the previous proposition, to classify supersolvable and simplicial arrangements, it suffices to
classify the irreducible ones.

The following property of the characteristic polynomial of a supersolvable arrangement is due to Stanley \cite{MR0309815},
cf.~\cite[Thm.~2.63]{orlik1992arrangements}.
\begin{theor}\label{ss charpoly fact}
Let $\Ac$ be a supersolvable $\ell$-arrangement with
$$
V = X_0 < X_1 < \ldots <X_\ell = T(\Ac)
$$
a maximal chain of modular elements.
Let $b_i := \vert \Ac_{X_i} \setminus \Ac_{X_{i-1}} \vert$ for $1 \leq i \leq \ell$. Then
$$
\chi_\Ac(t) = \prod_{i=1}^\ell (t-b_i).
$$
\end{theor}

A helpful result is due to Amend, Hoge, and R\"ohrle.
They checked which restrictions of irreducible reflection arrangements are supersolvable, \cite[Thm.~1.3]{MR3252667}.
Here we only need the following weaker version for real reflection arrangements of rank greater or equal to $4$.
\begin{theor}\label{ss res refarr}
Let $\Ac = \Ac(W)$ be an irreducible real reflection arrangement of rank $\ell \geq 4$ associated to
the finite reflection group $W$ and $X \in L(\Ac)$ with $m:= \dim(X) \geq 4$.
Then $\Ac^X$ is supersolvable if and only if one of the following holds:
\begin{enumerate}
\item $W= A_\ell$ and then $\Ac^X = \Ac(A_m)$,
\item $\Ac^X = \Ac^k_m$ with $k \in \{m,m-1\}$.
\end{enumerate}
\end{theor}

Together with Theorem \ref{cryst arr rk geq 4} this gives us the following classification of irreducible supersolvable crystallographic
arrangements of rank $\geq 4$.
\begin{theor}\label{classification ss cryst rk geq 4}
Let $\Ac$ be an irreducible supersolvable crystallographic $\ell$-arrangement with $\ell \geq 4$.
Then $\Ac$ is isomorphic to exactly one of the reflection arrangements $\Ac(A_\ell)$, $\Ac(C_\ell)$ or
$\Ac_\ell^{\ell-1} = \Ac(C_\ell) \setminus  \{ \{x_1=0\} \}$.
\end{theor}

\section{Coxeter graphs for simplicial arrangements}\label{sec:Coxeter graphs}

From now on until the end of this article we always assume arrangements to be real.

We introduce Coxeter graphs of chambers of simplicial arrangements and
use the results from Subsection \ref{ssec:simplArr} to derive their properties.

\begin{defin}
Let $K \in \Kc(\Ac)$ be a chamber of the simplicial $\ell$-arrangement $\Ac$
and $B^K$ some basis for $K$.
We define a labeled non directed simple graph $\Gamma(K) = (\Vg,\Eg)$ with vertices $\Vg = B^K$ and
edges $\Eg = \{ \{\alpha,\beta\} \mid \vert\Ac_{\alpha^\perp \cap \beta^\perp}\vert \geq 3 \}$. An
edge $e=\{\alpha,\beta\} \in \Eg$ is labeled with
$m^K(e) = m^K(\alpha,\beta) = \vert\Ac_{\alpha^\perp \cap \beta^\perp}\vert$.
Since the label $m(\alpha,\beta) = 3$ appears more often we omit it in drawing the graph.
We call $\Gamma(K)$ the \emph{Coxeter graph} of $K$.
If we have chosen a numbering $B^K = \{\alpha_1,\ldots,\alpha_\ell\}$ then
$\{\alpha_i,\alpha_j\} \in \Eg$ is simply denoted by $\{i,j\}$ and $\Vg = \{1,\ldots,\ell\}$,
see Figure \ref{fig:ExCoxeter}.

If $K_i$ is an adjacent chamber for some $1 \leq i \leq \ell$ and $\Gamma(K_i) = (\Vg_i,\Eg_i)$ its Coxeter graph
with $\Vg_i = B^{K_i} = \sigma_i^K(B^K) = \{ \sigma^K_i(\alpha_1),\ldots,\sigma^K_i(\alpha_\ell) \}$,
we similarly identify $\{ \sigma_i^K(\alpha_k),\sigma_i^K(\alpha_j) \} \in \Eg_i$ with $\{k,j\}$
and $\Vg_i$ with $\{1, \ldots, \ell\}$. 
\end{defin}

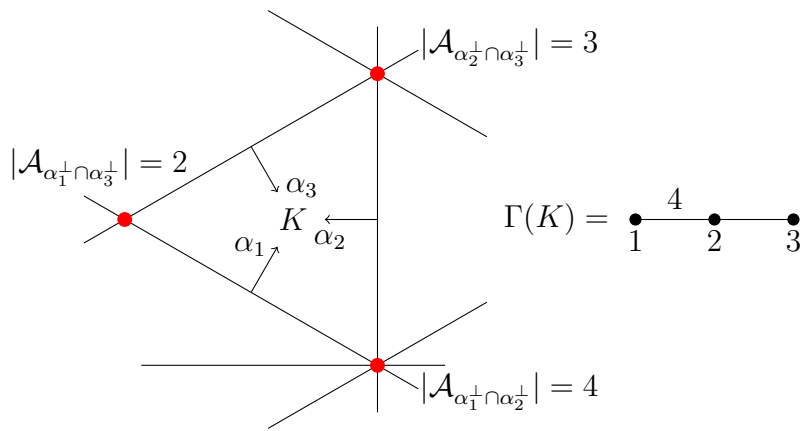
\begin{figure}

\begin{tikzpicture}[scale=0.7]
\draw (-3.974901573277509,0.44738963192723413) -- (2.3749015732775094,-3.218670924037438);  
\draw (1.6000000000000001,3.666060555964672) -- (1.6000000000000001,-3.666060555964672);  
\draw (2.3749015732775094,3.218670924037438) -- (-3.974901573277509,-0.44738963192723413);  
\draw (3.6784609690826526,-1.571281292110204) -- (-0.47846096908265423,-3.9712812921102043);  
\draw (2.8844410203711917,-2.7712812921102037) -- (-2.8844410203711917,-2.7712812921102037);  
\draw (-0.47846096908265423,3.9712812921102043) -- (3.6784609690826526,1.571281292110204);  

\node (x12) at (1.6000000000000001,-2.7712812921102041) {};
\node (x13) at (-3.2000000000000002,0.0) {};
\node (x23) at (1.6000000000000001,2.7712812921102041) {};

\node at ($(x12) + (2.5,-0.5)$) {$\vert\Ac_{\alpha_1^\perp\cap\alpha_2^\perp}\vert = 4$};
\node at ($(x23) + (2.5,0.5)$) {$\vert\Ac_{\alpha_2^\perp\cap\alpha_3^\perp}\vert = 3$};
\node at ($(x13) + (-0.5,1)$) {$\vert\Ac_{\alpha_1^\perp\cap\alpha_3^\perp}\vert = 2$};

\fill[red] (x12) circle[radius=4pt];  
\fill[red] (x13) circle[radius=4pt];  
\fill[red] (x23) circle[radius=4pt];  

\node at (0,0) {$K$};

\node (a) at ($0.5*(x12)+0.5*(x13)$) {};
\node (b) at ($0.5*(x12)+0.5*(x23)$) {};
\node (c) at ($0.5*(x13)+0.5*(x23)$) {};


\node (r1) at (0.5,0.866025) {};
\node (r2) at (-1,0) {};
\node (r3) at (0.5,-0.866025) {};

\draw[->] ($(a)$) -- ($(a) + 1*(r1)$);
\node[left] at ($(a) + 1*(r1)$) {$\alpha_1$};
\draw[->] ($(b)$) -- ($(b) + 1*(r2)$);
\node[below] at ($(b) + 0.9*(r2)$) {$\alpha_2$};
\draw[->] ($(c)$) -- ($(c) + 1*(r3)$);
\node[right] at ($(c) + 0.9*(r3)$) {$\alpha_3$};


\node at (5,0) {$\Gamma(K) = $};
\node (1) at (6.5,0) {};
\node (2) at (8,0) {};
\node (3) at (9.5,0) {};

\filldraw ($(1)$) circle[radius=3pt];
\filldraw ($(2)$) circle[radius=3pt];
\filldraw ($(3)$) circle[radius=3pt];

\draw ($(1)$) -- ($(2)$) -- ($(3)$);

\node[below] at ($(1)$) {$1$};
\node[below] at ($(2)$) {$2$};
\node[below] at ($(3)$) {$3$};

\node[above] at ($0.5*(1) + 0.5*(2)$) {$4$};

\end{tikzpicture}
\caption{The Coxeter graph $\Gamma(K)$ of a chamber $K$.}\label{fig:ExCoxeter}
\end{figure}

\begin{examp}
Let $\Ac(W)$ be the Coxeter arrangement of the Coxeter group $W$.
Then $\Ac$ is a simplicial arrangement (c.f.\ Example \ref{ex Coxeter group}) and for all $K \in \Kc(\Ac)$ the Coxeter graph $\Gamma(K)$
is indeed the Coxeter graph of $W$, see for example \cite[Ch.\ 2]{MR1066460}.
\end{examp}

\begin{lemma}\label{cij m=3}
Let $\Ac$ be a simplicial $\ell$-arrangement, $K \in \Kc(\Ac)$ with basis $B^K = \{\alpha_1,\ldots,\alpha_\ell\}$,
$K_i$ an adjacent chamber.
Then for $j\neq i$ we have $c^K_{ij}=c^K_{ji}=0$ if and only if $m^K(i,j)=2$.
Furthermore if $m^K(i,j) = 3$ then for $i \neq j$ we have $c^K_{ji} = 1/c^K_{ij}$.
In particular if $c^K_{ij}=-1$ then $c^K_{ji}=c^K_{ij}$.
\end{lemma}
\begin{proof}
For $j \neq i$ we have $c^K_{ij}=c^K_{ji}=0$ if and only if $\alpha_j^\perp \in B^{K_i}$ if and only if $m^K(i,j)=2$.
If $m^K(i,j) = 3$ we have
$$
\Ac_{\alpha_i^\perp\cap \alpha_j^\perp} = \{\alpha_i^\perp,\alpha_j^\perp,(\alpha_j - c^K_{ij}\alpha_i)^\perp\}.
$$
By the definition of $c^K_{ij}$ and $c^K_{ji}$ we thus have $(\alpha_j - c^K_{ij}\alpha_i)^\perp=(\alpha_i - c^K_{ji}\alpha_j)^\perp$.
Since $\alpha_i$ and $\alpha_j$ are linearly independent this can only occur if $c^K_{ji} = 1/c^K_{ij}$.
\end{proof}

\begin{lemma}\label{graph adj chamber}
Let $\Ac$ be a simplicial $\ell$-arrangement, $K \in \Kc(\Ac)$ a chamber, $B^K = \{\alpha_1,\ldots,\alpha_\ell\}$,
$\Gamma(K) = (\Vg,\Eg)$, and $K_i$ an adjacent chamber
with $B^{K_i} = \{\sigma^K_i(\alpha_1),\ldots,\sigma^K_i(\alpha_\ell)\}$ and $\Gamma(K_i)=(\Vg_i,\Eg_i)$.
Then if $\{i,j\} \notin \Eg$ ($i\neq j$) but $\{j,k\} \in \Eg$ then $\{j,k\} \in \Eg_i$ (disregarding the labels).
\end{lemma}
\begin{proof}
Since $\{j,k\} \in \Eg$ and $m^K(j,k) \geq 3$ by Lemma \ref{cij m=3} we have $c^K_{jk}\neq 0$. Hence $c^{K_i}_{jk}\neq 0$ by Lemma \ref{cij=0} and so again by Lemma \ref{cij m=3}
$m^{K_i}(j,k) \geq 3$ and $\{j,k\} \in \Eg_i$.
\end{proof}

The next lemma is a direct generalization of \cite[Prop.~4.6]{MR2498801} from crystallographic arrangements
to general simplicial arrangements. It may be proved completely analogously but here we give a more geometric proof.
\begin{lemma}\label{simpl irred graph}
Let $\Ac$ be a simplicial $\ell$-arrangement with chambers $\Kc(\Ac)$. Then the following are equivalent.
\begin{enumerate}
\item $\Ac$ is an irreducible arrangement.
\item $\Gamma(K)$ is connected for all $K \in \Kc(\Ac)$.
\item $\Gamma(K)$ is connected for some $K \in \Kc(\Ac)$.
\end{enumerate}
\end{lemma}
\begin{proof}
We may assume that $\ell$ is at least $2$ since otherwise the statement of the theorem is trivial.

The implication $(2){\Rightarrow}(3)$ is trivial.

$(1){\Rightarrow}(2)$. %
Suppose there is a $K \in \Kc(\Ac)$ such that $\Gamma(K) = (\Vg,\Eg)$ is not connected.
Then there is a partition $\Vg= B^K = \Delta_1 \dot{\cup} \Delta_2$ such that $\vert\Ac_{\alpha^\perp \cap \beta^\perp}\vert=2$
for $\alpha \in \Delta_1$, and $\beta \in \Delta_2$.
Without loss of generality let $\alpha \in \Delta_1$. Then
\begin{equation*}
B^{K\alpha} = \{-\alpha\} \dot{\cup} \{\alpha' + c_{\alpha'}\alpha \mid \alpha' \in \Delta_1 \setminus \{\alpha\} \} \dot{\cup} \Delta_2
\end{equation*}
is a basis for ${K\alpha}$ for certain $c_{\alpha'} \geq 0$, c.f.\ Lemma \ref{cij unique}.
Assume that there are $\alpha' + c\alpha \in B^{K\alpha}$ and $\beta \in \Delta_2 \subseteq B^{K\alpha}$ with
$\vert\Ac_{(\alpha' + c\alpha)^\perp \cap \beta^\perp}\vert \geq 3$.
Then there is a $b > 0$ such that $\alpha' + c\alpha + b\beta \in B^{K\alpha\beta}$.
Note that $K\alpha\beta(-\alpha) = K\beta$ since $\vert\Ac_{\alpha^\perp \cap \beta^\perp}\vert = 2$.
Then there is a $d \geq 0$ such that $\alpha' + c\alpha + b\beta + d(-\alpha) = \alpha' + (c-d)\alpha + b\beta \in B^{K\beta}$.
But $B^{K\beta} = \Delta_1 \dot{\cup} \{-\beta\} \dot{\cup} \{ \beta' + c_{\beta'}\beta \mid \beta' \in \Delta_2 \setminus \{\beta\}\}$
which gives a contradiction.
So for all $\alpha' + c_{\alpha'}\alpha \in B^{K\alpha}$ and $\beta \in \Delta_1$ we have
$\vert\Ac_{(\alpha' + c\alpha)^\perp \cap \beta^\perp}\vert = 2$.
We conclude that for all $\gamma \in B^K$, for the corresponding adjacent chamber $K\gamma$ there is a partition
$B^{K\gamma} = \tilde{\Delta}_1 \dot{\cup} \tilde{\Delta}_2$ with
$\tilde{\Delta}_i \subset \sum_{\lambda \in \Delta_i} \RR \lambda$ and
$\vert\Ac_{\tilde{\alpha}^\perp \cap \tilde{\beta}^\perp}\vert = 2$ for all
$\tilde{\alpha} \in \tilde{\Delta}_1$, $\tilde{\beta} \in \tilde{\Delta}_2$.
Hence for all $H \in \Ac$  by induction using a gallery from $K$ to some chamber $K'$ with $H \in W^{K'}$
we either have $H = (\sum_{\alpha \in \Delta_1} c_\alpha \alpha)^\perp$ with $c_\alpha \in \RR$,
or $H = (\sum_{\beta \in \Delta_2} c_\beta \beta)^\perp$ with $c_\beta \in \RR$ which means that $\Ac$ is reducible.

$(3){\Rightarrow}(1)$. Suppose that $\Ac$ is reducible. Then there exists a basis $\{x_1,\ldots,x_r\}\dot{\cup}\{y_1,\ldots,y_s\}$
of $V^*$ with $r,s \geq 1$ such that for $H \in \Ac$ and $H = \gamma^\perp$ for some $\gamma \in V^*$ we either have
$\gamma \in \sum_{i=1}^r \RR x_i$ or $\gamma \in \sum_{j=1}^s \RR y_j$.
Let $K \in \Kc(\Ac)$ be chamber of $\Ac$. Then $B^K = \Delta_1 \dot{\cup} \Delta_2$ with $\Delta_1 = B^K \cap \sum_i \RR x_i$ and
$\Delta_2 = B^K \cap \sum_j \RR y_j$.
Since $\Ac$ is simplicial, $B^K$ is a basis of $V^*$ and we have $\Delta_i \neq \emptyset$ for $i=1,2$.
Furthermore $\Ac_{\alpha^\perp \cap \beta^\perp} = \{ \alpha^\perp, \beta^\perp \}$ for $\alpha \in \Delta_1$, $\beta \in \Delta_2$
and hence $\Gamma(K)$ is not connected.
\end{proof}

\begin{lemma}\label{label not connected}
Let $\Ac$ be a simplicial $\ell$-ar\-range\-ment,
$K \in \Kc(\Ac)$ with $B^K=\{\alpha_1,\ldots,\alpha_\ell\}$
and $\Gamma(K) = (\Vg,\Eg)$ with vertices $\Vg= \{1,\ldots,\ell\}$.
Suppose that $\{i,j\} \in \Eg$ with label $m^K(i,j)$ and there is a $k \in \Vg\setminus \{i,j\}$ such that $\{k,i\} \notin \Eg$ and $\{k,j\} \notin \Eg$.
Then $\{i,j\}$ is an edge in $\Gamma(K_k)$ with the same label $m^{K_k}(i,j) = m^K(i,j)$.
\end{lemma}
\begin{proof}
That $\{i,j\}$ is an edge in $\Gamma(K_k)$ is simply Lemma \ref{graph adj chamber}.
The second statement holds because $\sigma^K_k(\alpha_i) = \alpha_i$ and $\sigma^K_k(\alpha_j) = \alpha_j$
and thus
$$
  m^{K_k}(i,j) = \vert \Ac_{\sigma^K_k(\alpha_i)^\perp\cap \sigma^K_k(\alpha_j)^\perp} \vert
  = \vert \Ac_{\alpha_i^\perp \cap \alpha_j^\perp} \vert = m^K(i,j).
$$
\end{proof}

\begin{lemma}\label{loc subgraph}
Let $\Ac$ be a simplicial $\ell$-ar\-range\-ment, $X \in L_q(\Ac)$ for $1\leq q \leq \ell$,
and $K_X \in \Kc(\Ac_X/X)$ be a chamber of the localization $\Ac_X/X$.
Let $K \in \Kc(\Ac)$ with $B^K=\{\alpha_1,\ldots,\alpha_\ell\}$ such that $X = \bigcap_{j=1}^q \alpha_{i_j}^\perp$,
$K_X = \bigcap_{j=1}^q \alpha_{i_j}^+/X$,
and $\Gamma(K)$ with corresponding vertices $\Vg= \{1,\ldots,\ell\}$.
Then $\Gamma(K_X)$ is the induced subgraph on the $q$ vertices $\{i_1,\ldots,i_q\} \subseteq \Vg$ of
$\Gamma(K)$ including the labels.
\end{lemma}
\begin{proof}
For $q=1$ the statement is trivially true.
For $q \geq 2$ this is easily seen as the intersection lattice $L(\Ac_X)$ is an interval in the intersection lattice $L(\Ac)$, i.e.\ $L(\Ac_X) = L(\Ac)_X  = [V,X] = \{ Z \in L(\Ac) \mid  Z \leq X \}$.
\end{proof}

With the correspondence from the previous lemma and Lemma \ref{simpl irred graph}
we obtain the following corollary for irreducible simplicial arrangements.

\begin{corol}\label{ex irred loc}
Let $\Ac$ be an irreducible simplicial $\ell$-arrangement and $K \in \Kc(\Ac)$. Then there is an
$X \in L_{\ell - 1}(W^K) \subseteq L(\Ac)$ such that $(\Ac_X/X,V/X)$ is an irreducible simplicial $(\ell-1)$-arrangement.
\end{corol}

To describe the connection between restrictions of simplicial arrangements and Coxeter graphs we need a
bit more notation.

\begin{defin}
Let $\Ac$ be a simplicial arrangement, $K \in \Kc(\Ac)$, $\alpha \in B^K$ and $H = \alpha^\perp \in W^K$. Then
we denote the induced chamber in the restriction $\Ac^H$ by
$$
	K^H = (\bigcap_{\beta \in B^K \setminus \{\alpha\}} \beta^+) \cap H,
$$
and a basis for $K^H$ is given by
$$
B^{K^H} = \{ \beta^H \mid \beta^H := \beta\vert_{H^*} \text{ and } \beta \in B^K \setminus \{\alpha\}\}.
$$

Let $\Gamma(K) = (\Vg,\Eg)$ be the Coxeter graph of $K$ and suppose that there is an edge
$\{\alpha,\beta\} \in \Eg$  connecting the vertices $\alpha$ and $\beta$.
Define $\Gamma^{\alpha\beta} := (\Vg^{\alpha\beta},\Eg^{\alpha\beta})$ to be the (unlabeled) graph with vertices
$$
\Vg^{\alpha\beta} := \Vg \setminus \{\alpha,\beta\} \cup \{\alpha\beta\},
$$
and edges
\begin{eqnarray*}
\Eg^{\alpha\beta} :=
&\{ \{\gamma,\delta\} \in \Eg \mid \{\gamma,\delta\} \cap \{\alpha,\beta\} = \emptyset \} \cup\\
	 &\{ \{\alpha\beta,\gamma\} \mid  \{\alpha,\gamma\} \in \Eg \text{ or } \{\beta,\gamma\} \in \Eg\},
\end{eqnarray*}
i.e.\ the \emph{contraction} of $\Gamma(K)$ along the edge $\{\alpha,\beta\}$.
\end{defin}

It is convenient to use the following notation:
If $\Gamma(K) = (\Vg,\Eg)$ with $\Vg=\{1,\ldots,\ell\}$ corresponding to $B^K=\{\alpha_1, \ldots, \alpha_\ell\}$,
$I \subseteq \Vg$ with $I = \{i_1, \ldots, i_r\}$ and
$X = \cap_{i \in I} \alpha_i^\perp$ then for the localization $\Ac_X$ at the intersection adjacent to the chamber $K$
we simply write $\Ac^K_{i_1i_2\cdots i_r}$,
e.g.\ for $\Ac_{\alpha_1^\perp\cap\alpha_2^\perp\cap\alpha_4^\perp}$ we write $\Ac^K_{124}$.

\begin{lemma}\label{res subgraph}
Let $\Ac$ be a simplicial $\ell$-arrangement and $K \in \Kc(\Ac)$ with Coxeter graph $\Gamma(K) = (\Vg,\Eg)$.
Suppose $\{\alpha,\beta\} \in \Eg$ is an edge. Let $H \in \Ac_{\alpha^\perp\cap\beta^\perp}$
be the wall of $K\alpha$ with $H \neq \alpha^\perp$, i.e.\ $H = \sigma^K_\alpha(\beta)^\perp$, and let
$\Gamma^H =(\Vg^H,\Eg^H):= \Gamma((K\alpha)^H)$ be the Coxeter graph of the chamber $(K\alpha)^H \in \Kc(\Ac^H)$.
Then we have the following:
\begin{enumerate}
\item The contracted graph $\Gamma^{\alpha\beta}$ is isomorphic to a subgraph of $\Gamma^H$ in the following way:
Let $\rho:\Vg^{\alpha\beta} \to \Vg^H$ be the bijective map defined by
$$
\rho(\gamma) :=
\begin{cases}
  \sigma^K_\alpha(\gamma)^H & \text{ if } \gamma \neq \alpha\beta \\
  \sigma^K_\alpha(\alpha)^H = (-\alpha)^H & \text{ if } \gamma=\alpha\beta
\end{cases}.
$$
If $\{\gamma,\delta\} \in \Eg^{\alpha\beta}$ then $\{\rho(\gamma),\rho(\delta)\} \in \Eg^H$, i.e.\ $\Eg^{\alpha\beta} \subseteq \Eg^H$ disregarding the labels.

\item If $\{\alpha, \gamma\} \in \Eg$ ($\gamma \neq \beta$) is labeled with $m(\alpha, \gamma)$
then for the corresponding label in $\Gamma^H$
we have $m^H(\rho(\alpha\beta),\rho(\gamma)) \geq m(\alpha, \gamma)$ (see Figure \ref{fig:contr res}(a)).

\item If $\{\alpha,\beta\},\{\alpha,\gamma\}$, and $\{\beta,\gamma\}$ are edges in $\Eg$, then
for the label of the edge $\{\rho(\alpha\beta),\rho(\gamma)\}$ in $\Gamma^H$ we have
$m^H(\rho(\alpha\beta),\rho(\gamma)) \geq m(\alpha,\gamma)+m(\beta,\gamma) - 2$ (see Figure \ref{fig:contr res}(b)).
\end{enumerate}

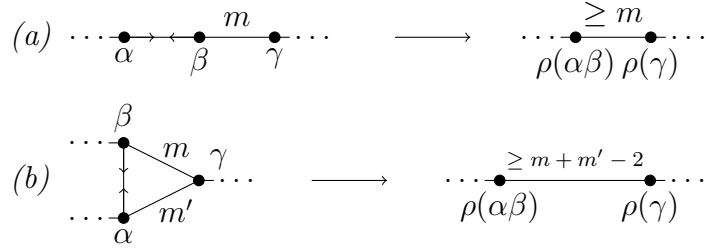
\begin{figure}
\begin{tikzpicture}
\node (a) at (0,0) {};
\node at ($(a)-(1.25,0)$) {(a)};

\node (b) at (6,0) {};

\filldraw [black]
	($(a) + (0,0)$) circle [radius=2pt]
	($(a) + (1,0)$) circle [radius=2pt]
	($(a) + (2,0)$) circle [radius=2pt];

\draw ($(a) + (-0.2,0)$) --  ($(a) + (0,0)$) -- ($(a) + (1,0)$) -- ($(a) + (2,0)$) -- ($(a) + (2.2,0)$);
\draw [->] ($(a) + (0,0)$) -- ($(a) + (0.4,0)$);
\draw [->] ($(a) + (1,0)$) -- ($(a) + (0.6,0)$);
\node at ($(a) + (-0.5,0)$){$\ldots$};
\node at ($(a) + (2.5,0)$){$\ldots$};
\node [above] at ($(a) + (1.5,0)$){$m$};

\node [below] at ($(a) + (0,0)$) {$\alpha$};
\node [below] at ($(a) + (1,0)$) {$\beta$};
\node [below] at ($(a) + (2,0)$) {$\gamma$};

\draw [->] ($0.4*(a)+0.6*(b)$) -- ($0.4*(a)+0.6*(b) + (1,0)$);

\filldraw [black]
	($(b) + (0,0)$) circle [radius=2pt]
	($(b) + (1,0)$) circle [radius=2pt];

\draw ($(b) + (-0.2,0)$) --  ($(b) + (0,0)$) -- ($(b) + (1,0)$) -- ($(b) + (1.2,0)$);
\node at ($(b) + (-0.5,0)$){$\ldots$};
\node at ($(b) + (1.5,0)$){$\ldots$};
\node [above] at ($(b) + (0.5,0)$){$\geq m$};

\node [below] at ($(b) + (0,0)$) {$\rho(\alpha\beta)$};
\node [below] at ($(b) + (1,0)$) {$\rho(\gamma)$};

\end{tikzpicture}

\begin{tikzpicture}
\node (a) at (0,0) {};
\node at ($(a)-(1.25,0)$) {(b)};

\filldraw [black]
	($(a) + (0,0.5)$) circle [radius=2pt]
	($(a) + (1,0)$) circle [radius=2pt]
	($(a) + (0,-0.5)$) circle [radius=2pt];
\draw
	 ($(a) + (0,0.5)$) -- ($(a) + (0,-0.5)$) -- ($(a) + (1,0)$) -- ($(a) + (0,0.5)$);
\draw [->]  ($(a) + (0,0.5)$) -- ($(a) + (0,0.1)$);
\draw [->]  ($(a) + (0,-0.5)$) -- ($(a) + (0,-0.1)$);
\draw ($(a) + (-0.2,0.5)$) -- ($(a) + (0,0.5)$);
\node at ($(a) + (-0.5,0.5)$) {$\ldots$};
\draw ($(a) + (-0.2,-0.5)$) -- ($(a) + (0,-0.5)$);
\node at ($(a) + (-0.5,-0.5)$) {$\ldots$};
\draw ($(a) + (1,0)$) -- ($(a) + (1.2,0)$);
\node at ($(a) + (1.5,0)$) {$\ldots$};
\node at ($(a) + (0.7,0.4)$) {$m$};
\node at ($(a) + (0.7,-0.4)$) {$m'$};
\node [above] at ($(a) + (0,0.5)$) {$\beta$};
\node [below] at ($(a) + (0,-0.5)$) {$\alpha$};
\node [above right] at ($(a) + (1,0)$) {$\gamma$};

\node (b) at (5,0) {};
\draw [->] ($0.5*(a)+0.5*(b)$) -- ($0.5*(a)+0.5*(b) + (1,0)$);

\filldraw [black]
	($(b) + (0,0)$) circle [radius=2pt]
	($(b) + (2,0)$) circle [radius=2pt];
\draw
	 ($(b) + (0,0)$) -- ($(b) + (2,0)$);
\draw ($(b) + (-0.2,0)$) -- ($(b) + (0,0)$);
\node at ($(b) + (-0.5,0)$) {$\ldots$};
\draw ($(b) + (2,0)$) -- ($(b) + (2.2,0)$);
\node at ($(b) + (2.5,0)$){$\ldots$};

\node [above] at ($(b) + (1,0)$) {\tiny{$\geq m+m' - 2$}};

\node [below] at ($(b) + (0,0)$) {$\rho(\alpha\beta)$};
\node [below] at ($(b) + (2,0)$) {$\rho(\gamma)$};
\end{tikzpicture}
\caption{Labels and contraction of Coxeter graphs.}
\label{fig:contr res}
\end{figure}
\end{lemma}
\begin{proof}
It suffices to prove the statements for $3$-arrangements (the statements are trivial for $2$-arrangements).
The general case then follows by taking localizations, the fact that $(\Ac^H)_X = (\Ac_X)^H$,
and Lemma \ref{loc subgraph}.
Let $B^K = \{ \alpha_1,\alpha_2,\alpha_3\}$ and denote the corresponding vertices of
$\Gamma(K)$ by $\{1,2,3\}$.

If $\Gamma(K)$ is not connected, i.e.\ $\Ac$ is reducible by Lemma \ref{simpl irred graph}, then either there is no edge in $\Gamma(K)$
and there is nothing to show, or $\Gamma(K)$ is the graph of Figure \ref{fig:red Graph rk3}.
In this case, all statements hold, since for all $H \in \Ac^K_{12}$ we
then have $\vert \Ac^H \vert =2$. So  $\Ac^H$ is reducible and the Coxeter graph of every chamber of $\Ac^H$
is the graph with $2$ vertices which are not connected and which is exactly isomorphic to the contracted graph $\Gamma^{\alpha_1\alpha_2}$.

\begin{figure}
\begin{tikzpicture}
\filldraw [black]
	(-1,0) circle [radius=2pt]
	(0,0) circle [radius=2pt]
	(1,0) circle [radius=2pt];
\draw
	(-1,0) -- (0,0);
\node[below] at (-1,0) {$1$};
\node[below] at (0,0) {$2$};
\node[below] at (1,0) {$3$};
\end{tikzpicture}
\caption{Coxeter graph of a reducible $3$-arrangement.}%
\label{fig:red Graph rk3}
\end{figure}
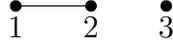

So assume $\Gamma(K)$ is connected. Without loss of generality let
$H = \sigma_1^K(\alpha_2)^\perp \in \Ac^K_{12}$.
Since $(\sigma^K_1(\alpha_1)^H)^\perp = (-\alpha_1^H)^\perp = \alpha_1^\perp\cap\alpha_2^\perp$
in $\Ac^H$, and 
$(\sigma^K_1(\alpha_3)^H)^\perp = \sigma^K_1(\alpha_2)^\perp \cap \sigma^K_1(\alpha_3)^\perp$
we have $(\sigma^K_1(\alpha_1)^H)^\perp \cap (\sigma^K_1(\alpha_3)^H)^\perp = \{0\}$.
We have to show that $\vert \Ac^H \vert \geq (\vert \Ac^K_{13} \vert -1) + (\vert \Ac^K_{23} \vert -1)$ to obtain all three statements.
Let $\Bc = \Ac^K_{13} \cup \Ac^K_{23}$. Then $\vert \Ac^H \vert \geq \vert \Bc^H \vert$ and
$\vert \Bc \vert = \vert \Ac^K_{13} \vert + \vert \Ac^K_{23} \vert -1$
(since $\Ac^K_{13} \cap \Ac^K_{23} = \{\alpha_3^\perp\}$).
We now deduce that $\vert \Bc^H \vert = \vert \Bc \vert -1$:

We have $W^K \subset \Bc$ and $\vert (W^K)^H \vert =2$.
Now let $H_1,H_2 \in \Bc \setminus W^K$ with $H_1 \neq H_2$.
We first observe that $H\cap H_1 \neq H \cap \tilde{H}$ for any $\tilde{H} \in W^K$.
But we also have $H_1 \cap H \neq H_2 \cap H$.
Hence all $H' \in \Bc \setminus W^K$ give different intersections with $H$.
Thus we obtain
\begin{eqnarray*}
\vert \Ac^H \vert &\geq
&\vert \Bc^H \vert =\vert (W^K)^H \vert + \vert (\Bc \setminus W^K)^H \vert \\
&= &2 + \vert (\Bc \setminus W^K) \vert = 2 + \vert \Bc \vert - 3 \\
&= &\vert \Ac^K_{13} \vert + \vert \Ac^K_{23} \vert -2.
\end{eqnarray*}
From this inequality by translating back to the corresponding Coxeter graphs (which are graphs with only two vertices) all statements directly follow.
\end{proof}

\begin{lemma}\label{res irred}
Let $\Ac$ be an irreducible simplicial $\ell$-arrangement and $X \in L_q(\Ac)$. Then the restriction $\Ac^X$ is an irreducible
simplicial $(\ell-q)$-arrangement.
\end{lemma}
\begin{proof}
It suffices to show the statement for $X = H \in \Ac$.

By Lemma \ref{loc res simpl} the restriction $\Ac^H$ is again simplicial.

Since $\Ac$ is irreducible, there is an $X \in L_2(\Ac)$ with $X \subseteq H$ and $\vert\Ac_X\vert \geq 3$.
So there is a chamber $K \in \Kc(\Ac)$ with $\Gamma(K) = (\Vg,\Eg)$, $\{\alpha,\beta\} \in \Eg$ such that
$X = \alpha^\perp\cap\beta^\perp$,
and $H$ the wall of $K\alpha$ not equal to $\alpha^\perp$.
Since $\Ac$ is irreducible, the Coxeter graph $\Gamma(K)$ is connected by Lemma \ref{simpl irred graph},
and by Lemma \ref{res subgraph} the Coxeter graph $\Gamma((K\alpha)^H)$
of the chamber $(K\alpha)^H$ of $\Ac^H$ contains a subgraph on $\ell-1$ vertices which is connected
(as it is isomorphic to a contraction of the connected graph $\Gamma(K)$).
So $\Gamma((K\alpha)^H)$ is also connected and hence again by Lemma \ref{simpl irred graph}
the restriction $\Ac^H$ is irreducible.
\end{proof}


\section{The rank $3$ case}\label{sec:rk3}
We firstly collect some useful results about supersolvable simplicial $3$-arrangements.

\begin{lemma}\label{SSS rk3 two modular}
Let $\Ac$ be a \sss $3$-arrangement with two modular elements $X,Y \in L_2(\Ac)$ such that
$\vert \Ac_X \vert \not= \vert \Ac_Y \vert$. Then $\Ac$ is reducible.
\end{lemma}

\begin{proof}
By Theorem \ref{ss charpoly fact} two different roots of $\chi_\Ac(t)$ are given by $\vert \Ac_X \vert -1$ and $\vert \Ac_Y \vert -1$.
So we have
$$
\chi_\Ac(t) = (t-1)(t-(\vert \Ac_X \vert -1))(t-(\vert \Ac_Y \vert -1)),
$$
and by Remark \ref{rem:CharPolyRk3} we further have
$$
\vert \Ac \vert = -\mu_1 = \vert \Ac_X \vert + \vert \Ac_Y \vert -1 \leq \vert \Ac_X \cup \Ac_Y \vert.
$$
Then there is a hyperplane $H \in \Ac$ with $\Ac^H = \{X,Y\}$, i.e.\ $H = X+Y$.
Hence $\Ac^H$ is reducible.
By Lemma \ref{res irred} the arrangement $\Ac$ is reducible.
\end{proof}

\begin{lemma}[{\cite[Lemma~2.1]{MR3243842}}]\label{Tohaneanu}
Let $\Ac$ be a super\-sol\-vable $3$-ar\-range\-ment. Then all elements
$X \in L_2(\Ac)$ with $\vert \Ac_X \vert$ maximal are modular.
\end{lemma}

Combining Lemma \ref{SSS rk3 two modular} and Lemma \ref{Tohaneanu} we get the following lemma.
\begin{lemma}\label{X mod max}
Let $\Ac$ be an \isss $3$-ar\-ran\-ge\-ment. Then $X \in L_2(\Ac)$
is modular if and only if $\vert \Ac_X \vert$ is maximal among all localizations at intersections of rank two.
\end{lemma}

\begin{corol}\label{Cor_Mod_ge3}
Let $\Ac$ be an \isss $3$-ar\-ran\-ge\-ment and $X \in L_2(\Ac)$ a modular element.
Then $|\Ac_X| \geq 3$.
\end{corol}
\begin{proof}
Suppose $|\Ac_X| =2$. Then $|\Ac_Z|=2$ for all $Z \in L_2(\Ac)$ by Lemma \ref{X mod max}.
Hence $\Gamma(K)$ is disconnected for all $K \in \Kc(\Ac)$ and $\Ac$ is reducible by Lemma \ref{simpl irred graph}. A contradiction.
\end{proof}

\begin{defin}\label{def:A2n}
Let $n \in \NN$ and $\zeta := \exp(\frac{2\pi i}{2n})$ be a primitive $2n$-th root of unity.
We write
$$
c_n(m) := \cos \frac{2\pi m}{2n} = \frac{1}{2}(\zeta^m + \zeta^{-m}),
$$
and
$$
s_n(m) := \sin \frac{2\pi m}{2n} = \frac{1}{2i}(\zeta^m - \zeta^{-m}).
$$
For $n\geq 3$ the arrangements $\Ac(2n,1)$ of the infinite series $\Rc(1)$ from \cite{MR2485643} may be defined by
$$
\begin{pmatrix}
-s_n(0) & -s_n(1) & \ldots & -s_n(n-1) & c_n(1) & c_n(3) & \ldots & c_n(2n-1) \\
c_n(0) & c_n(1) & \ldots & c_n(n-1) & s_n(1) & s_n(3) & \ldots & s_n(2n-1) \\
0 & 0 & \ldots & 0 & 1 & 1& \ldots & 1
\end{pmatrix}
$$

For $n\geq 2$ the arrangements $\Ac(4n+1,1)$ of the series $\Rc(2)$ are constructed as
$$
\Ac(4n+1,1) = \Ac(4n,1) \cup \{ (0,0,1)^\perp \}.
$$
Some examples are displayed as projective pictures of the arrangements in Figure \ref{ex A(n,1)}.
\end{defin}

\begin{figure}

\begin{tikzpicture}[scale=0.75]
\draw (0.,4.) -- (0.,-4.);
\draw (4.,0.) -- (-4.,0.);
\draw (0,4.2352941176470589) arc [start angle=90, end angle=0, radius=4.2352941176470589,] ;
\draw (3.344805190907731,3.344805190907731) node {$\infty$};
\draw (-2.8284271247461898,2.8284271247461898) -- (2.8284271247461898,-2.8284271247461898);
\draw (2.8284271247461898,2.8284271247461898) -- (-2.8284271247461898,-2.8284271247461898);
\draw (-2.,3.4641016151377544) -- (-2.,-3.4641016151377544);
\draw (2.,3.4641016151377544) -- (2.,-3.4641016151377544);
\draw (3.4641016151377544,-2.) -- (-3.4641016151377544,-2.);
\draw (3.4641016151377544,2.) -- (-3.4641016151377544,2.);

\end{tikzpicture}
\begin{tikzpicture}[scale=0.75]
\draw (4.,0.) -- (-4.,0.);
\draw (3.4641016151377544,2.0000000000000004) -- (-3.4641016151377544,-2.0000000000000004);
\draw (2.0000000000000004,3.4641016151377544) -- (-2.0000000000000004,-3.4641016151377544);
\draw (0.,4.) -- (0.,-4.);
\draw (-2.0000000000000004,3.4641016151377544) -- (2.0000000000000004,-3.4641016151377544);
\draw (-3.4641016151377544,2.0000000000000004) -- (3.4641016151377544,-2.0000000000000004);
\draw (-2.8025170768881473,2.8541019662496843) -- (1.0704662693192701,-3.8541019662496843);
\draw (3.872983346207417,-1.) -- (-3.872983346207417,-1.);
\draw (2.8025170768881473,2.8541019662496843) -- (-1.0704662693192701,-3.8541019662496843);
\draw (-1.0704662693192701,3.8541019662496843) -- (2.8025170768881473,-2.8541019662496843);
\draw (3.872983346207417,1.) -- (-3.872983346207417,1.);
\draw (1.0704662693192701,3.8541019662496843) -- (-2.8025170768881473,-2.8541019662496843);
\end{tikzpicture}
\begin{tikzpicture}[scale=0.75]
\draw (0,4.2352941176470589) arc [start angle=90, end angle=0, radius=4.2352941176470589] ;
\draw (3.344805190907731,3.344805190907731) node {$\infty$};  
\draw (4.,0.) -- (-4.,0.);  
\draw (3.6955181300451474,1.5307337294603591) -- (-3.6955181300451474,-1.5307337294603591);  
\draw (2.8284271247461907,2.8284271247461903) -- (-2.8284271247461907,-2.8284271247461903);  
\draw (1.5307337294603591,3.695518130045147) -- (-1.5307337294603591,-3.695518130045147);  
\draw (0.,4.) -- (0.,-4.);  
\draw (-1.5307337294603591,3.695518130045147) -- (1.5307337294603591,-3.695518130045147);  
\draw (-2.8284271247461907,2.8284271247461903) -- (2.8284271247461907,-2.8284271247461903);  
\draw (-3.6955181300451474,1.5307337294603591) -- (3.6955181300451474,-1.5307337294603591);  
\draw (-2.4060060929307725,3.1954866109530178) -- (0.55824702790819836,-3.960853475683197);  
\draw (-3.960853475683197,0.55824702790819836) -- (3.1954866109530173,-2.4060060929307721);  
\draw (3.960853475683197,0.55824702790819836) -- (-3.1954866109530173,-2.4060060929307721);  
\draw (2.4060060929307725,3.1954866109530178) -- (-0.55824702790819836,-3.960853475683197);  
\draw (-0.55824702790819836,3.960853475683197) -- (2.4060060929307725,-3.1954866109530178);  
\draw (-3.1954866109530173,2.4060060929307721) -- (3.960853475683197,-0.55824702790819836);  
\draw (3.1954866109530173,2.4060060929307721) -- (-3.960853475683197,-0.55824702790819836);  
\draw (0.55824702790819836,3.960853475683197) -- (-2.4060060929307725,-3.1954866109530178);  
\end{tikzpicture}
\begin{tikzpicture}[scale=0.75]
\draw (4.,0.) -- (-4.,0.);  
\draw (3.7587704831436342,1.3680805733026751) -- (-3.7587704831436342,-1.3680805733026751);  
\draw (3.0641777724759125,2.571150438746157) -- (-3.0641777724759125,-2.571150438746157);  
\draw (2.,3.4641016151377548) -- (-2.,-3.4641016151377548);  
\draw (0.69459271066772088,3.939231012048833) -- (-0.69459271066772088,-3.939231012048833);  
\draw (-0.69459271066772088,3.939231012048833) -- (0.69459271066772088,-3.939231012048833);  
\draw (-2.,3.4641016151377548) -- (2.,-3.4641016151377548);  
\draw (-3.0641777724759125,2.571150438746157) -- (3.0641777724759125,-2.571150438746157);  
\draw (-3.7587704831436342,1.3680805733026751) -- (3.7587704831436342,-1.3680805733026751);  
\draw (-2.2643309399536973,3.2973937275321559) -- (0.38494569838188053,-3.9814340141834936);  
\draw (-3.8541019662496847,1.0704662693192697) -- (2.8541019662496856,-2.8025170768881473);  
\draw (3.9877922042991591,-0.31227125280892087) -- (-3.6404958489653017,-1.6573442532154961);  
\draw (3.2555501505834181,2.3240897609679978) -- (-1.7234612643454619,-3.6096649803410763);  
\draw (1.,3.872983346207417) -- (1.,-3.872983346207417);  
\draw (-1.7234612643454619,3.6096649803410763) -- (3.2555501505834181,-2.3240897609679978);  
\draw (-3.6404958489653017,1.6573442532154961) -- (3.9877922042991591,0.31227125280892087);  
\draw (2.8541019662496856,2.8025170768881473) -- (-3.8541019662496847,-1.0704662693192697);  
\draw (0.38494569838188053,3.9814340141834936) -- (-2.2643309399536973,-3.2973937275321559);  

\end{tikzpicture}

\caption{Projective pictures of $\Ac(9,1)$, $\Ac(12,1)$, $\Ac(17,1)$, and $\Ac(18,1)$.}
\label{ex A(n,1)}
\end{figure}
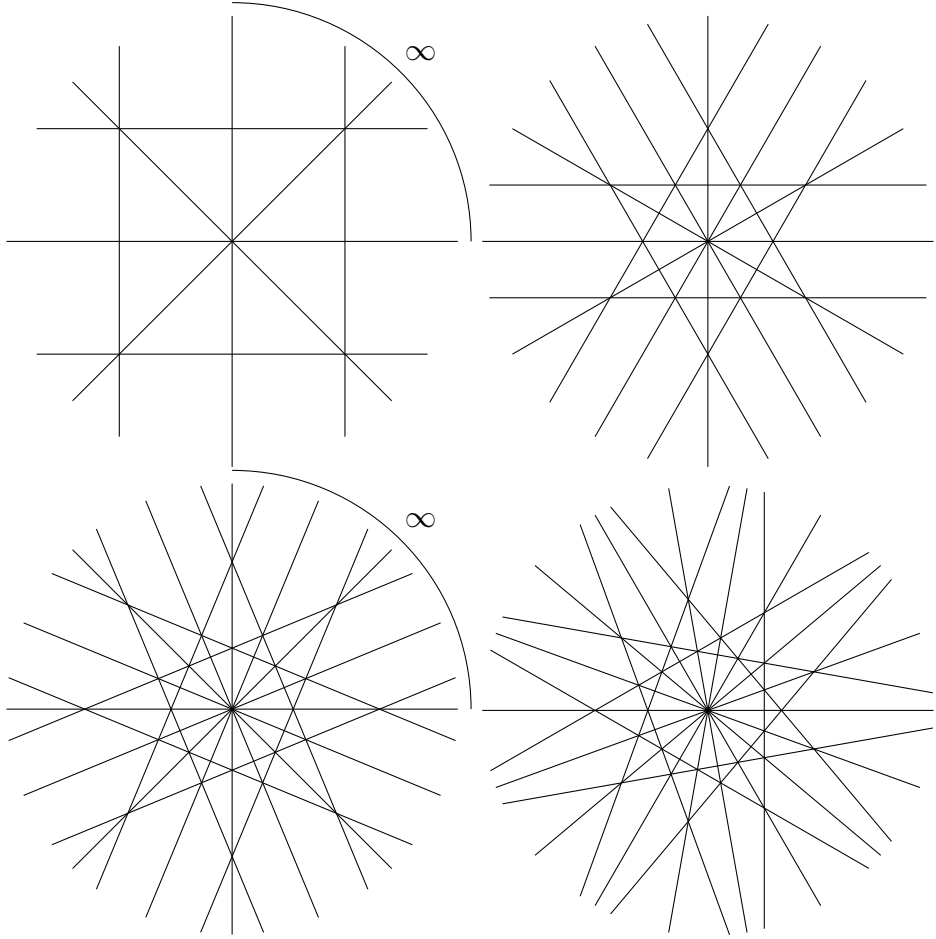

\begin{remar}\label{Rem_LatticeA2nA4n+1}
Let $\Ac = \Ac(2n,1) = \{H_1,\ldots,H_n,I_1,\ldots,I_n\}$ for an $n \geq 3$
where $H_i = (-s_n(i-1),c_n(i-1),0)^\perp$ and $I_j=(c_n(2j-1),s_n(2j-1),1)^\perp$, and let $X = \cap_{i=1}^n H_i \in L_2(\Ac)$.
The following facts are easily seen from the definition: 
\begin{itemize}
\item The rank $2$ part of the intersection lattice has the following form:
\begin{align*}
 L_2(\Ac) = \{X\} & \dot{\cup} \{ I_i \cap I_j \cap H_k \mid 1 \leq i,j,k \leq n, i\neq j \text{ and } i+j \equiv k \:(\md n) \} \\
    & \dot{\cup} \{I_i \cap H_k \mid 1 \leq i,k \leq n \text{ and } k \equiv 2i \:(\md n) \}.
\end{align*}
\item The intersection $X$ is modular and hence $\Ac$ is supersolvable.
\item We have the following multiset of invariants of $L(\Ac)$:
  $$
     \{ \{\vert \Ac_Z \vert \mid Z \in L_2(\Ac) \} \} = \{\{2^n,3^{\vert L_2(\Ac) \vert - n - 1},n^1 \}\}.
  $$
  By Lemma \ref{point charact} the arrangement $\Ac$ is simplicial.
\end{itemize}
Now let $\Ac = \Ac(4n+1,1) = \{H_1,\ldots,H_{2n},I_1,\ldots,I_{2n},J\}$ for an $n \geq 2$
where $H_i = (-s_{2n}(i-1),c_{2n}(i-1),0)^\perp$, $I_j=(c_{2n}(2j-1),s_{2n}(2j-1),1)^\perp$ for $1 \leq i,j \leq 2n$,
$J=(0,0,1)^\perp$, and let $X= \cap_{i=1}^{2n} H_i \in L_2(\Ac)$.
Then similarly we have:
\begin{itemize}
\item The rank $2$ part of the intersection lattice has the following form:
\begin{align*}
 L_2(\Ac) = \{X\} & \dot{\cup} \{ I_i \cap I_j \cap H_k \mid 1 \leq i,j,k \leq n, i\not\equiv j \:(\md n), \text{ and } i+j \equiv k \:(\md 2n)  \} \\
    & \dot{\cup} \{ I_i \cap I_{i+n} \cap H_k \cap J \mid 2i+n \equiv k \:(\md 2n) \text{ where } 1 \leq i \leq n\} \\
    & \dot{\cup} \{I_i \cap H_k \mid 1 \leq i,k \leq 2n \text{ and } k \equiv 2i \:(\md 2n) \} \dot{\cup} \{ H_{2i-1} \cap J \mid 1 \leq i \leq n\}. 
\end{align*}
\item The intersection $X$ is modular and hence $\Ac$ is supersolvable.
\item We have the following multiset of invariants of $L(\Ac)$:
  $$
     \{ \{\vert \Ac_Z \vert \mid Z \in L_2(\Ac) \} \} = \{\{2^{3n},3^{\vert L_2(\Ac) \vert - 4n - 1},4^{n},(2n)^1 \}\}.
  $$
  By Lemma \ref{point charact} the arrangement $\Ac$ is simplicial.
\end{itemize}
\end{remar}

From the remark we immediately get:
\begin{lemma}\label{Lem_Rc12ISSS}
Let $\Ac$ be a $3$-arrangement which is $L$-equivalent to one of the arrangements in $\Rc(1) \cup \Rc(2)$. 
Then $\Ac$ is irreducible, supersolvable and simplicial.
\end{lemma}

\begin{lemma}\label{charpoly facts s3}
 Let $\Ac$ be a simplicial $3$-arrangement such that $\chi_{\Ac}=(t-1)(t-a)(t-b)$ factors over $\NN$.
 If $|\Ac|$ is even, then exactly one of the numbers $a,b$ is even.
 If $|\Ac|$ is odd, then $a,b$ are also odd.
\begin{proof}
Compare the coefficient of $t$, i.e.\ by Remark \ref{rem:CharPolyRk3} we have
$$
  ab + \vert\Ac\vert-1 = ab +a + b = \mu_2.
$$
Since $\Ac$ is simplicial we further have $\mu_2 = 2\vert L_2(\Ac) \vert -3$ by Lemma \ref{point charact}.
So $\mu_2$ is always odd and we obtain $ab\equiv \vert\Ac\vert \:(\md 2)$.
Thus $|\Ac|$ is odd if and only if both $a$ and $b$ are odd.
We further have $a+b+1=\vert\Ac\vert$. Hence if $|\Ac|$ is even then exactly one of the numbers $a,b$ is even.
\end{proof}
\end{lemma}

\begin{lemma}\label{rk3 modular graph}
Let $\Ac$ be an irreducible simplicial $3$-arrangement, $X \in L_2(\Ac)$ a modular element,
$n = \vert \Ac_X \vert$, and $K \in \Kc(\Ac)$ a chamber with $\langle \overline{K}\cap X \rangle= X$.
Then the Coxeter graph $\Gamma(K)$ is the graph of Figure \ref{Graph mod Elt}.
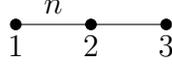
\begin{figure}
\begin{tikzpicture}
\filldraw [black]
	(-1,0) circle [radius=2pt]
	(0,0) circle [radius=2pt]
	(1,0) circle [radius=2pt];
\draw
	(-1,0) -- (0,0) -- (1,0);
\node[below] at (-1,0) {$1$};
\node[below] at (0,0) {$2$};
\node[below] at (1,0) {$3$};
\node[above] at (-0.5,0)  {$n$};
\end{tikzpicture}
\caption{The Coxeter graph of a chamber adjacent to the modular element $X$} \label{Graph mod Elt}
\end{figure}
\end{lemma}
\begin{proof}
Let $B^K = \{ \alpha_1,\alpha_2,\alpha_3\}$, and $\Vg = \{1,2,3\}$ the corresponding vertices of $\Gamma(K) = (\Vg,\Eg)$.
Since $\Ac$ is irreducible by Lemma \ref{simpl irred graph} the graph $\Gamma(K)$ is connected.
We may assume that $\{1,2\},\{2,3\} \in \Eg$ and that $m(1,2)=n$ by Corollary \ref{Cor_Mod_ge3}.

First suppose that $\{1,3\} \in \Eg$ and let $H = \sigma^K_1(\alpha_3)^\perp$. Then in particular
$H \in \Ac \setminus \Ac_X$ so $\vert \Ac^H \vert =n$ by Lemma \ref{Lem_ModCompl}.
Let $\Gamma^H = (\Vg^H,\Eg^H) = \Gamma(K_1^H)$, $\Vg^H=\{\gamma,\delta\}$ and denote the label by $m^H(\gamma,\delta) = \vert \Ac^H \vert$.
But by Lemma \ref{res subgraph}(3) we find that $n = m^H(\gamma,\delta) \geq m(1,2)+m(2,3)-2 = n+1$ which is absurd.

Now suppose that $m(2,3) \geq 4$.
Then $(\sigma^K_3(\alpha_2))^\perp$ (the blue line in Figure \ref{4er intersection}) intersects $(\sigma^{K_1}_2(\sigma^K_1(\alpha_3)))^\perp$ in $Z$.
But $Z$ must lie in $(-\alpha_1)^+ \cap (\sigma^K_1(\alpha_2))^+$ or $\alpha_1^+ \cap \alpha_2^+$
since otherwise $m(2,3) \leq 3$, see Figure \ref{4er intersection}.
This implies $Z + X \not \in L(\Ac)$ which contradicts the modularity of $X$.
\end{proof}

\begin{figure}
\begin{tikzpicture}[scale=1.2]
\draw (4.,0.) -- (-4.,0.);  
\draw (3.282069519732,2.2864863147734256) -- (-3.9839361903750659,-0.35812348571380725);  
\draw (-3.9839361903750659,0.35812348571380725) -- (3.282069519732,-2.2864863147734256);  
\draw (3.6305300503556284,1.6790627008735415) -- (-1.0231190384045969,-3.8669403193292302);  
\draw (1.,3.872983346207417) -- (1.,-3.872983346207417);  
\draw (-1.0231190384045969,3.8669403193292302) -- (3.6305300503556284,-1.6790627008735415);  

\draw[color=blue] (0.074267351366830145,3.9993104856363626) -- (2.6275993192762352,-3.0159114405663621);  
\draw [color=red,dashed](-3.9981663996434418,0.12110095277162514) -- (3.9111312808197192,-0.83848202377481074);  
\draw[color=blue, dashed]  (2.6275993192762352,3.0159114405663621) -- (0.074267351366830145,-3.9993104856363626);  

\fill[red] (-3.,0.0) circle[radius=2pt];  
\fill[blue] (1.,1.4558809370648094) circle[radius=2pt];  
\fill[red] (1.7391701236163843,-0.57497228670668132) circle[radius=2pt];  
\fill[blue] (1.,-1.4558809370648094) circle[radius=2pt];  
\fill[blue] (1.5298973257271573,0.0) circle[radius=2pt];  

\draw[->] (-1,0) -- (-1,0.3);
\node[above right] at (-1,0.1) {$\alpha_1$};

\draw[->] (0., 1.09191) -- ( 0.102606, 0.810003);
\node[left] at ( 0.102606, 0.810003) {$\alpha_2$};

\draw[->] (1., 0.72794) -- (0.7, 0.72794);
\node[below] at (0.7, 0.72794) {$\alpha_3$};

\node[above] at (-3.,0.0){$X$};

\node[below right] at (1.7391701236163843,-0.57497228670668132) {$Z$};

\node[above] at (-0.,0.2) {$K$};

\end{tikzpicture}

\caption{Proof of Lemma \ref{rk3 modular graph}} \label{4er intersection}

\end{figure}
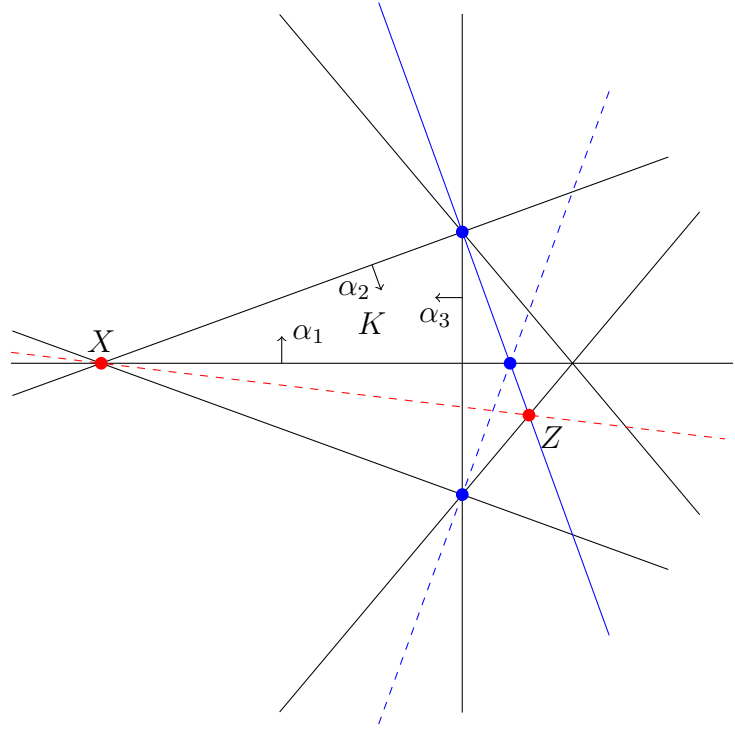

\begin{propo}\label{Prop_A2n1Subarr}
Let $\Ac$ be an \isss $3$-ar\-ran\-ge\-ment, and $X \in L_2(\Ac)$ modular. Set $n := |\Ac_X|$.
Then there is a subarrangement $\Bc \subseteq \Ac$ with $\Ac_X=\Bc_X$ and $\Bc \sim_L \Ac(2n,1)$.
\end{propo}
\begin{proof}
Since $\Ac$ is irreducible we have $n \geq 3$ by Corollary \ref{Cor_Mod_ge3}.
We define
$$
\Kc_X := \{K \in \Kc(\Ac)\mid \langle\overline{K}\cap X\rangle = X \},
$$
i.e.\ the subset of chambers adjacent to $X$,
and the subarrangement
$$
\Bc := \bigcup_{K \in \Kc_X} W^K.
$$
Note that by the definition of $\Bc$ and $\Kc_X$ we have $\Ac_X = \Bc_X$ and $|\Kc_X|=2n$.
Furthermore, for each $K \in \Kc_X$ we have $|W^K \setminus \Bc_X|=1$, and for each $H \in \Bc \setminus \Bc_X$
there are exactly two adjacent chambers $K,K' \in \Kc_X$ with $H \in W^K \cap W^{K'}$ by Lemma \ref{rk3 modular graph}.
Hence we have $|\Bc \setminus \Bc_X| = n$ and thus $\vert \Bc \vert = 2n$.

In the following we consider the projective picture of $\Ac$ respectively $\Bc$.
Then in this picture the $n$ lines in $\Bc \setminus \Bc_X$ are the edge-lines of a convex $n$-gon.
By Lemma \ref{rk3 modular graph} all chambers $K \in \Kc_X$ have the Coxeter graph
of Figure \ref{Graph mod Elt} and for those we have $\Bc_Y = \Ac_Y$ for all $Y \in L_2(W^K)$,
i.e.\ no line of $\Ac \setminus \Bc$ intersects
the convex $n$-gon. 

Now let $\Kc_X = \{K^1,\ldots,K^{2n}\}$ such that $K^{i}$ and $K^{j}$ are adjacent for $1\leq i,j \leq 2n$ with $j-i \equiv \pm1 \:(\md 2n)$.
Let $\Bc_X = \{H_1,\ldots,H_n\}$ and $\Bc \setminus \Bc_X = \{I_1,\ldots,I_n\}$ such that
$H_{a_i}$ and $H_{b_i}$ are walls of $K^i$ with $a_i \equiv b_i-1 \equiv i \:(\md n)$, $1 \leq a_i,b_i \leq n$, $1\leq i \leq 2n$, 
and $I_k$ is a wall of both the chambers $K^{2k-1}$ and $K^{2k}$ for $1 \leq k \leq n$.
Note that with this labeling for $1 \leq i,j \leq n$ we have $|\Bc_{I_i \cap H_j}| = 2$ if $2i \equiv j \:(\md n)$
by Lemma \ref{rk3 modular graph} 
(since each edge of the $n$-gon contains one such point).
The subarrangement $\Bc$ is supersolvable with modular element $X$ because $\Ac$ is supersolvable with modular element $X$ and $\Ac_X = \Bc_X$.
Since exactly $2$ edge-lines of the convex $n$-gon intersect in a common point we further have $\vert \Bc_Y \vert \leq 3$ for
all $Y \in L_2(\Bc) \setminus \{X\}$.
Suppose there is a $Y \in L_2(\Bc)$ with $\vert \Bc_Y \vert = 2$ and $Y \notin L_2(W^K)$ for any chamber $K \in \Kc_X$,
i.e.\ $Y$ is an intersection outside of the $n$-gon.
By the modularity of $X$ in $L(\Bc)$ we have $Y = I_i \cap H_j$ for some $1\leq i,j \leq n$. 
But then $\vert \Bc^{I_i} \vert \geq n+1$ contradicting Lemma \ref{Lem_ModCompl}.
Thus all intersections $Y$ outside the $n$-gon are of size $3$, in particular $\Bc_Y = \{ I_i,I_j,H_k \}$ for some $1\leq i<j \leq n$, and $1\leq k \leq n$. 
We obtain the following
multiset of invariants of the intersection lattice of $\Bc$:
$$
 \{ \{\vert \Bc_Y \vert \mid Y \in L_2(\Bc) \} \} = \{\{2^n,3^{\vert L_2(\Bc) \vert - n - 1},n^1 \}\}.
$$
In particular by Remark \ref{Rem_LatticeA2nA4n+1} we have 
$$
\{ \{\vert \Bc_Y \vert \mid Y \in L_2(\Bc) \}\}  = \{ \{\vert \Ac(2n,1)_Z \vert \mid Z \in L_2(\Ac(2n,1)) \} \}.
$$
To finally see with Remark \ref{Rem_LatticeA2nA4n+1} that $\Bc \sim_L \Ac(2n,1)$ we claim that 
$$
\{ I_i \cap I_j \cap H_k \mid 1 \leq i,j,k \leq n, i\neq j \text{ and } i+j \equiv k \:(\md n) \} \subseteq L_2(\Bc).
$$
Let $1 \leq i < j \leq n$. Without loss of generality we may assume $i=1$.
We have $\Bc^{I_1} = \{H_j \cap I_1 \mid 1 \leq j \leq n \}$ and there is exactly one simple intersection $H_2\cap I_1$. 
But from the projective picture we see that the next intersection point $H_3\cap I_1$ on $I_1$ is contained in the next edge-line $I_2$ of the $n$-gon.
Continuing this way gives exactly $\Bc^{I_1} = \{I_1\cap H_2, I_1 \cap I_2 \cap H_3, \ldots, I_1 \cap I_{n-1} \cap H_n, I_1\cap I_n \cap H_1 \}$ and the claim follows.
\end{proof}

\begin{remar}\label{Rem_LEquiField}
Let $\Ac \in \Rc(1)\cup\Rc(2)$.
Then by \cite[Thm.~3.6]{p-C10b} there exists a minimal subfield $\LL \leq \RR$ such that there is
an arrangement $\Bc$ in $\LL^3$ with $L(\Bc) \cong L(\Ac)$. Furthermore 
if $\Bc'$ is another arrangement in $\LL^3$ which is $L$-equivalent to $\Bc$,
then there is a collineation (projective semi-linear transformation) $\varphi \in \PGAL(\LL^3)$ with $\Bc'=\varphi(\Bc) = \{\varphi(H) \mid H \in \Bc\}$.
Hence, by the fundamental theorem of projective geometry (see e.g.\ \cite[Sec.\ II.9]{Artin88_GeomAlg}) there is a field automorphism
$\mu$ of $\LL$ and a $\psi \in \GL(\RR^3)$ such that $\psi(\mu(\Bc) \otimes_\LL \RR) = \Ac$.
So any (real) arrangement $\Ac'$ which is $L$-equivalent to $\Ac(2n,1)$ or $\Ac(4m+1,1)$ is essentially this arrangement.
\end{remar}

\begin{propo}\label{Prop_A2n4n}
Let $\Ac$ be an \isss $3$-arrangement with modular element $X \in L_2(\Ac)$, and $n := |\Ac_X|$. 
Let $\Bc$ be the subarrangement from Proposition \ref{Prop_A2n1Subarr} 
which is lattice equivalent to $\Ac(2n,1)$.
Then 
\begin{enumerate}
\item $| \Ac \setminus \Bc | \leq 1$, and 
\item If $\Ac = \Bc \dot{\cup} \{J\}$ then $n$ is even and $\Ac \sim_L \Ac(4\frac{n}{2}+1,1)$.
\end{enumerate}
\end{propo}
\begin{proof}
By the preceding remark we may assume that $\Bc=\Ac(2n,1)$ and is given as in Definition \ref{def:A2n}.

\begin{figure}
\def \sc {1}

\begin{tikzpicture}[scale=\sc]
\draw[dashed] (-1.1547005383792515,3.8297084310253524) -- (-1.1547005383792515,-3.8297084310253524);  
\node at (-1.1547005383792515,-4.0381513922420824) {\small $J$};
\draw (3.9984686968287599,0.11067104626107283) -- (-3.7135182035477854,-1.4865337372287333);  
\draw (4.,0.) -- (-4.,0.);  
\node at (-4.2000000000000002,0.) {\small $H$};
\draw (-3.7135182035477854,1.4865337372287333) -- (3.9984686968287599,-0.11067104626107283);  
\draw (-2.807973978395951,2.8487334267444564) -- (2.4463191254730421,-3.1647310685656707);  
\node at ($(-2.807973978395951,2.8487334267444564)+(-0.1,0.2)$) {\small $I_1$};
\draw (2.4463191254730421,3.1647310685656707) -- (-2.807973978395951,-2.8487334267444564);  
\node at ($(-2.9397950640386408,-2.9996008036827244)+(0.1,-0.1)$) {\small $I_1'$};
\draw (2.1522938065798431,3.3715918154720406) -- (-3.8667252762004534,-1.0239314617651551);  
\node at (-4.0394692976975994,-1.1500816462140691) {\small $I_2$};
\draw (-3.8667252762004534,1.0239314617651551) -- (2.1522938065798431,-3.3715918154720406);  
\node at (2.3250378280769888,-3.4977419999209545) {\small $I_2'$};

\fill[red] (-1.1547005383792515,-0.95658524695240099) circle[radius=2pt];  
\fill[red] (-1.1547005383792515,0.95658524695240099) circle[radius=2pt];  
\fill[red] (3.4641016151377548,0.0) circle[radius=2pt];  

\node[above] at (3.4641016151377548,0.0) {$X$};

\node at ($(-1.1547005383792515,0.95658524695240099) + 0.25*(-1,-2)$) {$K$};
\node at ($(-1.1547005383792515,-0.95658524695240099) + 0.25*(-1,2)$) {$K'$};
\end{tikzpicture}
\caption{The structure of $L(\Bc)$ yields only one possibility for $J \in \Ac\setminus \Bc$.}\label{fig:HinAwoB}
\end{figure}
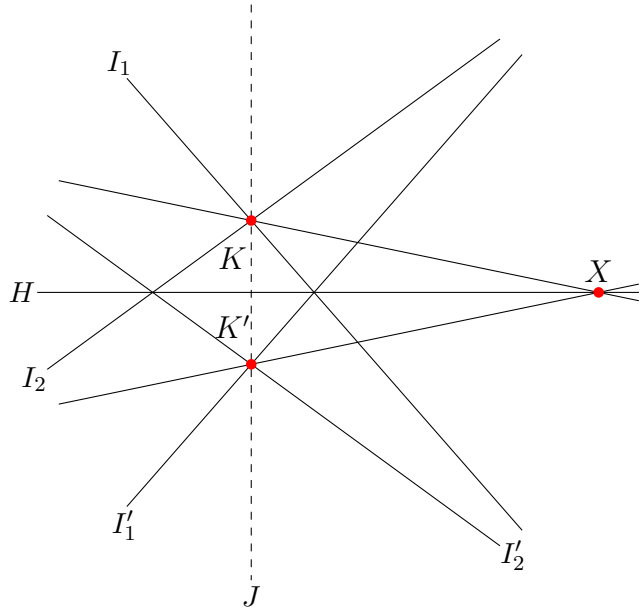

Assume there is a $J \in \Ac \setminus \Bc$. Then for all $I' \in \Bc \setminus \Ac_X$, $Y=J\cap I'$ we have
$\vert \Bc_Y \vert = 3$ since otherwise $\vert \Ac^J \vert \geq n+1$ contradicting Lemma \ref{Lem_ModCompl}.
In particular if $K \in \Kc(\Bc)$ such that $K \cap J \neq \emptyset$,
$W^K=\{H,I_1,I_2\}$ (since $\overline{K}\cap X = \{0\}$)
with $I \in \Bc_X$, $H_1,H_2 \in \Bc \setminus \Bc_X$, then
$I_1\cap I_2 \subseteq J$. Furthermore for the adjacent chamber $K'$ with
$\langle \overline{K}\cap\overline{K'}\rangle = H$, $W^{K'}=\{H,I_1',I_2'\}$
we also have $J\cap K' \neq \emptyset$ so similarly $I_1'\cap I_2' \subseteq J$.
Since $I_1,I_2,I_1',I_2'$ are pairwise different, $J = I_1\cap I_2 + I_1'\cap I_2'$, see Figure \ref{fig:HinAwoB}.
Let $\tilde{J} \in \Ac \setminus \Bc$ be another hyperplane.
Then there exists a chamber $K \in \Kc(\Bc)$ such that $J\cap K \neq \emptyset$ and $\tilde{J}\cap K \neq \emptyset$ 
(since otherwise there is an $I' \in \Bc \setminus \Bc _X$ such that
$\tilde{J}\cap I' \notin L_2(\Bc)$ which contradicts the modularity of $X$).
Hence $J=\tilde{J}$.
So there is only one possibility for such a $J$ and we obtain $\vert \Ac \setminus \Bc \vert \leq 1$.

Now suppose $n = \vert \Ac_X \vert$ is odd. Since $\Ac$ is supersolvable with modular element $X \in L_2(\Ac)$
 by Lemma \ref{ss charpoly fact} we have
$$
	\chi_\Ac(t) = (t-1)(t-a)(t-b),
$$
with $a = n-1$ and $b = \vert \Ac \vert - n$.
The first root $a$ is even so by Lemma \ref{charpoly facts s3} $b$ has to be odd,
i.e.\ $\vert \Ac \vert$ is even and hence $\Ac = \Bc$.
If $n$ is even then either $\Ac = \Bc$
or there is one more hyperplane $J \in \Ac \setminus \Bc$ which has to be $J = (0,0,1)^\perp$
after a possible coordinate change and $\Ac = \Ac(4\frac{n}{2} + 1,1)$.
\end{proof}

We are now prepared to prove the main result of this section. 
Notice that if $\Ac$ is not assumed to be finite, then one also obtains an infinite arrangement described in \cite{p-CG17}.

\begin{theor}\label{classification rk 3}
Let $\Ac$ be an \isss $3$-ar\-ran\-ge\-ment.
Then $\Ac$ is lattice equivalent to exactly one of the arrangements in $\Rc(1)\cup\Rc(2)$.
\end{theor}
\begin{proof}
By Lemma \ref{Lem_Rc12ISSS} all arrangements in $\Rc(1)\cup\Rc(2)$ are irreducible, supersolvable, and simplicial.

Conversely by Proposition \ref{Prop_A2n1Subarr} and Proposition \ref{Prop_A2n4n} we have $\Ac \sim_L \Ac(2n,1)$ if $n$ is odd,
or $\Ac \sim_L \Ac(2n,1)$ or $\Ac \sim_L \Ac(4\frac{n}{2}+1,1)$ if $n$ is even.
\end{proof}

From the proof of Proposition \ref{Prop_A2n1Subarr} we obtain the following corollary.

\begin{corol}\label{SSs3 leq 4 outside mod}
Let $\Ac$ be an \isss $3$-ar\-rangement and $X \in L_2(\Ac)$ a modular element.
Then for all $X' \in L_2(\Ac) \setminus \{X\}$ we have $\vert \Ac_{X'} \vert \leq 4$.
\end{corol}
 
From the proof of Proposition \ref{Prop_A2n4n} we obtain:
\begin{corol}\label{Coxeter graph rk3}
Let $\Ac$ be an \isss $3$-ar\-rangement, $X \in L_2(\Ac)$ modular with $n = \vert \Ac_X \vert$,
and $K \in \Kc(\Ac)$.
Then $\Gamma(K)$ is one of the Coxeter graphs of Figure \ref{fig:Coxeter graphs rk3}.
In particular, if $\vert \Ac \vert$ is even or $n \leq 5$,
then there is no chamber $K \in \Kc(\Ac)$ such that $\Gamma(K) = \Gamma_3^5$
and if $n > 4$ and $\vert  \Ac \vert$ is even then there is also no chamber $K \in \Kc(\Ac)$ such that $\Gamma(K) = \Gamma_3^2$.
\begin{figure}
\begin{tikzpicture}
\node at (-1.5,0) {$\Gamma_3^1$};
\filldraw [black]
	(-1,0) circle [radius=2pt]
	(0,0) circle [radius=2pt]
	(1,0) circle [radius=2pt];
\draw
	(-1,0) -- (0,0) -- (1,0);
\node[above] at (-0.5,0)  {$n$};
\end{tikzpicture}\hspace{0.5cm}
\begin{tikzpicture}
\node at (-1.5,0) {$\Gamma_3^2$};
\filldraw [black]
	(-1,0) circle [radius=2pt]
	(0,0) circle [radius=2pt]
	(1,0) circle [radius=2pt];
\draw
	(-1,0) -- (0,0) -- (1,0);
\node[above] at (-0.5,0)  {$4$};
\end{tikzpicture}\hspace{0.5cm}
\begin{tikzpicture}
\node at (-1.5,0) {$\Gamma_3^3$};
\filldraw [black]
	(-1,0) circle [radius=2pt]
	(0,0) circle [radius=2pt]
	(1,0) circle [radius=2pt];
\draw
	(-1,0) -- (0,0) -- (1,0);
\end{tikzpicture}

\vspace{0.5cm}
\begin{tikzpicture}
\node at (-0.5,0) {$\Gamma_3^4$};
\filldraw [black]
	(0,0.5) circle [radius=2pt]
	(0,-0.5) circle [radius=2pt]
	(1,0) circle [radius=2pt];
\draw
	(0,0.5) -- (0,-0.5) -- (1,0) -- (0,0.5);
\end{tikzpicture}\hspace{0.5cm}
\begin{tikzpicture}
\node at (-0.75,0) {$\Gamma_3^5$};
\filldraw [black]
	(0,0.5) circle [radius=2pt]
	(0,-0.5) circle [radius=2pt]
	(1,0) circle [radius=2pt];
\draw
	(0,0.5) -- (0,-0.5) -- (1,0) -- (0,0.5);
\node[left] at (0,0)  {$4$};
\end{tikzpicture}
\caption{Possible Coxeter graphs for an \isss $3$-arrangement.}\label{fig:Coxeter graphs rk3}
\end{figure}
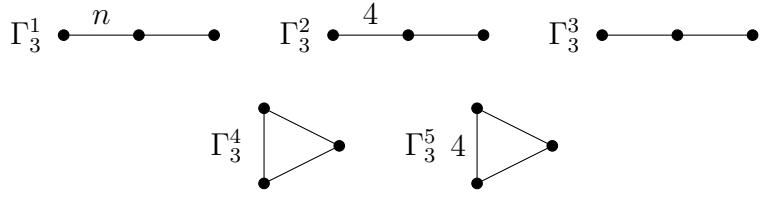
 \end{corol}

\begin{lemma}\label{rk 3 size rest}
Let $\Ac$ be an \isss $3$-ar\-rangement and $H \in \Ac$.
Then for all $H \in \Ac$ we have
$$
 	\vert \Ac^H \vert \geq \lceil \frac{\vert \Ac \vert}{4}\rceil + 1.
$$
\end{lemma}
\begin{proof}
Let $X \in L_2(\Ac)$ be modular, $n= \vert \Ac_X \vert$, and $H \in \Ac$.
If $H \in \Ac \setminus \Ac_X$ then by Lemma \ref{Lem_ModCompl} we have
$\vert \Ac^H \vert = n \geq \frac{\vert \Ac \vert}{2} \geq \lceil \frac{\vert \Ac \vert}{4} \rceil +1$.

Let $t_r^H := \vert \{ X \in \Ac^H \mid \vert \Ac_X \vert = r \} \vert$.
Then we always have the identity $\sum_{r \geq 2} (r-1)t_r^H = \vert \Ac \vert - 1$.
By Corollary \ref{SSs3 leq 4 outside mod} for $H \in \Ac_X$ we see that $t_r^H = 0$ for $r \notin \{2,3,4,n\}$,
and $t_n^H=1$.
Furthermore by Theorem \ref{classification rk 3} we have
$t_2^H \in \{0,1,2\}$ and $t_4^H=1$ if and only if $\vert \Ac \vert = 2n+1$ and $n$ is even.
So we obtain
$$
t_3^H = \frac{\vert \Ac \vert  - 1 - t_2^H - 3t_4^H - (n-1)t_n^H}{2} = \frac{\vert \Ac \vert - n - t_2^H - 3t_4^H}{2},
$$
and hence
$$
	\vert \Ac^H \vert = t_2^H + t_3^H + t_4^H + t_n^H = \frac{n + t_2^H}{2} +1 \geq \lceil \frac{\vert \Ac \vert}{4} \rceil +1.
$$
\end{proof}

\section{The rank $4$ case}\label{sec:rk4}

The following proposition and its immediate corollary are the key for the classification of irreducible
supersolvable simplicial arrangements of rank $\ell \geq 4$.

\begin{propo}\label{rk4 only 4er or smaller}
Let $\Ac$ be an \isss $4$-ar\-rangement.
Then for all $X \in L_2(\Ac)$ we have $\vert\Ac_X\vert \leq 4$.
\end{propo}
\begin{proof}
The proof is in three steps.
First we show that if $X \in L_2(\Ac)$ with $\vert \Ac_X \vert \geq 5$ then $X$
necessarily has to be the only rank $2$ modular element in $L(\Ac)$.
From this we derive that $\vert \Ac_X \vert \leq 6$.
Finally by some geometric arguments and using the classification in dimension $3$
we exclude the cases $\vert \Ac_X \vert = 5, 6$.

Let $X \in L_2(\Ac)$ be fixed and suppose $\vert\Ac_X\vert \geq 5$.

First assume that there is a modular $X' \in L_2(\Ac) \setminus \{X\}$.
By the irreducibility of $\Ac$ there is an $H \in \Ac$ transversal to $X$ and $X'$, i.e.\ such that $X \nsubseteq H$, $X' \nsubseteq H$,
and also $X \cap X' \nsubseteq H$ if $X \cap X' \in L_3(\Ac)$.
Let $Y = H \cap X$ and $Y' = H \cap X'$.
By Lemma \ref{SS restriction} and Lemma \ref{res irred} the restriction $\Ac^H$ is an \isss $3$-arrangement.
Furthermore, $Y \neq Y'$ and $5 \leq \vert \Ac^H_Y \vert \leq \vert \Ac^H_{Y'} \vert$ for the $3$-arrangement $\Ac^H$
by Lemma \ref{X mod max} since $Y'$ is a modular element in $L_2(\Ac^H)$ by Lemma \ref{SS restriction}.
But this contradicts Corollary \ref{SSs3 leq 4 outside mod},
the irreducible \sss $3$-arrangement $\Ac^H$ cannot have two distinct rank 2 intersections of size greater or equal to $5$,
one of them modular.
Hence $X$ is the only modular element in $L_2(\Ac)$ and also the one single element in $L_2(\Ac)$
with $\vert \Ac_X \vert \geq 5$.

From now on to the end of the proof let $Y \in L_3(\Ac)$ be a fixed modular intersection of rank $3$
with $Y > X$.

Suppose that $\vert \Ac_X \vert \geq 7$.
Then since $\Ac$ is irreducible, by Lemma \ref{SS localization} the localization $\Ac_Y/Y$
regarded as an essential 3-arrangement in $V / Y$ is an \isss 3-arrangement with modular element $X/Y \in L_2(\Ac_Y/Y))$.
So by Theorem \ref{classification rk 3} we have $\vert \Ac_Y \vert \geq 14$.
Let $H \in \Ac_X$.
By Lemma \ref{res irred} the restriction $\Ac^H$ is irreducible and by Corollary \ref{ex irred loc}
there is a $Y' \in L_2(\Ac^H) \setminus \{Y\}$ with
$Y' \subseteq X$ such that $\vert (\Ac^H)_{Y'} \vert \geq 3$.
Since $\Ac_{Y'}/Y'$ is an \isss $3$-arrangement with modular element $X/Y'$,
as for $\Ac_Y$ we have $\vert \Ac_{Y'} \vert \geq 14$.
By Lemma \ref{SS restriction} the rank 3 intersection $Y\cap H = Y$ is modular in $L(\Ac^H)$ for $H \in \Ac_X$.
By Lemma \ref{rk 3 size rest} we further have $\vert (\Ac^H)_Y \vert = \vert (\Ac_Y)^H \vert \geq 5$
and similarly $\vert (\Ac^H)_{Y'} \vert \geq 5$.
Because of the choice of $Y' \in L_2(\Ac^H) \setminus \{Y\}$ the \isss $3$-arrangement $\Ac^H$ has two distinct  rank 2 intersections of size greater or equal to $5$
which contradicts Corollary \ref{SSs3 leq 4 outside mod}.
Hence $\vert \Ac_X \vert \leq 6$.

To exclude the cases $\vert \Ac_X \vert \in \{5,6\}$ first assume that $\vert \Ac_X \vert =6$.
\begin{figure}
\def \sc {0.48}
\begin{tikzpicture}[scale=\sc]
\node[below] at (0,-4) {$\Ac(10,1)$};
\draw [line width=0.5mm] (4.,0.) -- (-4.,0.);  
\draw (3.2360679774997898,2.3511410091698925) -- (-3.2360679774997898,-2.3511410091698925);  
\draw (1.2360679774997896,3.804226065180615) -- (-1.2360679774997896,-3.804226065180615);  
\draw (-1.2360679774997896,3.804226065180615) -- (1.2360679774997896,-3.804226065180615);  
\draw (-3.2360679774997898,2.3511410091698925) -- (3.2360679774997898,-2.3511410091698925);  
\draw (-3.0854994876500208,2.5455240937204779) -- (1.4674654989001255,-3.7210945983054238);  
\draw (3.9924430432881204,0.24576115661408859) -- (-3.3744090545382268,-2.1478741892043964);  
\draw (1.,3.872983346207417) -- (1.,-3.872983346207417);  
\draw (-3.3744090545382268,2.1478741892043964) -- (3.9924430432881204,-0.24576115661408859);  
\draw (1.4674654989001255,3.7210945983054238) -- (-3.0854994876500208,-2.5455240937204779);  

\def \colp {red}
\fill[\colp] (0.0,0.0) circle[radius=1.5mm];  
\fill[\colp] (-1.2360679774997894,0.0) circle[radius=1.5mm];  
\fill[\colp] (3.2360679774997898,0.0) circle[radius=1.5mm];  
\fill[\colp] (1.,0.0) circle[radius=1.5mm];  

\end{tikzpicture}
\begin{tikzpicture}[scale=\sc]
\node[below] at (0,-4) {$\Ac(12,1)$};
\draw [line width=0.5mm]  (4.,0.) -- (-4.,0.);  
\draw [line width=0.5mm]  (3.4641016151377548,2.) -- (-3.4641016151377548,-2.);  
\draw (2.,3.4641016151377548) -- (-2.,-3.4641016151377548);  
\draw (0.,4.) -- (0.,-4.);  
\draw (-2.,3.4641016151377548) -- (2.,-3.4641016151377548);  
\draw (-3.4641016151377548,2.) -- (3.4641016151377548,-2.);  
\draw (-2.8025170768881469,2.8541019662496847) -- (1.0704662693192699,-3.8541019662496847);  
\draw (3.872983346207417,-1.) -- (-3.872983346207417,-1.);  
\draw (2.8025170768881469,2.8541019662496847) -- (-1.0704662693192699,-3.8541019662496847);  
\draw (-1.0704662693192699,3.8541019662496847) -- (2.8025170768881469,-2.8541019662496847);  
\draw (3.872983346207417,1.) -- (-3.872983346207417,1.);  
\draw (1.0704662693192699,3.8541019662496847) -- (-2.8025170768881469,-2.8541019662496847);  

\fill[red] (0.0,0.0) circle[radius=1.5mm];  
\fill[red] (-1.1547005383792515,0.0) circle[radius=1.5mm];  
\fill[red] (1.1547005383792515,0.0) circle[radius=1.5mm];  
\fill[red] (4.2352941176470589,0) circle[radius=1.5mm];
\fill[red] (-0.86602540378443871,-0.5) circle[radius=1.5mm];  
\fill[red] (-1.7320508075688774,-1.) circle[radius=1.5mm];  
\fill[red] (1.7320508075688774,1.) circle[radius=1.5mm];  
\fill[red] (0.86602540378443871,0.5) circle[radius=1.5mm];  

\end{tikzpicture}
\begin{tikzpicture}[scale=\sc]
\node[below] at (0,-4) {$\Ac(13,1)$};
\draw (0,4.2352941176470589) arc [start angle=90, end angle=0, radius=4.2352941176470589] ;
\draw (3.344805190907731,3.344805190907731) node {$\infty$};  
\draw [line width=0.5mm]  (4.,0.) -- (-4.,0.);  
\draw [line width=0.5mm]  (3.4641016151377548,2.) -- (-3.4641016151377548,-2.);  
\draw (2.,3.4641016151377548) -- (-2.,-3.4641016151377548);  
\draw (0.,4.) -- (0.,-4.);  
\draw (-2.,3.4641016151377548) -- (2.,-3.4641016151377548);  
\draw (-3.4641016151377548,2.) -- (3.4641016151377548,-2.);  
\draw (-2.8025170768881469,2.8541019662496847) -- (1.0704662693192699,-3.8541019662496847);  
\draw (3.872983346207417,-1.) -- (-3.872983346207417,-1.);  
\draw (2.8025170768881469,2.8541019662496847) -- (-1.0704662693192699,-3.8541019662496847);  
\draw (-1.0704662693192699,3.8541019662496847) -- (2.8025170768881469,-2.8541019662496847);  
\draw (3.872983346207417,1.) -- (-3.872983346207417,1.);  
\draw (1.0704662693192699,3.8541019662496847) -- (-2.8025170768881469,-2.8541019662496847);  

\fill[red] (0.0,0.0) circle[radius=1.5mm];  
\fill[red] (-1.1547005383792515,0.0) circle[radius=1.5mm];  
\fill[red] (1.1547005383792515,0.0) circle[radius=1.5mm];  
\fill[red] (4.2352941176470589,0) circle[radius=1.5mm];
\fill[red] (3.6641016151377548,2.12)  circle[radius=1.5mm];
\fill[red] (-0.86602540378443871,-0.5) circle[radius=1.5mm];  
\fill[red] (-1.7320508075688774,-1.) circle[radius=1.5mm];  
\fill[red] (1.7320508075688774,1.) circle[radius=1.5mm];  
\fill[red] (0.86602540378443871,0.5) circle[radius=1.5mm];  

\end{tikzpicture}

\caption{$\vert \Ac_Y^H \vert = 4,5,6$ respectively for $H \in \Ac_X$.}\label{A10,12,13}
\end{figure}
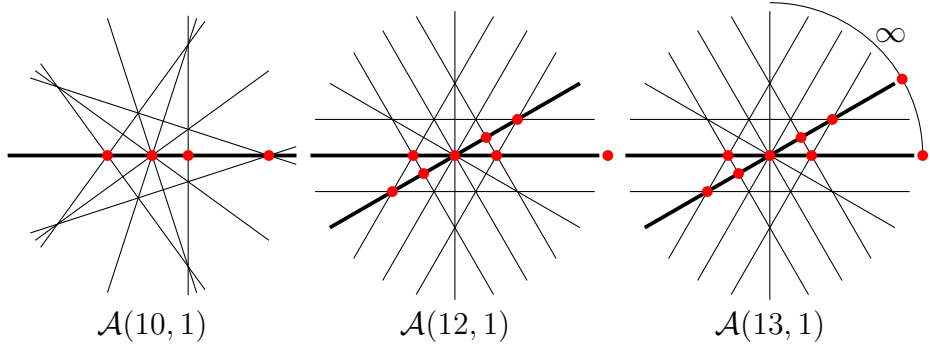
We may assume that there is an $Y' \in L_3(\Ac)$, $Y' \neq Y$, and $Y' > X$ such that $\Ac_{Y'}/Y'$
is an \isss $3$-arrangement.
So we have $\Ac_{Y'}/Y' \sim_L \Ac(12,1)$ or $\Ac_{Y'}/Y' \sim_L \Ac(13,1)$. But then there is an $H \in \Ac_X$
such that $\vert \Ac_{Y'}^H \vert \geq 5$ which is immediately clear by Figure \ref{A10,12,13}.
Since by Lemma \ref{SS restriction} $Y = Y \cap H$ is a rank 2 modular element in $L(\Ac^H)$
different from $Y' \cap H = Y' \in L_2(\Ac^H)$, with
Corollary \ref{SSs3 leq 4 outside mod} we get a contradiction.

Finally suppose $\vert \Ac_X \vert = 5$.
Then we have $\Ac_Y/Y \sim_L \Ac(10,1)$.
Again we may assume that there is an $Y' \in L_3(\Ac)$, $Y' \neq Y$, and $Y' > X$ such that $\Ac_{Y'}/Y'$
is an \isss $3$-arrangement. So $\Ac_{Y'}/Y' \sim_L \Ac(10,1)$. Let $H \in \Ac_X$.
Then $\vert \Ac_{Y}^H \vert = \vert \Ac_{Y'}^H \vert = 4$, see Figure \ref{A10,12,13}.
Since by Lemma \ref{res irred} $\Ac^H$ is an \isss $3$-arrangement with modular element $Y$
by Theorem \ref{classification rk 3} we have $\Ac^H \sim_L \Ac(9,1) \cong \Ac(B_3)$.
For the other restrictions $\Ac^{H'}$ with ${H'} \in \Ac \setminus \Ac_X$ we have $\Ac^{H'} \sim_L \Ac(10,1)$.
The arrangement $\Ac$ is supersolvable and by Theorem \ref{ss charpoly fact} the characteristic polynomial
factors as follows over the integers
$$
\chi_\Ac(t) = (t-1)(t-4)(t-5)(t-(\vert \Ac \vert -10)).
$$
Similarly for $H \in \Ac_X$  by Theorem \ref{ss charpoly fact} we have
$$
\chi_{\Ac^H}(t) = (t-1)(t-3)(t-5),
$$
and for $H \in \Ac \setminus \Ac_X$
$$
\chi_{\Ac^H}(t) = (t-1)(t-4)(t-5).
$$
Now we use Lemma \ref{charpoly charact} and insert the numbers,
\begin{eqnarray*}
0 	&= &\ell\vert\chi_\Ac(-1)\vert - 2 \sum_{H \in \Ac} \vert \chi_{\Ac^H}(-1) \vert \\
	&= & \ell\vert\chi_\Ac(-1)\vert - 2(\sum_{H \in \Ac_X} \vert \chi_{\Ac^H}(-1) \vert +
		 \sum_{H \in \Ac \setminus \Ac_X} \vert \chi_{\Ac^H}(-1) \vert ) \\
 	&= & (4\cdot2\cdot5\cdot6)(\vert \Ac \vert -9) - 2 ( 5\cdot2\cdot4\cdot6 + (\vert \Ac \vert - 5\cdot2\cdot5\cdot6) \\
 	&= & 2\vert \Ac \vert -18 - 4 - \vert \Ac \vert + 5 \\
 	&=& \vert \Ac \vert - 17,
\end{eqnarray*}
so $\vert \Ac \vert = 17$.
Since $\vert \Ac_Y \cup \Ac_{Y'} \vert = 15$ there are exactly $2$ other hyperplanes $H_1, H_2$ not contained
in either $\Ac_Y$ or in $\Ac_{Y'}$.
But then there is a $Z \in L_2(\Ac)$, $Z \subseteq H_i$ for an $i = 1,2$ such that $Z \notin \Ac_Y^{H_i}$.
This contradicts the modularity of $Y$
and finishes the proof.
\end{proof}

From the previous proposition, by taking localizations and Lemma \ref{loc subgraph} we
immediately obtain the following theorem.

\begin{theor}\label{Thm:at most 4er}
Let $\Ac$ be an \isss $\ell$-ar\-rangement with $\ell \geq 4$.
Then for all $X \in L_2(\Ac)$ we have $\vert\Ac_X\vert \leq 4$.
\end{theor}

After establishing this strong constraint, in a sequence of lemmas we will decimate the number of possible
Coxeter graphs for \isss $4$-arrangements. We will use this to derive the crystallographic property at the end of this section.

\begin{lemma}\label{no chordal circles}
Let $\Ac$ be an \isss $4$-ar\-rangement and let $K \in \Kc(\Ac)$ be a chamber.
Then $\Gamma(K)$ has no subgraph of the form shown in Figure \ref{fig:imp_subgraph1}.
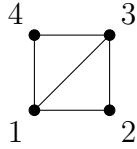
\begin{figure}
\begin{tikzpicture}
\filldraw [black]
	(-0.5,-0.5) circle [radius=2pt]
	(-0.5,0.5) circle [radius=2pt]
	(0.5,-0.5) circle [radius=2pt]
	(0.5,0.5) circle [radius=2pt];
\draw
	(-0.5,-0.5) -- (-0.5,0.5) -- (0.5,0.5) -- (0.5,-0.5) -- (-0.5,-0.5);
\draw
	(-0.5,-0.5) -- (0.5,0.5);
\node[below left] at (-0.5,-0.5) {$1$};
\node[below right] at (0.5,-0.5) {$2$};
\node[above right] at (0.5,0.5) {$3$};
\node[above left] at (-0.5,0.5) {$4$};
\end{tikzpicture}
\caption{Forbidden subgraph.}%
\label{fig:imp_subgraph1}
\end{figure}

\end{lemma}

\begin{proof}
Suppose there exists a chamber $K \in \Kc(\Ac)$ with $B^K =\{ \alpha_1,\ldots,$ $\alpha_4\}$ such that $\Gamma(K)$
has a subgraph of this form.
Then by Lemma \ref{res subgraph} for $H = \sigma^K_1(\alpha_3)^\perp \in \Ac_{13}$
and the chamber $K_1^H \in \Kc(\Ac^H)$ we find that the graph of Figure \ref{fig:contr_imp_subgraph1}
is a subgraph of $\Gamma(K^H)$.

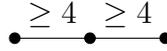
\begin{figure}
\begin{tikzpicture}
\filldraw [black]
	(-1,0) circle [radius=2pt]
	(0,0) circle [radius=2pt]
	(1,0) circle [radius=2pt];
\draw
	(-1,0) -- (0,0) -- (1,0);

\node[above] at (-0.5,0)  {$\geq 4$};
\node[above] at (0.5,0)  {$\geq 4$};
\end{tikzpicture}
\caption{Forbidden subgraph of a chamber in $\Ac^H$ by Corollary \ref{Coxeter graph rk3}.}%
\label{fig:contr_imp_subgraph1}
\end{figure}

But this is a contradiction since by Lemma \ref{SS restriction} and Lemma \ref{res irred}
the restricted arrangement $\Ac^H$ is an
\isss 3-arrangement and such a graph is not contained in the list of Corollary \ref{Coxeter graph rk3}.
\end{proof}

\begin{lemma}\label{no big circles}
Let $\Ac$ be an irreducible \sss $4$-ar\-rangement and let $K \in \Kc(\Ac)$ be a chamber.
Then $\Gamma(K)$ has no subgraph of the form shown in Figure \ref{fig:imp_subgraph2}.
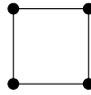
\begin{figure}
\begin{tikzpicture}
\filldraw [black]
	(-0.5,-0.5) circle [radius=2pt]
	(-0.5,0.5) circle [radius=2pt]
	(0.5,-0.5) circle [radius=2pt]
	(0.5,0.5) circle [radius=2pt];
\draw
	(-0.5,-0.5) -- (-0.5,0.5) -- (0.5,0.5) -- (0.5,-0.5) -- (-0.5,-0.5);
\end{tikzpicture}
\caption{No $4$-cycles.}%
\label{fig:imp_subgraph2}
\end{figure}
\end{lemma}
\begin{proof}
It is convenient to denote the graph of Figure \ref{fig:imp_subgraph2} by $\tilde{\Gamma}$.

Suppose there is a chamber $K$ such that $\tilde{\Gamma}$ is a subgraph of $\Gamma(K)$
and let $K'$ be an adjacent chamber.
By Lemma \ref{no chordal circles} the graph $\Gamma(K)$ cannot have a chord.
But then by Lemma \ref{graph adj chamber} the Coxeter graph $\Gamma(K')$ of the adjacent chamber also has a subgraph
of the form shown in Figure \ref{fig:imp_subgraph2} and hence, disregarding the labels, $\Gamma(K')$ is the same graph as
$\Gamma(K)$.
Thus by induction for all chambers $K \in \Kc(\Ac)$ the graph $\tilde{\Gamma}$ is a subgraph of $\Gamma(K)$.
Now let $X \in L_3(\Ac)$ and $K \in \Kc(\Ac)$ be some chamber adjacent to $X$, i.e.\ $X \in L_3(W^K)$.
Then by Lemma \ref{loc subgraph} the Coxeter graph $\Gamma(K_X)$ for a chamber $K_X\in \Kc(\Ac_X/X)$ contains an
induced subgraph on $3$ vertices of $\tilde{\Gamma}$ and thus is connected. So $\Ac_X$ is irreducible for all $X \in L_3(\Ac)$ by Lemma \ref{simpl irred graph}.
This is a contradiction to Theorem \ref{gen Silvester}.
\end{proof}

To give a complete list of all possible Coxeter graphs of \isss $4$-arrangements we need the explicit description
of the change of Coxeter graphs for adjacent chambers in the three possible irreducible rank $3$ localizations
given by the next lemma.

\begin{lemma}\label{diagram graph change}
Let $\Ac$ be one of the arrangements $\Ac(A_3)$, $\Ac(B_3)$ or $\Ac_3^2$. Then Figure \ref{fig:change graphs}
gives a complete description of the change of the Coxeter graphs for adjacent chambers where an arrow labeled with $\sigma_i$
means crossing the $i$-th wall corresponding to the $i$-th vertex of the Coxeter graph.
\begin{figure}
\begin{tikzpicture}
\node at (-1.8,0) {$\Ac(6,1) \cong \Ac(A_3)$:};
\filldraw [black]
	(0,0) circle [radius=2pt]
	(1,0) circle [radius=2pt]
	(2,0) circle [radius=2pt];
\draw (0,0) -- (1,0) -- (2,0);
\node[below] at (0,0)  {$1$};
\node[below] at (1,0)  {$2$};
\node[below] at (2,0)  {$3$};

\draw [<->] (2.5,0.125) arc [start angle=140, end angle=-140, radius=0.25] ;
\node[right] at (2.9,0) {$\sigma_1, \sigma_2, \sigma_3$};

\end{tikzpicture}
\begin{tikzpicture}
\node at (-1.8,0) {$\Ac(9,1) \cong \Ac(B_3)$:};
\filldraw [black]
	(0,0) circle [radius=2pt]
	(1,0) circle [radius=2pt]
	(2,0) circle [radius=2pt];
\draw (0,0) -- (1,0) -- (2,0);
\node[below] at (0,0)  {$1$};
\node[below] at (1,0)  {$2$};
\node[below] at (2,0)  {$3$};

\node[above] at (1.5,0) {$4$};

\draw [<->] (2.5,0.125) arc [start angle=140, end angle=-140, radius=0.25] ;
\node[right] at (2.9,0) {$\sigma_1, \sigma_2, \sigma_3$};

\end{tikzpicture}

\begin{tikzpicture}
\node at (1.6,0.8) {$\Ac(8,1) \cong \Ac_3^2$:};
\filldraw [black]
	(0,0) circle [radius=2pt]
	(1,0) circle [radius=2pt]
	(2,0) circle [radius=2pt];
\draw (0,0) -- (1,0) -- (2,0);
\node[below] at (0,0)  {$1$};
\node[below] at (1,0)  {$2$};
\node[below] at (2,0)  {$3$};
\node[above] at (1.5,0) {$4$};

\draw [<->] (1.2,-0.5) arc [start angle=40, end angle=-220, radius=0.25] ;
\node[right] at (1.2,-0.75) {$\sigma_2, \sigma_3$};

\node (a) at (2.35,0) {};
\draw [<->] ($(a)+(0,0)$) -- ($(a)+(1,0)$);
\node[above] at ($(a)+(0.5,0)$) {$\sigma_1$};
\end{tikzpicture}
\begin{tikzpicture}
\filldraw [black]
	(0,0) circle [radius=2pt]
	(1,0) circle [radius=2pt]
	(2,0) circle [radius=2pt];
\draw (0,0) -- (1,0) -- (2,0);
\node[below] at (0,0)  {$1$};
\node[below] at (1,0)  {$2$};
\node[below] at (2,0)  {$3$};

\draw [<->] (1.2,-0.5) arc [start angle=40, end angle=-220, radius=0.25] ;
\node[right] at (1.2,-0.75) {$\sigma_3$};

\node (a) at (2.35,0) {};
\draw [<->] ($(a)+(0,0)$) -- ($(a)+(1,0)$);
\node[above] at ($(a)+(0.5,0)$) {$\sigma_2$};
\end{tikzpicture}
\begin{tikzpicture}
\filldraw [black]
	(0,-0.5) circle [radius=2pt]
	(0,0.5) circle [radius=2pt]
	(1,0) circle [radius=2pt];
\draw (0,-0.5) -- (0,0.5) -- (1,0) -- (0,-0.5);
\node[below] at (0,-0.5)  {$1$};
\node[above] at (0,0.5)  {$2$};
\node[below] at (1,0)  {$3$};

\draw [<->] (0.8,-0.5) arc [start angle=60, end angle=-200, radius=0.25] ;
\node[right] at (0.9,-0.75) {$\sigma_1$};
\end{tikzpicture}

\begin{tikzpicture}

\node (a) at (-1.35,0) {};
\draw [<->] ($(a)+(0,0)$) -- ($(a)+(1,0)$);
\node[above] at ($(a)+(0.5,0)$) {$\sigma_3$};

\filldraw [black]
	(0,0) circle [radius=2pt]
	(1,0) circle [radius=2pt]
	(2,0) circle [radius=2pt];
\draw (0,0) -- (1,0) -- (2,0);
\node[below] at (0,0)  {$1$};
\node[below] at (1,0)  {$3$};
\node[below] at (2,0)  {$2$};

\draw [<->] (1.2,-0.5) arc [start angle=40, end angle=-220, radius=0.25] ;
\node[right] at (1.2,-0.75) {$\sigma_2$};

\node (a) at (2.35,0) {};
\draw [<->] ($(a)+(0,0)$) -- ($(a)+(1,0)$);
\node[above] at ($(a)+(0.5,0)$) {$\sigma_1$};
\end{tikzpicture}
\begin{tikzpicture}
\filldraw [black]
	(0,0) circle [radius=2pt]
	(1,0) circle [radius=2pt]
	(2,0) circle [radius=2pt];
\draw (0,0) -- (1,0) -- (2,0);
\node[below] at (0,0)  {$1$};
\node[below] at (1,0)  {$3$};
\node[below] at (2,0)  {$2$};
\node[above] at (1.5,0) {$4$};

\draw [<->] (1.2,-0.5) arc [start angle=40, end angle=-220, radius=0.25] ;
\node[right] at (1.2,-0.75) {$\sigma_3, \sigma_2$};
\end{tikzpicture}

\caption{Diagrams for the change of Coxeter graphs of adjacent chambers in $\Ac(A_3), \Ac(B_3)$, and $\Ac_3^2$ respectively.}
\label{fig:change graphs}
\end{figure}
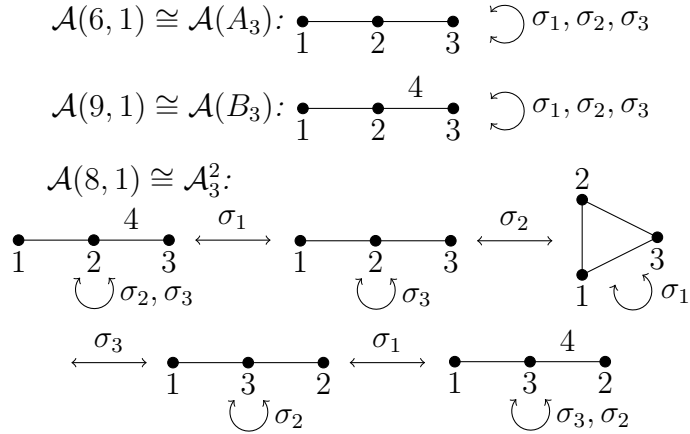
\end{lemma}
\begin{proof}
The diagrams for $\Ac(A_3)$ and $\Ac(B_3)$ are obvious since they are reflection arrangements and hence for all
chambers the Coxeter graph is the same.

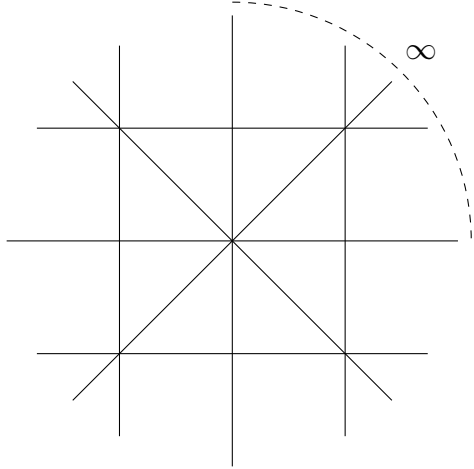
\begin{figure}
\def \sc {0.75}
\begin{tikzpicture}[scale=\sc]
\draw (0.,4.) -- (0.,-4.);
\draw (4.,0.) -- (-4.,0.);
\draw [dashed] (0,4.2352941176470589) arc [start angle=90, end angle=0, radius=4.2352941176470589] ;
\draw (3.344805190907731,3.344805190907731) node {$\infty$};
\draw (-2.8284271247461898,2.8284271247461898) -- (2.8284271247461898,-2.8284271247461898);
\draw (2.8284271247461898,2.8284271247461898) -- (-2.8284271247461898,-2.8284271247461898);
\draw (-2.,3.4641016151377544) -- (-2.,-3.4641016151377544);
\draw (2.,3.4641016151377544) -- (2.,-3.4641016151377544);
\draw (3.4641016151377544,-2.) -- (-3.4641016151377544,-2.);
\draw (3.4641016151377544,2.) -- (-3.4641016151377544,2.);
\end{tikzpicture}
\caption{The arrangements $\Ac(8,1)$ and $\Ac(9,1)$}%
\label{fig:A8_1}
\end{figure}

The third diagram can be seen by looking at a projective picture of the arrangement (see Figure \ref{fig:A8_1}) or as
a special case of \cite[Prop.~3.8]{MR3341467}.
\end{proof}

\begin{lemma}\label{no Type F4 D4_4}
Let $\Ac$ be an irreducible \sss $4$-ar\-rangement and let $K \in \Kc(\Ac)$ be a chamber.
Then $\Gamma(K)$ is not one of the graphs of Figure \ref{fig:imp_graph3}.
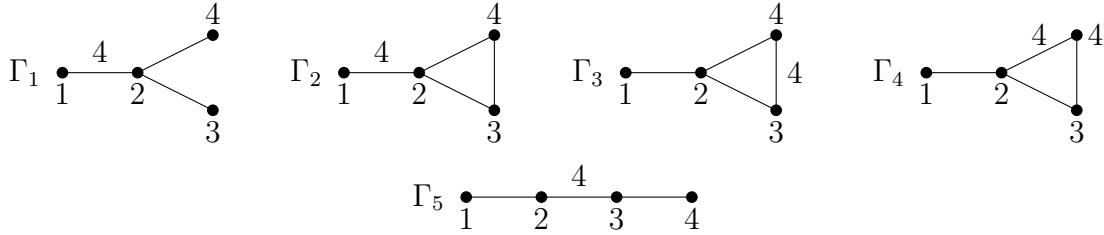
\begin{figure}
\begin{tikzpicture}
\node at (-0.5,0)  {$\Gamma_1$};
\filldraw [black]
	(2,-0.5) circle [radius=2pt]
	(2,0.5) circle [radius=2pt]
	(0,0) circle [radius=2pt]
	(1,0) circle [radius=2pt];
\draw
	(0,0) -- (1,0) -- (2,-0.5)
	(1,0) -- (2,0.5);
\node[above] at (0.5,0)  {$4$};
\node[below] at (0,0)  {$1$};
\node[below] at (1,0)  {$2$};
\node[below] at (2,-0.5)  {$3$};
\node[above] at (2,0.5)  {$4$};
\end{tikzpicture}\hspace{0.5cm}
\begin{tikzpicture}
\node at (-0.5,0)  {$\Gamma_2$};
\filldraw [black]
	(2,-0.5) circle [radius=2pt]
	(2,0.5) circle [radius=2pt]
	(0,0) circle [radius=2pt]
	(1,0) circle [radius=2pt];
\draw
	(0,0) -- (1,0) -- (2,-0.5)
	(1,0) -- (2,0.5) -- (2,-0.5);
\node[above] at (0.5,0)  {$4$};
\node[below] at (0,0)  {$1$};
\node[below] at (1,0)  {$2$};
\node[below] at (2,-0.5)  {$3$};
\node[above] at (2,0.5)  {$4$};
\end{tikzpicture}\hspace{0.5cm}
\begin{tikzpicture}
\node at (-0.5,0)  {$\Gamma_3$};
\filldraw [black]
	(2,-0.5) circle [radius=2pt]
	(2,0.5) circle [radius=2pt]
	(0,0) circle [radius=2pt]
	(1,0) circle [radius=2pt];
\draw
	(0,0) -- (1,0) -- (2,-0.5)
	(1,0) -- (2,0.5) -- (2,-0.5);
\node[right] at (2,0)  {$4$};
\node[below] at (0,0)  {$1$};
\node[below] at (1,0)  {$2$};
\node[below] at (2,-0.5)  {$3$};
\node[above] at (2,0.5)  {$4$};
\end{tikzpicture}\hspace{0.5cm}
\begin{tikzpicture}
\node at (-0.5,0)  {$\Gamma_4$};
\filldraw [black]
	(2,-0.5) circle [radius=2pt]
	(2,0.5) circle [radius=2pt]
	(0,0) circle [radius=2pt]
	(1,0) circle [radius=2pt];
\draw
	(0,0) -- (1,0) -- (2,-0.5)
	(1,0) -- (2,0.5) -- (2,-0.5);
\node at (1.5,0.5)  {$4$};
\node[below] at (0,0)  {$1$};
\node[below] at (1,0)  {$2$};
\node[below] at (2,-0.5)  {$3$};
\node[right] at (2,0.5)  {$4$};
\end{tikzpicture}\hspace{0.5cm}
\begin{tikzpicture}
\node at (-0.5,0)  {$\Gamma_5$};
\filldraw [black]
	(0,0) circle [radius=2pt]
	(1,0) circle [radius=2pt]
	(2,0) circle [radius=2pt]
	(3,0) circle [radius=2pt];
\draw
	(0,0) -- (1,0) -- (2,0) -- (3,0);
\node[above] at (1.5,0)  {$4$};
\node[below] at (0,0)  {$1$};
\node[below] at (1,0)  {$2$};
\node[below] at (2,0)  {$3$};
\node[below] at (3,0)  {$4$};
\end{tikzpicture}

\caption{Impossible Coxeter graphs for an \isss $4$-arrangement.}%
\label{fig:imp_graph3}
\end{figure}
\end{lemma}
\begin{proof}
Let $B^K = \{\alpha_1,\ldots,\alpha_4\}$ be a basis for $K$.

First suppose that there is a $K \in \Kc(\Ac)$ such that $\Gamma(K) = \Gamma_1$.
By Lemma \ref{SSS loc res}(1) the arrangements $\Ac^K_{123} := \Ac_X/X$ with
$X = \alpha_1^\perp \cap \alpha_2^\perp \cap \alpha_3^\perp$ and  $\Ac^K_{124} := \Ac_{X'}/X'$ with
$X' = \alpha_1^\perp \cap \alpha_2^\perp \cap \alpha_4^\perp$ are supersolvable and simplicial.
By Lemma \ref{loc subgraph} both localizations contain a chamber with Coxeter graph
the induced subgraph on the vertices $\{1,2,3\}$ or $\{1,2,4\}$ of $\Gamma_1$.
Hence by Lemma \ref{simpl irred graph} both localizations are irreducible and by Corollary \ref{Coxeter graph rk3}
they are either lattice equivalent to $\Ac(8,1)$ or $\Ac(9,1)$.
Since $\vert \Ac^K_{23} \vert = \vert \Ac^K_{24} \vert = 3$ by Lemma \ref{X mod max} the intersection
$Y = \alpha_1^\perp \cap \alpha_2^\perp$ is modular in $\Ac_X$ and $\Ac_{X'}$.
Let $H=\sigma^K_2(\alpha_1) \in \Ac_Y$ then by Lemma \ref{res subgraph} the Coxeter graph of $K_2^H \in \Kc(\Ac^H)$
contains a subgraph of the form shown in Figure \ref{fig:subgraph3}.
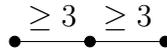
\begin{figure}
\begin{tikzpicture}
\filldraw [black]
	(-1,0) circle [radius=2pt]
	(0,0) circle [radius=2pt]
	(1,0) circle [radius=2pt];
\draw
	(-1,0) -- (0,0) -- (1,0);

\node[above] at (-0.5,0)  {$\geq 3$};
\node[above] at (0.5,0)  {$\geq 3$};
\end{tikzpicture}
\caption{Subgraph of a chamber in $\Ac^H$.}%
\label{fig:subgraph3}
\end{figure}

For the arrangements $\Ac(8,1)$ and $\Ac(9,1)$ in both cases we have
$\vert \Ac_X^H \vert = \vert \Ac_{X'}^H \vert = 4$.
So actually both labels of the Coxeter subgraph are equal to $4$ and $\Gamma(K_2^H)$
contains a subgraph as in Figure \ref{fig:contr_imp_subgraph1}.
This is a contradiction to Corollary \ref{Coxeter graph rk3} and we can exclude the graph $\Gamma_1$ from the list
of possible Coxeter graphs of \isss $4$-arrangements.

Secondly suppose that $\Gamma(K)=\Gamma_2$. Then by Lemma \ref{res subgraph} there is a hyperplane $H \in \Ac^K_{23}$
and a chamber $K^H \in \Kc(\Ac^H)$ such that the graph shown in Figure \ref{fig:contr_imp_subgraph1}
is a subgraph of $\Gamma(K^H)$ contradicting Corollary \ref{Coxeter graph rk3} again.

For the graphs $\Gamma_3$ and $\Gamma_4$ the localization $\Ac^K_{234}$ is an
\isss $3$-arrangement. By Theorem \ref{Thm:at most 4er} it has rank $2$ localizations of size at most $4$.
By Lemma \ref{loc subgraph} there is a chamber in $\Ac^K_{234}$ with Coxeter graph the induced subgraph
on the vertices $\{2,3,4\}$. But this a contradiction to Corollary \ref{Coxeter graph rk3}.

Finally suppose that there is a $K \in \Kc(\Ac)$ such that $\Gamma(K) = \Gamma_5$ and let $B^K = \{\alpha_1,\ldots,\alpha_4\}$.
Let $X = \alpha_1^\perp \cap \alpha_2^\perp \cap \alpha_3^\perp$, $X' = \alpha_2^\perp \cap \alpha_3^\perp \cap \alpha_4^\perp$, $\Ac^K_{123} = \Ac_X /X$ and
$\Ac^K_{234} = \Ac_{X'}/X'$.
By Lemma \ref{SSS loc res}(1) these arrangements are supersolvable and simplicial, and
by Lemma \ref{loc subgraph}, Lemma \ref{simpl irred graph} and Corollary \ref{Coxeter graph rk3}
the two arrangements are either $\Ac(8,1)$ or $\Ac(9,1)$.
If both arrangements are $\Ac(9,1)$ then for all $H \in \Ac_Y$ with $Y = \alpha_2^\perp \cap \alpha_3^\perp$
we have $\vert \Ac_X^H \vert = \vert \Ac_{X'}^H \vert = 4$ (see Figure \ref{fig:A8_1})
and similarly to the first part of this proof we can find an $H'$ and a $K'^{H'} \in \Kc(\Ac^{H'})$
which contains the forbidden Coxeter subgraph of Figure \ref{fig:contr_imp_subgraph1}.
So assume without loss of generality that $\Ac_{123}$ is equal to $\Ac(8,1)$.
Using Lemma \ref{loc subgraph}, Lemma \ref{graph adj chamber} and Lemma \ref{diagram graph change} we obtain one of the sequences of Coxeter graphs
for the corresponding sequence of chambers (A)--(D) of Figure \ref{Fig_chamber_sequences} (depending on $\Ac^{K_1}_{234}$).
But in each sequence the last graph is (up to renumbering the vertices) one which we already excluded above.
E.g.\ the last graph in sequence (A) is the graph $\Gamma_1$ which we already excluded.
Hence $\Gamma_5$ is not the Coxeter graph of a chamber of an \isss $4$-arrangement.

\begin{figure}

\subcaptionbox{$\Ac^{K_1}_{234} \sim_L \Ac(6,1)$}[0.75\textwidth]{
\begin{tikzpicture}
\filldraw [black]
	(0,0) circle [radius=2pt]
	(1,0) circle [radius=2pt]
	(2,0) circle [radius=2pt]
	(3,0) circle [radius=2pt];
\draw
	(0,0) -- (1,0) -- (2,0) -- (3,0);
\draw (0,0) circle [radius=3.25pt];
\node[above] at (1.5,0)  {$4$};
\node[below] at (0,0)  {$1$};
\node[below] at (1,0)  {$2$};
\node[below] at (2,0)  {$3$};
\node[below] at (3,0)  {$4$};

\node (a) at (3.35,0) {};
\draw [->] ($(a)+(0,0)$) -- ($(a)+(1,0)$);
\node[above] at ($(a)+(0.5,0)$) {$\sigma_1$};

\end{tikzpicture}
\begin{tikzpicture}
\filldraw [black]
	(0,0) circle [radius=2pt]
	(1,0) circle [radius=2pt]
	(2,0) circle [radius=2pt]
	(3,0) circle [radius=2pt];
\draw (1,0) circle [radius=3.25pt];
\draw
	(0,0) -- (1,0) -- (2,0) -- (3,0);
\node[below] at (0,0)  {$1$};
\node[below] at (1,0)  {$2$};
\node[below] at (2,0)  {$3$};
\node[below] at (3,0)  {$4$};

\node (a) at (3.35,0) {};
\draw [->] ($(a)+(0,0)$) -- ($(a)+(1,0)$);
\node[above] at ($(a)+(0.5,0)$) {$\sigma_2$};
\end{tikzpicture}
\begin{tikzpicture}
\node (g1) at (0,0) {};


\node (a1) at ($(g1)+(0,-0.5)$)  {};
\node[below] at (a1) {$1$};
\node (a2) at ($(g1)+(0,0.5)$)  {};
\node[above] at (a2)  {$2$};
\node (a3) at ($(g1)+(1,0)$)  {};
\node[below] at (a3)  {$3$};
\node (a4) at ($(g1)+(2,0)$)  {};
\node[below] at (a4)  {$4$};

\filldraw [black]
	(a1) circle [radius=2pt]
	(a2) circle [radius=2pt]
	(a3) circle [radius=2pt]
	(a4) circle [radius=2pt];
\draw (a3) circle [radius=3.25pt];
\draw
	(a1) -- (a2) -- (a3) -- (a1)
	(a3) -- (a4);
\end{tikzpicture}

\begin{tikzpicture}
\node (a) at (-1.35,0) {};
\draw [->] ($(a)+(0,0)$) -- ($(a)+(1,0)$);
\node[above] at ($(a)+(0.5,0)$) {$\sigma_3$};

\filldraw [black]
	(0,-0.5) circle [radius=2pt]
	(0,0.5) circle [radius=2pt]
	(1,0) circle [radius=2pt]
	(2,0) circle [radius=2pt];
\draw (0,-0.5) circle [radius=3.25pt];
\draw
	(0,0.5) -- (1,0) -- (0,-0.5)
	(1,0) -- (2,0);
\node[below] at (0,-0.5)  {$1$};
\node[above] at (0,0.5)  {$2$};
\node[below] at (1,0)  {$3$};
\node[below] at (2,0)  {$4$};
\end{tikzpicture}
\begin{tikzpicture}
\node (a) at (-1.35,0) {};
\draw [->] ($(a)+(0,0)$) -- ($(a)+(1,0)$);
\node[above] at ($(a)+(0.5,0)$) {$\sigma_1$};

\filldraw [black]
	(0,-0.5) circle [radius=2pt]
	(0,0.5) circle [radius=2pt]
	(1,0) circle [radius=2pt]
	(2,0) circle [radius=2pt];
\draw
	(0,0.5) -- (1,0) -- (0,-0.5)
	(1,0) -- (2,0);
\node[below] at (0,-0.5)  {$1$};
\node[above] at (0,0.5)  {$2$};
\node[below] at (1,0)  {$3$};
\node[below] at (2,0)  {$4$};

\node at (0.6,0.5) {$4$};
\end{tikzpicture}
}

\subcaptionbox{$\Ac^{K_1}_{234} \sim_L \Ac(8,1)$}[0.75\textwidth]{
\begin{tikzpicture}
\filldraw [black]
	(0,0) circle [radius=2pt]
	(1,0) circle [radius=2pt]
	(2,0) circle [radius=2pt]
	(3,0) circle [radius=2pt];
\draw
	(0,0) -- (1,0) -- (2,0) -- (3,0);
\draw (0,0) circle [radius=3.25pt];
\node[above] at (1.5,0)  {$4$};
\node[below] at (0,0)  {$1$};
\node[below] at (1,0)  {$2$};
\node[below] at (2,0)  {$3$};
\node[below] at (3,0)  {$4$};

\node (a) at (3.35,0) {};
\draw [->] ($(a)+(0,0)$) -- ($(a)+(1,0)$);
\node[above] at ($(a)+(0.5,0)$) {$\sigma_1$};

\end{tikzpicture}
\begin{tikzpicture}
\filldraw [black]
	(0,0) circle [radius=2pt]
	(1,0) circle [radius=2pt]
	(2,0) circle [radius=2pt]
	(3,0) circle [radius=2pt];
\draw (1,0) circle [radius=3.25pt];
\draw
	(0,0) -- (1,0) -- (2,0) -- (3,0);
\node[below] at (0,0)  {$1$};
\node[below] at (1,0)  {$2$};
\node[below] at (2,0)  {$3$};
\node[below] at (3,0)  {$4$};

\node (a) at (3.35,0) {};
\draw [->] ($(a)+(0,0)$) -- ($(a)+(1,0)$);
\node[above] at ($(a)+(0.5,0)$) {$\sigma_2$};
\end{tikzpicture}
\begin{tikzpicture}
\node (g1) at (0,0) {};


\node (a1) at ($(g1)+(0,-0.5)$)  {};
\node[below] at (a1) {$1$};
\node (a2) at ($(g1)+(0,0.5)$)  {};
\node[above] at (a2)  {$2$};
\node (a3) at ($(g1)+(1,0)$)  {};
\node[below] at (a3)  {$3$};
\node (a4) at ($(g1)+(2,0)$)  {};
\node[below] at (a4)  {$4$};

\filldraw [black]
	(a1) circle [radius=2pt]
	(a2) circle [radius=2pt]
	(a3) circle [radius=2pt]
	(a4) circle [radius=2pt];
\draw (a3) circle [radius=3.25pt];
\draw
	(a1) -- (a2) -- (a3) -- (a1)
	(a3) -- (a4);
\end{tikzpicture}
\begin{tikzpicture}
\node (a) at (-1.35,0) {};
\draw [->] ($(a)+(0,0)$) -- ($(a)+(1,0)$);
\node[above] at ($(a)+(0.5,0)$) {$\sigma_3$};

\node (g1) at (0,0) {};


\node (a1) at ($(g1)+(0,0)$)  {};
\node[below] at (a1) {$1$};
\node (a2) at ($(g1)+(2,0.5)$)  {};
\node[above] at (a2)  {$2$};
\node (a3) at ($(g1)+(1,0)$)  {};
\node[below] at (a3)  {$3$};
\node (a4) at ($(g1)+(2,-0.5)$)  {};
\node[below] at (a4)  {$4$};

\filldraw [black]
	(a1) circle [radius=2pt]
	(a2) circle [radius=2pt]
	(a3) circle [radius=2pt]
	(a4) circle [radius=2pt];
\draw (a1) circle [radius=3.25pt];
\draw
	(a1) -- (a3) -- (a4) -- (a2) --
	(a3);
\end{tikzpicture}
\begin{tikzpicture}
\node (a) at (-1.35,0) {};
\draw [->] ($(a)+(0,0)$) -- ($(a)+(1,0)$);
\node[above] at ($(a)+(0.5,0)$) {$\sigma_1$};

\node (g1) at (0,0) {};


\node (a1) at ($(g1)+(0,0)$)  {};
\node[below] at (a1) {$1$};
\node (a2) at ($(g1)+(2,0.5)$)  {};
\node[above] at (a2)  {$2$};
\node (a3) at ($(g1)+(1,0)$)  {};
\node[below] at (a3)  {$3$};
\node (a4) at ($(g1)+(2,-0.5)$)  {};
\node[below] at (a4)  {$4$};

\filldraw [black]
	(a1) circle [radius=2pt]
	(a2) circle [radius=2pt]
	(a3) circle [radius=2pt]
	(a4) circle [radius=2pt];
\draw
	(a1) -- (a3) -- (a4) -- (a2) --
	(a3);
\node at ($0.5*(a1)+0.5*(a3)+(0,0.25)$) {$4$};
\end{tikzpicture}
}

\subcaptionbox{$\Ac^{K_1}_{234} \sim_L \Ac(9,1)$}[0.75\textwidth]{
\begin{tikzpicture}
\filldraw [black]
	(0,0) circle [radius=2pt]
	(1,0) circle [radius=2pt]
	(2,0) circle [radius=2pt]
	(3,0) circle [radius=2pt];
\draw
	(0,0) -- (1,0) -- (2,0) -- (3,0);
\draw (0,0) circle [radius=3.25pt];
\node[above] at (1.5,0)  {$4$};
\node[below] at (0,0)  {$1$};
\node[below] at (1,0)  {$2$};
\node[below] at (2,0)  {$3$};
\node[below] at (3,0)  {$4$};

\node (a) at (3.35,0) {};
\draw [->] ($(a)+(0,0)$) -- ($(a)+(1,0)$);
\node[above] at ($(a)+(0.5,0)$) {$\sigma_1$};

\end{tikzpicture}
\begin{tikzpicture}
\filldraw [black]
	(0,0) circle [radius=2pt]
	(1,0) circle [radius=2pt]
	(2,0) circle [radius=2pt]
	(3,0) circle [radius=2pt];
\draw (1,0) circle [radius=3.25pt];
\draw
	(0,0) -- (1,0) -- (2,0) -- (3,0);
\node[below] at (0,0)  {$1$};
\node[below] at (1,0)  {$2$};
\node[below] at (2,0)  {$3$};
\node[below] at (3,0)  {$4$};

\node (a) at (3.35,0) {};
\draw [->] ($(a)+(0,0)$) -- ($(a)+(1,0)$);
\node[above] at ($(a)+(0.5,0)$) {$\sigma_2$};
\end{tikzpicture}
\begin{tikzpicture}
\node (g1) at (0,0) {};

\node (a1) at ($(g1)+(0,-0.5)$)  {};
\node[below] at (a1) {$1$};
\node (a2) at ($(g1)+(0,0.5)$)  {};
\node[above] at (a2)  {$2$};
\node (a3) at ($(g1)+(1,0)$)  {};
\node[below] at (a3)  {$3$};
\node (a4) at ($(g1)+(2,0)$)  {};
\node[below] at (a4)  {$4$};
\node at ($0.5*(a2)+0.5*(a3)+(0.15,0.15)$) {$4$};

\filldraw [black]
	(a1) circle [radius=2pt]
	(a2) circle [radius=2pt]
	(a3) circle [radius=2pt]
	(a4) circle [radius=2pt];
\draw
	(a1) -- (a2) -- (a3) -- (a1)
	(a3) -- (a4);
\end{tikzpicture}
}

\subcaptionbox{$\Ac^{K_1}_{234} \sim_L \Ac(8,1),\Ac(9,1)$}[0.75\textwidth]{
\begin{tikzpicture}
\filldraw [black]
	(0,0) circle [radius=2pt]
	(1,0) circle [radius=2pt]
	(2,0) circle [radius=2pt]
	(3,0) circle [radius=2pt];
\draw
	(0,0) -- (1,0) -- (2,0) -- (3,0);
\draw (0,0) circle [radius=3.25pt];
\node[above] at (1.5,0)  {$4$};
\node[below] at (0,0)  {$1$};
\node[below] at (1,0)  {$2$};
\node[below] at (2,0)  {$3$};
\node[below] at (3,0)  {$4$};

\node (a) at (3.35,0) {};
\draw [->] ($(a)+(0,0)$) -- ($(a)+(1,0)$);
\node[above] at ($(a)+(0.5,0)$) {$\sigma_1$};

\end{tikzpicture}
\begin{tikzpicture}
\filldraw [black]
	(0,0) circle [radius=2pt]
	(1,0) circle [radius=2pt]
	(2,0) circle [radius=2pt]
	(3,0) circle [radius=2pt];
\draw (1,0) circle [radius=3.25pt];
\draw
	(0,0) -- (1,0) -- (2,0) -- (3,0);
\node[below] at (0,0)  {$1$};
\node[below] at (1,0)  {$2$};
\node[below] at (2,0)  {$3$};
\node[below] at (3,0)  {$4$};

\node (a) at (3.35,0) {};
\draw [->] ($(a)+(0,0)$) -- ($(a)+(1,0)$);
\node[above] at ($(a)+(0.5,0)$) {$\sigma_2$};
\end{tikzpicture}
\begin{tikzpicture}
\node (g1) at (0,0) {};

\node (a1) at ($(g1)+(0,-0.5)$)  {};
\node[below] at (a1) {$1$};
\node (a2) at ($(g1)+(0,0.5)$)  {};
\node[above] at (a2)  {$2$};
\node (a3) at ($(g1)+(1,0)$)  {};
\node[below] at (a3)  {$3$};
\node (a4) at ($(g1)+(2,0)$)  {};
\node[below] at (a4)  {$4$};
\node at ($0.5*(a3)+0.5*(a4)+(0,0.25)$) {$4$};

\filldraw [black]
	(a1) circle [radius=2pt]
	(a2) circle [radius=2pt]
	(a3) circle [radius=2pt]
	(a4) circle [radius=2pt];
\draw
	(a1) -- (a2) -- (a3) -- (a1)
	(a3) -- (a4);
\end{tikzpicture}
}

\caption{Sequences of graphs of chambers in $\Ac$ starting at $K$ and leading to a contradiction.}%
\label{Fig_chamber_sequences}
\end{figure}
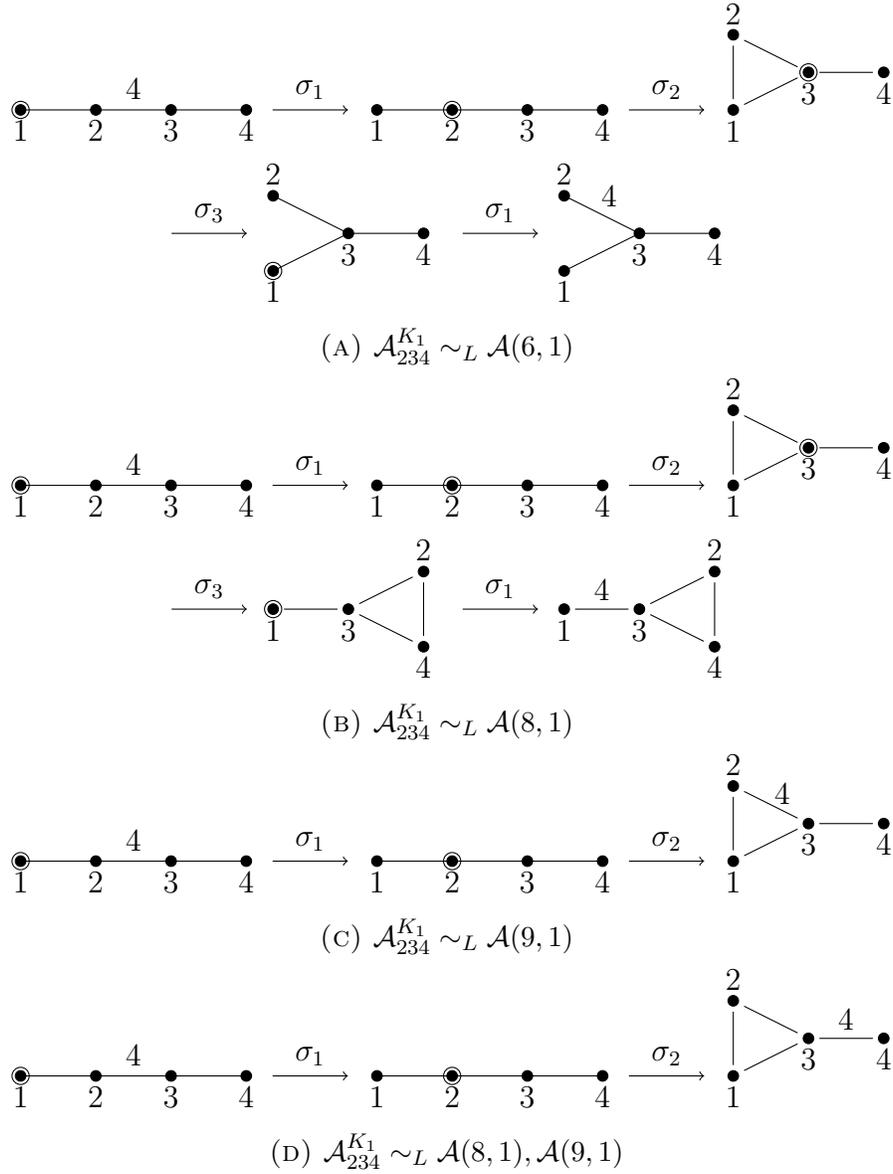
\end{proof}

\begin{propo}\label{Coxeter graphs rk4}
Let $\Ac$ be an \isss $4$-arrange\-ment and  $K \in \Kc(\Ac)$.
Then $\Gamma(K)$ is one of the Coxeter graphs displayed in Figure \ref{fig:Coxeter graphs rk4}.
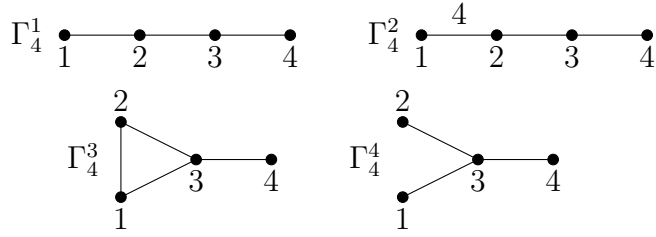
\begin{figure}
\begin{tikzpicture}
\node at (-0.5,0)  {$\Gamma_4^1$};
\filldraw [black]
	(0,0) circle [radius=2pt]
	(1,0) circle [radius=2pt]
	(2,0) circle [radius=2pt]
	(3,0) circle [radius=2pt];
\draw
	(0,0) -- (1,0) -- (2,0) -- (3,0);
\node[below] at (0,0)  {$1$};
\node[below] at (1,0)  {$2$};
\node[below] at (2,0)  {$3$};
\node[below] at (3,0)  {$4$};
\end{tikzpicture}\hspace{0.5cm}
\begin{tikzpicture}
\node at (-0.5,0)  {$\Gamma_4^2$};
\filldraw [black]
	(0,0) circle [radius=2pt]
	(1,0) circle [radius=2pt]
	(2,0) circle [radius=2pt]
	(3,0) circle [radius=2pt];
\draw
	(0,0) -- (1,0) -- (2,0) -- (3,0);
\node[above] at (0.5,0)  {$4$};
\node[below] at (0,0)  {$1$};
\node[below] at (1,0)  {$2$};
\node[below] at (2,0)  {$3$};
\node[below] at (3,0)  {$4$};
\end{tikzpicture}

\begin{tikzpicture}
\node at (-0.5,0)  {$\Gamma_4^3$};
\filldraw [black]
	(0,-0.5) circle [radius=2pt]
	(0,0.5) circle [radius=2pt]
	(1,0) circle [radius=2pt]
	(2,0) circle [radius=2pt];
\draw
	(0,-0.5) -- (0,0.5) -- (1,0) -- (0,-0.5)
	(1,0) -- (2,0);
\node[below] at (0,-0.5)  {$1$};
\node[above] at (0,0.5)  {$2$};
\node[below] at (1,0)  {$3$};
\node[below] at (2,0)  {$4$};
\end{tikzpicture}\hspace{0.5cm}
\begin{tikzpicture}
\node at (-0.5,0)  {$\Gamma_4^4$};
\filldraw [black]
	(0,-0.5) circle [radius=2pt]
	(0,0.5) circle [radius=2pt]
	(1,0) circle [radius=2pt]
	(2,0) circle [radius=2pt];
\draw
	(0,0.5) -- (1,0) -- (0,-0.5)
	(1,0) -- (2,0);
\node[below] at (0,-0.5)  {$1$};
\node[above] at (0,0.5)  {$2$};
\node[below] at (1,0)  {$3$};
\node[below] at (2,0)  {$4$};
\end{tikzpicture}\hspace{0.5cm}

\caption{Remaining Coxeter graphs for an \isss $4$-arrangement.}\label{fig:Coxeter graphs rk4}
\end{figure}
\end{propo}
\begin{proof}
By Lemma \ref{no big circles} no big cycles are possible and
by Proposition \ref{rk4 only 4er or smaller} all labels are at most $4$.
Furthermore with Lemma \ref{loc subgraph}, Lemma \ref{res subgraph} and Corollary \ref{Coxeter graph rk3}
we see that the graph cannot contain two edges labeled with $4$
by looking at the appropriate restriction respectively localization not fitting into the classification of rank $3$ arrangements and their Coxeter graphs (see Theorem \ref{classification rk 3} and Corollary \ref{Coxeter graph rk3}).
Now by Lemma \ref{no Type F4 D4_4} the only possible Coxeter graphs left
are the ones of Figure  \ref{fig:Coxeter graphs rk4}.
\end{proof}

\begin{propo}\label{Cartan matrices rk4}
Let $\Ac$ be an \isss $4$-ar\-rangement and $K \in \Kc(\Ac)$.
\begin{enumerate}
\item There exists a locally crystallographic basis $B^K$ for $K$ such that the Cartan matrix $C^K$ with respect to $B^K$
	is of type $A,C,D'$ or $D$.
\item If $B^K = \{\alpha_1,\ldots,\alpha_4\}$ is a locally crystallographic basis for $K$ with $C^K$ of type $A,C,D'$ or $D$, then for $1 \leq i \leq 4$ the basis
$B^{K_i} = \sigma_i^K(B^K) = \{ \alpha_j + c^K_{ij}\alpha_i \mid 1 \leq j \leq 4\}$
is a locally crystallographic basis with Cartan matrix $C^{K_i}$ of type $A,C,D'$ or $D$.
\end{enumerate}

\end{propo}
\begin{proof}
Part (1). By Proposition \ref{Coxeter graphs rk4} the Coxeter graph $\Gamma(K)$
is one of the graphs of Figure \ref{fig:Coxeter graphs rk4}.
Let $W^K = \{H_1,\ldots,H_4\}$, and $\Gamma(K) = (\Vg,\Eg)$ with numbering of the walls corresponding
to the numbering of the vertices of the graphs in Figure \ref{fig:Coxeter graphs rk4}.

Firstly suppose that $\Gamma(K) = \Gamma_4^1$.
By Lemma \ref{SSS loc res} and Lemma \ref{loc subgraph} the localization $\Ac^K_{123}$ adjacent to $K$ is an
\isss $3$-arrangement with a modular rank $2$ intersection of size at most $4$ by Theorem \ref{Thm:at most 4er}.
Hence by Theorem \ref{classification rk 3} and Corollary \ref{Coxeter graph rk3} it is the arrangement $\Ac(6,1)$ or $\Ac(8,1)$ and in particular crystallographic (see Example \ref{ex:Akl_cryst}).
By choosing a crystallographic root system for $\Ac^K_{123}$ and taking the corresponding basis for the chamber in the localization
by Example \ref{ex:Akl_cryst} there are $\alpha_1,\alpha_2,\alpha_3 \in (\RR^4)^*$ such that $\alpha_i^\perp = H_i$, $K \subseteq \alpha_i^+$,
$(\alpha_1 + \alpha_2)^\perp \in W^{K_1}, W^{K_2}$, and $(\alpha_2 + \alpha_3)^\perp \in W^{K_2}, W^{K_3}$.
Let $\tilde{\alpha_4} \in (\RR^4)^*$ such that $\tilde{\alpha_4}^\perp = H_4$ and $\tilde{\alpha_4}^+ \supseteq K$.
Since $\{3,4\} \in \Eg$ with label $m^K(3,4) = 3$ there is a unique $\lambda \in \RR_{>0}$ such that
$(\alpha_3 + \lambda\tilde{\alpha_4})^\perp \in W^{K_3}, W^{K_4}$.
But then with $\alpha_4 := \lambda\tilde{\alpha_4}$ we have $(\alpha_3 + \alpha_4)^\perp \in W^{K_3}, W^{K_4}$.
Hence $B^K := \{\alpha_1,\alpha_2,\alpha_3,\alpha_4\}$ is a locally crystallographic basis for $K$ with Cartan matrix $C^K = (c^K_{ij})$ of type $A$.

The same arguments work for the Coxeter graphs $\Gamma_4^3$ and $\Gamma_4^4$ since the vertex denoted as $4$
is only connected with the vertex $3$ and the localization $\Ac^K_{123}$ is $\Ac(6,1)$ or $\Ac(8,1)$ by Theorem \ref{classification rk 3}.
So similarly there is a locally crystallographic basis $B^K$ for $K$
such that the Cartan matrix is of type $D'$ if $\Gamma(K) =\Gamma_4^3$, or of type $D$ if $\Gamma(K) = \Gamma_4^4$.

Now assume that $\Gamma(K)  = \Gamma_4^2$. Then $\Ac^K_{123}$ is $\Ac(8,1) = \Ac_3^2$ or $\Ac(9,1) = \Ac(B_3)$.
Then there are $\alpha_1,\alpha_2,\alpha_3 \in (\RR^4)^*$ such that $\alpha_i^\perp = H_i$, $K \subseteq \alpha_i^+$,
$(2\alpha_1 + \alpha_2)^\perp \in W^{K_2}$, $(\alpha_1 + \alpha_2)^\perp \in W^{K_1}$, and $(\alpha_2+\alpha_3)^\perp \in W^{K_2},W^{K_3}$
(by choosing a proper crystallographic root system for the localization and taking the corresponding basis for the chamber in the localization).
Again it is clear that we can find an $\alpha_4 \in (\RR^4)^*$, $K \subseteq \alpha_4^+$ such that $(\alpha_3+\alpha_4)^\perp \in W^{K_3},W^{K_4}$
and hence $B^K := \{\alpha_1,\alpha_2,\alpha_3,\alpha_4\}$ is a locally crystallographic basis for $K$ with Cartan matrix $C^K = (c^K_{ij})$
of type $C$.

Part (2). For the second part we use Proposition \ref{Coxeter graphs rk4},
Lemma \ref{cij=cij'}, Lemma \ref{cij=0}, Lemma \ref{cij m=3},
Lemma \ref{label not connected}, and Lemma \ref{diagram graph change} to obtain the Coxeter graphs
for the adjacent chamber $K_i$ and deduce the claimed property of the induced basis $B^{K_i}$ and the coefficients $c^K_{ij}$:

We check the cases in turn.
First assume that $\Gamma(K) = \Gamma_4^1$, $B^K$ is locally crystallographic and $C^K$ is of type $A$.
As we have seen in the proof of Part (1), the localization $\Ac^K_{123}$ is the arrangement $\Ac(6,1)$ or $\Ac(8,1)$.

Let $i=1$.
By Proposition \ref{Coxeter graphs rk4} The Coxeter graph $\Gamma(K_1)$ is one of the four graphs of Figure \ref{fig:Coxeter graphs rk4} and
by Lemma \ref{cij=0} and Lemma \ref{label not connected} only $\Gamma_4^1$ is possible. Thus $\Gamma(K_1) = \Gamma(K)$
and by Lemma \ref{cij=cij'}, Lemma \ref{cij=0}, Lemma \ref{cij m=3}, and Lemma \ref{label not connected}
the basis $B^{K_1}$ induced by $C^K$ and $B^K$ is locally crystallographic with Cartan matrix $C^{K_1}=C^K$ of type $A$.
Note that from this case the statement also follows for $\Gamma(K)=\Gamma_4^1$ and $i=4$ by symmetry.

Next let $i=2$.
If the localizations $\Ac^K_{123}$ and $\Ac^K_{234}$ are both isomorphic to $\Ac(6,1)$ then using the same lemmas from
Subsection \ref{ssec:simplArr} as above,
the basis $B^{K_2}$ defined by $C^K$ is locally crystallographic with Cartan matrix $C^{K_2} = C^K$ of type $A$.
If $\Ac^K_{123}$ is the arrangement $\Ac(8,1)$ then $\Ac^K_{234}$ has to be the arrangement $\Ac(6,1)$ and $\Gamma(K_2) = \Gamma_4^3$.
Otherwise by Lemma \ref{diagram graph change} we would get a forbidden Coxeter graph of Figure \ref{fig:imp_graph3} for $K_2$.
With the lemmas from Subsection \ref{ssec:simplArr} and Section \ref{sec:Coxeter graphs}
we again obtain all coefficients $c^{K_2}_{ij}$ except the ones with $\{i,j\} = \{1,3\}$.
But $\Ac^K_{123} = \Ac^{K_2}_{123}$ is the arrangement $\Ac(8,1)$ for which we know that with respect to the basis
$B^K_{123} = \{\alpha_1,\alpha_2,\alpha_3\} \subseteq B^K$ we have $(\alpha_1 + 2\alpha_2+\alpha_3)^\perp \in \Ac^K_{123}$ and in particular
$(\alpha_1 + 2\alpha_2+\alpha_3)^\perp = (\sigma^K_2(\alpha_1)+\sigma^K_2(\alpha_3))^\perp \in W^{(K_2)_1}, W^{(K_2)_3}$.
So $c^{K_2}_{13} = c^{K_2}_{31} = -1$ and $B^{K_2}$ is locally crystallographic with Cartan matrix $C^{K_2}$ of type $D'$.
Note that from this case the statement also follows for $\Gamma(K)=\Gamma_4^3$ and $i=1,2$ by symmetry.

Now let $i=3$.
If the localizations $\Ac^K_{123}$ and $\Ac^K_{234}$ are both isomorphic to $\Ac(6,1)$
or if $\Ac^K_{234}$ is the arrangement $\Ac(8,1)$ then by symmetry we already handled these cases.
So suppose that $\Ac^K_{123}$ is the arrangement $\Ac(8,1)$.
Then by Lemma \ref{diagram graph change} and the lemmas from Subsection \ref{ssec:simplArr} and Section \ref{sec:Coxeter graphs} we have
$\Gamma(K_3) = \Gamma_4^2$ of Figure \ref{fig:Coxeter graphs rk4} and we also obtain all $c^{K_3}_{ij}$ except $c^{K_3}_{21}$.
But with respect to the basis $B^K_{123} = \{\alpha_1,\alpha_2,\alpha_3\} \subseteq B^K$ we have
$(2\alpha_1+ \alpha_2 +\alpha_3)^\perp = (2\sigma^K_3(\alpha_1) + \sigma^K_3(\alpha_2))^\perp \in \Ac^K_{123}$ so
$c^{K_3}_{21} = -2$ and $B^{K_3}$ is locally crystallographic with Cartan matrix $C^{K_3}$ of type $C$.
Note that from this case the statement also follows for $\Gamma(K)=\Gamma_4^2$ and $i=3$ by symmetry.
%

The other remaining cases, i.e.\
\begin{itemize}
\item $\Gamma(K) = \Gamma_4^2$ and $i\in \{1,2,4\}$,
\item $\Gamma(K) = \Gamma_4^3$ and $i\in \{3,4\}$,
\item $\Gamma(K) = \Gamma_4^4$ and $i \in \{1,2,3,4\}$,
\end{itemize}
can be handled completely analogously.
\end{proof}

Proposition \ref{Cartan matrices rk4} immediately tells us that an \isss $4$-ar\-rangement defines a Weyl groupoid and thus a crystallographic arrangement (cf.\ \cite[Thm.\ 1.1]{MR2820159}). Since we did not introduce the notion of a Weyl groupoid, we repeat the argument without this terminology:

\begin{propo}\label{rk4 crys rootsystem}
Let $\Ac$ be an \isss $4$-ar\-range\-ment, and fix a chamber $K^0 \in \Kc(\Ac)$.
Then there exists a basis $B^{K^0}$ for $K^0$ such that
$$
R := \bigcup_{G \in \Gc(K^0,\Ac)} B_G
$$
is a crystallographic root system for $\Ac$.
\end{propo}
\begin{proof}
By Proposition \ref{Cartan matrices rk4}(1) for $K^0$ there exists a locally crystallographic basis $B^{K^0}$ with
Cartan matrix of type $A,C,D'$ or $D$. Such a basis will have the desired property and from now on we fix it.

First we show that for $K \in \Kc(\Ac)$ the basis $B^K_G$
does not depend on the chosen $G \in \Gc(K^0,\Ac)$ with $e(G) = K$.
Let $G,\tilde{G} \in \Gc(K^0,\Ac)$ with $e(G) = e(\tilde{G}) = K$,
say
$$
G = (K^0,K^1,\ldots,K^{n-1},K^n=K),
$$
and
$$
\tilde{G} = (K^0,\tilde{K}^1,\ldots,\tilde{K}^{m-1},\tilde{K}^m=K).
$$
Then
$$
B_G = (\sigma_{\mu_{n-1}}^{K^{n-1}} \circ \ldots \circ \sigma_{\mu_{0}}^{K^{0}})(B^{K^0}),
$$
where the linear map $\sigma_{\mu_{n-1}}^{K^{n-1}} \circ \ldots \circ \sigma_{\mu_{0}}^{K^{0}}$ is represented with respect to $B^{K^0}$ by a product of reflection matrices
$$
S_{\mu_{n-1}}^{K^{n-1}} \cdots S_{\mu_{0}}^{K^{0}} =: S.
$$
By Proposition \ref{Cartan matrices rk4}(2) and an easy induction over the length $n$ of $G$
all reflection matrices $S_{\mu_i}^{K^i}$ are integral matrices with determinant $-1$.
Hence the product $S$ has only entries in $\ZZ$ and has determinant $\pm1$.
Similarly for $\tilde{G}$ we have
$$
B_{\tilde{G}} = (\sigma_{\tilde{\mu}_{m-1}}^{\tilde{K}^{m-1}} \circ \ldots \circ \sigma_{\tilde{\mu}_{0}}^{K^{0}})(B^{K^0}),
$$
where the linear map is represented with respect to $B^{K^0}$ by a product of integral reflection matrices
$$
S_{\tilde{\mu}_{m-1}}^{\tilde{K}^{n-1}} \cdots S_{\tilde{\mu}_{0}}^{K^{0}} =: \tilde{S},
$$
and $\tilde{S}$ also has only entries in $\ZZ$ and determinant equal to $\pm1$.
Now $S \tilde{S}^{-1} = \diag(\lambda_1,\ldots,\lambda_4)P$ where $P$ is a permutation matrix and $\lambda_i \in \RR_{>0}$
because $\{\alpha^\perp \mid \alpha \in B_G\} = \{\beta^\perp \mid \beta \in B_{\tilde{G}} \}$.
But $S \tilde{S}^{-1}$ is an integer matrix entries and has determinant $\pm 1$. Hence $\diag(\lambda_1,\ldots,\lambda_4)$ has
determinant $1$ so $\lambda_1 = \ldots = \lambda_4 = 1$ and thus $B_G = B_{\tilde{G}}$
(up to a permutation 
of the basis elements).

From the above consideration we obtain
$$
B_G \subseteq \sum_{\alpha \in B_{G'}} \ZZ\alpha,
$$
for $G,G' \in \Gc(K^0,\Ac)$.
Hence for $R$ we have
$$
R \subseteq \sum_{\alpha \in B^K_R} \ZZ\alpha,
$$
for all $K \in \Kc(\Ac)$ since $B^K_R = B_G$ for some $G \in \Gc(K^0,\Ac)$ with $e(G) = K$ and each $\beta \in R$ is contained in some
$B_{G'}$, $G' \in \Gc(K^0,\Ac)$.

It remains to show that $R$ is reduced, i.e.\ that for $\beta \in R$ we have $R \cap \RR\beta = \{\pm\beta\}$.
So suppose that $\beta \in R$ and $\lambda\beta \in R$ for some $\lambda \in \RR\setminus \{0\}$.
Then there are $G,G' \in \Gc(K^0,\Ac)$ such that $\beta \in B_G$ and $\lambda\beta \in B_{G'}$.
But as above $\lambda\beta \in \ZZ\beta$ and $\beta \in \ZZ(\lambda\beta)$. Hence $\lambda \in \{\pm1\}$.
\end{proof}

\begin{theor}\label{classification rk4}
Let $\Ac$ be an \isss $4$-ar\-rangement. Then $\Ac$ is isomorphic to exactly one of the reflection arrangements
$\Ac(A_4)$, $\Ac(C_4)$, or $\Ac_{4}^3 = \Ac(C_4) \setminus  \{ \{x_1=0\} \}$.
In particular $\Ac$ is crystallographic.
\end{theor}
\begin{proof}
By Proposition \ref{rk4 crys rootsystem} there exists a crystallographic root system for $\Ac$,
 so the arrangement $\Ac$ is crystallographic.
By Theorem \ref{classification ss cryst rk geq 4}
the only irreducible crystallographic $4$-arrangements which are supersolvable are the three arrangements
$\Ac(A_4)$, $\Ac(C_4)$, and $\Ac_{4}^3 = \Ac(C_4) \setminus  \{ \{x_1=0\} \}$.
\end{proof}

\begin{corol}\label{no type D}
Let $\Ac$ be an \isss $4$-arrange\-ment and  $K \in \Kc(\Ac)$.
Then $\Gamma(K)$ is not the Coxeter graph $\Gamma_4^4$ of Figure \ref{fig:Coxeter graphs rk4}.
\end{corol}

\section{The rank $\geq 5$ case}

\begin{lemma}\label{no big circles rk ell}
Let $\Ac$ be an irreducible simplicial supersolvable $\ell$-ar\-range\-ment and let $K \in \Kc(\Ac)$ be a chamber.
Then $\Gamma(K)$ has no cycles with more than $3$ vertices.
\end{lemma}

\begin{proof}
Suppose there is a chamber $K \in \Kc(\Ac)$ such that $\Gamma(K)$ has a cycle with more than three vertices.
Then we localize at the intersection of the walls corresponding to these vertices and use Lemma \ref{loc subgraph} and
Lemma \ref{res subgraph} (possibly several times) to arrive at a $4$-arrangement which is irreducible by Lemma \ref{res irred},
simplicial and supersolvable by Lemma \ref{SSS loc res}, and contains a chamber $K'$ such that the Coxeter graph $\Gamma(K')$
contains a subgraph  of the form displayed in Figure \ref{fig:imp_subgraph2}.
This is a contradiction to Lemma \ref{no big circles}.
\end{proof}

\begin{lemma}\label{imp subgraph middle triangle}
Let $\Ac$ be an \isss $\ell$-ar\-range\-ment, $\ell\geq5$ and let $K \in \Kc(\Ac)$ be a chamber.
Then the Coxeter graph $\Gamma(K)$ does not contain a subgraph of the form shown in Figure \ref{fig:imp_subgraph_triangle}.
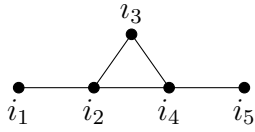
\begin{figure}
\begin{tikzpicture}
\filldraw [black]
	(0,0) circle [radius=2pt]
	(1,0) circle [radius=2pt]
	(1.5,0.71) circle [radius=2pt]
	(2,0) circle [radius=2pt]
	(3,0) circle [radius=2pt];
\draw
	(0,0) -- (1,0) -- (1.5,0.71) -- (2,0) -- (1,0)
	(2,0) -- (3,0);

\node[below] at (0,0)  {$i_1$};
\node[below] at (1,0)  {$i_2$};
\node[above] at (1.5,0.71)  {$i_3$};
\node[below] at (2,0)  {$i_4$};
\node[below] at (3,0)  {$i_5$};

\end{tikzpicture}\hspace{1cm}
\caption{Forbidden subgraph}\label{fig:imp_subgraph_triangle}
\end{figure}
\end{lemma}
\begin{proof}
Let $B^K = \{\alpha_1,\ldots,\alpha_\ell\}$, and suppose that $\Gamma(K)$ has a subgraph of this form containing the vertices
$\{i_1,\ldots,i_5\} \subseteq \{1,\ldots,\ell\}$.
By Lemma \ref{loc subgraph} and Lemma \ref{res subgraph} localizing $\Ac^K_{i_1\cdots i_5}$
and restricting to $H = \sigma_{i_2}^K(\alpha_{i_3})^\perp$ gives the \isss $4$-arrange\-ment $(\Ac^K_{i_1\cdots i_5})^H$
which contains a chamber with a Coxeter graph not included in the list from Proposition \ref{Coxeter graphs rk4}.
Hence $\Gamma(K)$ could not have such a subgraph in the first place.
\end{proof}

\begin{propo}\label{Coxeter graphs rk geq4}
Let $\Ac$ be an \isss $\ell$-ar\-rangement, $\ell\geq4$ and let $K \in \Kc(\Ac)$ be a chamber.
Then $\Gamma(K)$ is one of the Coxeter graphs of Figure \ref{fig:Coxeter graphs rk geq4}.
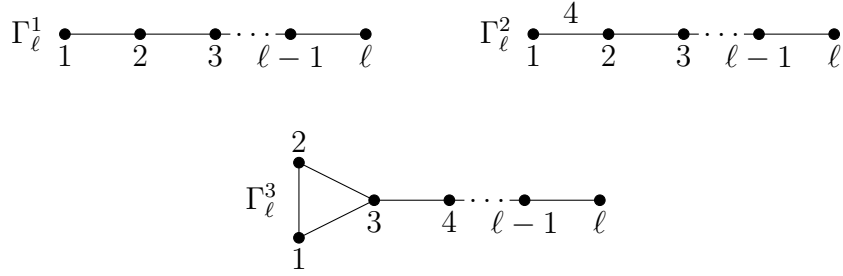
\begin{figure}
\begin{tikzpicture}
\node at (-0.5,0)  {$\Gamma_\ell^1$};
\filldraw [black]
	(0,0) circle [radius=2pt]
	(1,0) circle [radius=2pt]
	(2,0) circle [radius=2pt]
	(3,0) circle [radius=2pt]
	(4,0) circle [radius=2pt];
\draw
	(0,0) -- (1,0) -- (2,0) -- (2.2,0)
	(2.8,0) -- (3,0) -- (4,0);
\node at (2.5,0)  {$\ldots$};

\node[below] at (0,0)  {$1$};
\node[below] at (1,0)  {$2$};
\node[below] at (2,0)  {$3$};
\node[below] at (3,0)  {$\ell-1$};
\node[below] at (4,0)  {$\ell$};
\end{tikzpicture}\hspace{1cm}
\begin{tikzpicture}
\node at (-0.5,0)  {$\Gamma_\ell^2$};
\filldraw [black]
	(0,0) circle [radius=2pt]
	(1,0) circle [radius=2pt]
	(2,0) circle [radius=2pt]
	(3,0) circle [radius=2pt]
	(4,0) circle [radius=2pt];
\draw
	(0,0) -- (1,0) -- (2,0) -- (2.2,0)
	(2.8,0) -- (3,0) -- (4,0);
\node at (2.5,0)  {$\ldots$};
\node[above] at (0.5,0)  {$4$};

\node[below] at (0,0)  {$1$};
\node[below] at (1,0)  {$2$};
\node[below] at (2,0)  {$3$};
\node[below] at (3,0)  {$\ell-1$};
\node[below] at (4,0)  {$\ell$};
\end{tikzpicture}

\vspace{0.5cm}

\begin{tikzpicture}
\node at (-0.5,0)  {$\Gamma_\ell^3$};
\filldraw [black]
	(0,0.5) circle [radius=2pt]
	(0,-0.5) circle [radius=2pt]
	(1,0) circle [radius=2pt]
	(2,0) circle [radius=2pt]
	(3,0) circle [radius=2pt]
	(4,0) circle [radius=2pt];
\draw
	(0,0.5) -- (0,-0.5) -- (1,0) -- (0,0.5)
	(1,0) -- (2,0) -- (2.2,0)
	(2.8,0) -- (3,0) -- (4,0);
\node at (2.5,0)  {$\ldots$};

\node[below] at (0,-0.5)  {$1$};
\node[above] at (0,0.5)  {$2$};
\node[below] at (1,0)  {$3$};
\node[below] at (2,0)  {$4$};

\node[below] at (3,0)  {$\ell-1$};
\node[below] at (4,0)  {$\ell$};
\end{tikzpicture}
\caption{Possible Coxeter graphs for an \isss $\ell$-arrangement ($\ell \geq 4$).}\label{fig:Coxeter graphs rk geq4}
\end{figure}
\end{propo}

\begin{proof}
By Lemma \ref{imp subgraph middle triangle} the Coxeter graph $\Gamma(K)$ cannot have a triangle  somewhere in the middle.

The statement then follows by induction on $\ell$, Lemma \ref{loc subgraph}, Proposition \ref{Coxeter graphs rk4},
Corollary \ref{no type D}, and Lemma \ref{no big circles rk ell}.
\end{proof}

\begin{propo}\label{Cartan matrices rk geq4}
Let $\Ac$ be an \isss $\ell$-ar\-rangement, $\ell \geq4$ and $K \in \Kc(\Ac)$.
\begin{enumerate}
\item There exists a locally crystallographic basis $B^K$ for $K$ such that the Cartan matrix $C^K$
	is of type $A,C$ or $D'$.
\item If $B^K = \{\alpha_1,\ldots,\alpha_4\}$ is a locally crystallographic basis for $K$ with $C^K$ of type $A,C$ or $D'$, then for $1 \leq i \leq \ell$ the basis
$B^{K_i} = \sigma_i^K(B^K) = \{ \alpha_j + c^K_{ij}\alpha_i \mid 1 \leq j \leq \ell\}$
is a locally crystallographic basis with Cartan matrix $C^{K_i}$ of type $A,C$ or $D'$.
\end{enumerate}
\end{propo}
\begin{proof}
We argue by induction on $\ell \geq 4$.
For $\ell=4$ this is Proposition \ref{Cartan matrices rk4}.
Let $\ell \geq 5$ and assume both statements are true for $\ell-1$.
By Proposition \ref{Coxeter graphs rk geq4} the Coxeter graph $\Gamma(K)$
is one of the graphs of Figure \ref{fig:Coxeter graphs rk geq4}. Let $W^K = \{ H_1,\ldots,H_\ell\}$ where the numbering of
the walls corresponds to the numbering of the vertices in $\Gamma_\ell^1, \Gamma_\ell^2, \Gamma_\ell^3$ respectively.

By the induction hypothesis for the localization $\Ac^K_{12\cdots(\ell-1)}$ there are
$\{\alpha_1,\ldots,\alpha_{\ell-1}\} \subseteq (\RR^\ell)^*$ which form a locally crystallographic basis for the corresponding chamber
in $\Ac^K_{12\cdots(\ell-1)}$. Furthermore
$\alpha_i^\perp = H_i$ for $1 \leq i \leq \ell-1$, there are $c^K_{ij} \in \ZZ, 1 \leq i,j\leq \ell-1$
such that $(\alpha_j - c^K_{ij}\alpha_i)^\perp \in W^{K_i}$,
and the matrix $C'^K = (c^K_{ij})_{1 \leq i,j \leq \ell-1}$ is a Cartan matrix of type $A,C$, or $D'$.
But in $\Gamma(K)$ the vertex $\ell$ is only connected to $\ell-1$ by an edge with label $3$.
Hence there is an $\alpha_\ell \in (\RR^\ell)^*$ such that $\alpha_\ell^\perp = H_\ell$, $K \subseteq \alpha_\ell^+$,
$(\alpha_{\ell-1}+\alpha_\ell)^\perp \in W^{K_{\ell-1}},W^{K_\ell}$.
This is to say for $B^K := \{\alpha_1,\ldots,\alpha_\ell\}$ we have $c^K_{(\ell-1)\ell} = c^K_{\ell(\ell-1)} = -1$,
$c^K_{\ell j} = c^K_{j\ell} =0$ for $j \notin \{\ell-1,\ell\}$, the other $c^K_{ij}$ are
given by the localization $\Ac^K_{12\cdots(\ell-1)}$, and hence $B^K$
is a locally crystallographic basis for $K$ with Cartan matrix of type $A,C$, or $D'$
if $\Gamma(K)$ is $\Gamma_\ell^1, \Gamma_\ell^2$, or $\Gamma_\ell^3$ respectively.

Now let $B^K$ be a locally crystallographic basis with Cartan matrix of type $A,C$, or $D'$
and $B^{K_i}$ the induced basis for $K_i$.

If $i=\ell$ then in each case
$\Gamma(K_i)=\Gamma(K)$ by Lemma \ref{graph adj chamber}, Lemma \ref{label not connected}
and Proposition \ref{Coxeter graphs rk geq4} where the vertex $k$ in $\Gamma(K_i)$ corresponds
to the root $\sigma_\ell^K(\alpha_k)$.
In all graphs $\Gamma_\ell^k$ the vertex $\ell$ is not connected with the vertex $j$ for $j \leq \ell-2$, and
$m^K(i,j) =3$ for all $i,j \in \{1,\ldots,\ell\}$ except possibly for $\{i,j\} = \{1,2\}$.
So by Lemma \ref{cij=cij'}, and Lemma \ref{cij=0} the induced basis $B^{K_\ell}$
is locally crystallographic with Cartan matrix $C^{K_i} = C^{K}$ of type $A,C$, or $D'$.

For $i \in \{1, \ldots, \ell-1\}$ we have $\Ac^K_{1\cdots(\ell-1)} = \Ac^{K_i}_{1\cdots(\ell-1)}$.
So at least $C'^{K_i} =(c^{K_i})_{1 \leq i,j \leq \ell-1}$ is a Cartan matrix of type $A,C$, or $D'$ by
the induction hypothesis.
If $C'^{K_i}$ is of type $C$, or $D$ then by Proposition \ref{Coxeter graphs rk geq4}
the Coxeter graph $\Gamma(K_i)$ is $\Gamma_\ell^2$, or $\Gamma_\ell^3$ respectively
with the numbering of the vertices corresponding to $B^{K_i} = \{\sigma^K_i(\alpha_1),\ldots,\sigma^K_i(\alpha_\ell) \}$.
If $C'^{K_i}$ is of type $A$ we may also assume that $\Gamma(K_i)$ is the Coxeter graph $\Gamma_\ell^1$ since otherwise
we can renumber the bases $B^K$ and $B^{K_i}$ respectively such that $C'^{K_i}$ is of type $C$, or $D'$
and we actually are in one of the above cases.
We observe next that in $\Gamma(K_i)$ the vertex $\ell$
is not connected with the vertex $j$ for $j \leq \ell-2$.
So $c^{K_i}_{\ell j} = c^{K_i}_{j\ell} = 0$ for $1 \leq j \leq \ell-2$.

If $i \in \{1,\ldots,\ell-2\}$ we have $c^K_{i\ell}=0$ and then by Lemma \ref{cij=0} we get
$c^{K_i}_{\ell (\ell-1)} = c^K_{\ell (\ell-1)}$. But $m^{K_i}(\ell-1,\ell)=3$ and by Lemma \ref{cij m=3}
for the last remaining coefficient we obtain $c^{K_i}_{(\ell-1)\ell} =-1$ and $B^{K_i}$ is a locally crystallographic
basis with Cartan matrix of type $A,C$, or $D'$.

Finally for $i={\ell-1}$ by Lemma \ref{cij=cij'} we also have $c^{K_{\ell-1}}_{(\ell-1)\ell} = c^K_{(\ell-1)\ell} = -1$.
Again since $m^{K_{\ell-1}}(\ell-1,\ell)=3$, by Lemma \ref{cij m=3}
for the remaining coefficient we get $c^{K_{\ell-1}}_{(\ell-1)\ell} =-1$ and $B^{K_{\ell-1}}$ is a locally crystallographic
basis with Cartan matrix $C^{K_{\ell-1}}$ of type $A,C$, or $D'$.
This finishes the proof.
\end{proof}

\begin{propo}\label{rk geq4 crys rootsystem}
Let $\Ac$ be an \isss $\ell$-ar\-range\-ment, $\ell \geq 4$, and fix a chamber $K^0 \in \Kc(\Ac)$.
Then there exists a basis $B^{K^0}$ for $K^0$ such that
$$
R := \bigcup_{G \in \Gc(K^0,\Ac)} B_G
$$
is a crystallographic root system for $\Ac$.
\end{propo}

\begin{proof}
This is exactly the same as the proof of Proposition \ref{rk4 crys rootsystem}
using Proposition \ref{Cartan matrices rk geq4} instead of Proposition \ref{Cartan matrices rk4}.
\end{proof}

\begin{theor}
Let $\Ac$ be an irreducible simplicial supersolvable $\ell$-ar\-rangement, $\ell \geq 4$.
Then $\Ac$ is isomorphic to exactly one of the reflection arrangements
$\Ac(A_\ell)$, $\Ac(C_\ell)$, or $\Ac_{\ell}^{\ell-1} = \Ac(C_\ell) \setminus \{ \{x_1=0\} \}$.
In particular $\Ac$ is crystallographic.
\end{theor}
\begin{proof}
By Proposition \ref{rk geq4 crys rootsystem} there exists a crystallographic root system for $\Ac$,
so the arrangement $\Ac$ is crystallographic.
By Theorem \ref{classification ss cryst rk geq 4}
the only irreducible crystallographic $\ell$-arrangements, $\ell \geq 4$ which are supersolvable are the arrangements
$\Ac(A_\ell)$, $\Ac(C_\ell)$, and $\Ac_{\ell}^{\ell-1} = \Ac(C_\ell) \setminus  \{ \{x_1=0\} \}$.
\end{proof}


\renewcommand{\refname}{References}

\providecommand{\bysame}{\leavevmode\hbox to3em{\hrulefill}\thinspace}
\providecommand{\MR}{\relax\ifhmode\unskip\space\fi MR }
\providecommand{\MRhref}[2]{%
  \href{http://www.ams.org/mathscinet-getitem?mr=#1}{#2}
}
\providecommand{\href}[2]{#2}

\end{document}